\documentclass[11pt]{article}
\usepackage[margin=0.95in]{geometry}
\usepackage{graphicx,psfrag}
\usepackage{amsmath,amsthm,amsfonts,amssymb,epsfig,verbatim, fge, bm}
\usepackage{xcolor}
\usepackage[pagebackref=true, colorlinks, citecolor=blue, linkcolor=blue]{hyperref}

\usepackage{upref}

\usepackage[normalem]{ulem}
\usepackage{bbm}
\usepackage{circledsteps}
\usepackage{graphicx}
\usepackage{tikz}
\usetikzlibrary{arrows.meta,decorations.pathreplacing,positioning, calc, decorations.pathmorphing}
\usepackage{enumitem}
\usepackage{mathrsfs}

\usepackage{enumitem}
\setlist{
  topsep=4pt,
  itemsep=2pt,
  parsep=0pt
}

\newtheorem{thm}{Theorem}[section]
\newtheorem{lem}[thm]{Lemma}
\newtheorem{prop}[thm]{Proposition}

\newtheorem{df}[thm]{Definition}

\newtheorem*{rem}{Remark}

\numberwithin{equation}{section}

\newcommand{\E}{\mathbf{E}}
\newcommand{\G}{\mathbb{G}}
\newcommand{\prob}{\mathbf{P}}

\newcommand{\C}{\mathbb{C}}

\newcommand{\W}{\mathsf{W}}
\newcommand{\sZ}{\mathsf{Z}}
\DeclareMathOperator{\Var}{Var}
\DeclareMathOperator{\Cov}{Cov}

\newcommand{\M}{\mathsf{M}}
\newcommand{\B}{\mathsf{B}}
\renewcommand{\setminus}{\mathbin{\fgebackslash}}
\newcommand{\bB}{\mathbb{B}}
\newcommand{\bs}{\bm{\sigma}}
\newcommand{\bp}{\bm{\pi}}
\newcommand{\sw}{\mathfrak{S}}
\newcommand{\bG}{\mathbf{G}}
\newcommand{\cE}{\mathcal{E}}
\newcommand{\I}{\mathsf{I}}
\newcommand{\ind}{\mathbf{1}}

\newcommand{\cF}{\mathcal{F}}
\newcommand{\Ggood}[2]{{#1}^{\! [#2]}}

\begin{document}
    \title{Central Limit Theorem of Maximum Weight Matching on Random Graphs with Prescribed Degrees}
    \author{Shichen Jing \thanks{University of Minnesota. Email: jing0101@umn.edu.} 
    \and  Wai-Kit Lam\thanks{National Taiwan University. Email: waikitlam@ntu.edu.tw. The research of W.-K. L. is supported by the National Science and Technology Council in Taiwan grant 113-2115-M-002-009-MY3 and the NCTS Young Theoretical Scientist Award.} 
		\and Arnab Sen\thanks{University of Minnesota. Email: arnab@umn.edu.  The research of A.-S. is partly supported by Simons Foundation MP-TSM-00002716.} }
	\date{}
	\maketitle

    \begin{abstract}
    We prove an annealed central limit theorem for the weight of the maximum
weight matching on uniformly random simple graphs with prescribed,
uniformly bounded degrees and i.i.d.\ exponential edge weights. In particular, the result applies to random $d$-regular
graphs for every fixed $d \ge 2$.
The proof separates the fluctuations arising from the edge weights from
those arising from the graph. The correlation decay estimate of Lam and
Sen \cite{lamsen2026} yields Gaussian fluctuations for the former.  The main difficulty is to analyze the fluctuations of the conditional mean of the optimal weight
given the graph. To address this, we prove a stronger perturbative correlation decay estimate that, together with a variance bound, reduces the problem to the central limit theorem of Barbour and R\"ollin \cite{BarbourRollin} for local statistics of the configuration model.

    \end{abstract}
	

\section{Introduction}

Let $G=(V(G),E(G))$ be a finite graph. A matching in $G$ is a
collection of edges, no two of which share a vertex. Assign i.i.d.\
continuous weights $(w_e)_{e\in E(G)}$ to the edges, and define the
weight of a matching $M$ by $\sum_{e\in M}w_e$. Let $\M_G$ denote
the almost surely unique maximum weight matching (MWM), and write
$\W_G:=\sum_{e\in\M_G}w_e$ for its weight.

First-order asymptotics for this optimization problem have been
established for several locally convergent graph models. In these
settings, the limiting weight per vertex, i.e., $|V(G)|^{-1} \W_G$, is determined by
the local weak limit and the edge-weight distribution. Examples include uniform labeled trees with i.i.d.\ nonnegative continuous edge weights of finite mean \cite{AldousSteele};
random $d$-regular graphs, $d\ge2$, and sparse Erd\H{o}s--R\'enyi
graphs $G(n,\lambda/n)$, $\lambda>0$, with exponential edge weights
\cite{Gamarnik}; and configuration models converging locally to
unimodular Galton--Watson trees, under suitable degree assumptions,
with i.i.d.\ continuous edge weights of finite mean
\cite{enriquez2025optimal}. For exponential weights, Lam and Sen \cite{lamsen2026} established a law of large numbers for
bounded-degree graph sequences converging locally to arbitrary
unimodular random trees.

The next natural question concerns fluctuations of the optimal weight
about its mean. Cao \cite{Cao} proved an annealed CLT for MWM on sparse
Erd\H{o}s--R\'enyi graphs with exponential weights. A corresponding CLT
was established in \cite{SturmWemheuer} for a class of sparse
inhomogeneous random graphs. Beyond the locally tree-like setting,
a CLT for MWM on $d$-dimensional tori was obtained in
\cite{KrishnanRay} for a class of absolutely continuous weight laws
including Gaussian weights.

Another classical setting is provided by mean-field minimum-weight
perfect matching problems on complete and complete bipartite graphs
with independent edge costs. Several predictions of M\'ezard and Parisi
based on replica symmetry \cite{MezardParisi1985,MezardParisi1986}
were later established rigorously by Aldous and W\"astlund
\cite{Aldous2001,Wastlund2012}. A central limit theorem was recently
proved for the bipartite random assignment problem with i.i.d.\
uniform costs \cite{Mordant2026}. To the best of our knowledge, a
corresponding CLT for the nonbipartite minimum-weight perfect matching
problem remains open.

In this paper, we establish an annealed CLT for uniformly random simple
graphs with prescribed bounded degrees, including random regular graphs.
For each $n\ge2$, let
$\mathbf d^{(n)}=(d_1^{(n)},\ldots,d_n^{(n)})$ be a graphical degree
sequence satisfying
\begin{equation} \label{eq: degree_seq_cond}
\sum_{i=1}^n d_i^{(n)}\ge n, \qquad
\max_{1\le i\le n}d_i^{(n)}\le D,
\end{equation}
where $D\ge2$ is fixed. Let $\G_n$ be chosen uniformly from the simple
graphs on $\{1,\ldots,n\}$ with degree sequence $\mathbf d^{(n)}$.
Independently of the graph, assign i.i.d.\ $\mathrm{Exp}(1)$ weights
to its edges. Here and below, $\E$ and $\Var$ average over both the
graph and the edge weights.
\begin{thm}
\label{thm: annealed_CLT}
Let $\G_n$ be a uniformly random simple graph with degree sequence
$\mathbf d^{(n)}$ satisfying \eqref{eq: degree_seq_cond}, equipped with
i.i.d.\ $\mathrm{Exp}(1)$ edge weights independent of the graph. Then
\[\frac{\W_{\G_n}-\E\W_{\G_n}}
{\sqrt{\Var(\W_{\G_n})}}
\stackrel{d}{\rightarrow}\mathrm N(0,1) \qquad\text{as }n\to\infty.
\]
\end{thm}
No convergence of the empirical degree distributions is assumed. In
particular, the theorem applies to random $d$-regular graphs for every
fixed $d\ge2$. The proof relies essentially on the memoryless property
of the exponential distribution. Extending the result to a broader
class of continuous edge-weight distributions remains an open problem.

We also briefly discuss the unweighted matching problem, which
maximizes the number of matched edges. Karp and Sipser
\cite{KarpSipser} determined the limiting matching density for sparse
Erd\H{o}s--R\'enyi graphs; see also \cite{AronsonFriezePittel}. For bounded-degree locally convergent graph sequences, the normalized
matching number converges \cite{ElekLippner}. This result was extended
beyond bounded degree, with the limit characterized through the local
weak limit \cite{BordenaveLelargeSalez}. A Gaussian CLT for the matching number is known
for sparse Erd\H{o}s--R\'enyi graphs at every fixed mean degree
\cite{Kreacic2017,GlasgowKwanSahSawhney}. For subcritical bounded-degree configuration models, a general CLT also applies to the matching number,  provided its variance grows
linearly  \cite{AthreyaYogeshwaran}. A related problem, maximizing weight among maximum-cardinality matchings, was studied in \cite{enriquez2026optimal}.

A positive-temperature analogue of MWM is the disordered monomer--dimer
model. Gaussian fluctuations of its free energy are known on cylinder
graphs \cite{DeyKrishnan} and, under suitable moment assumptions, on
general bounded-degree graphs \cite{lamsenpositivetemp}.

More broadly, fluctuation results for optimization problems on random
regular and configuration-model graphs remain limited. For example, interpolation
methods yield laws of large numbers for maximum cut, maximum bisection,
and the independence number on random regular graphs
\cite{BayatiGamarnikTetali,Salez,Huang}, but do not by themselves give
the corresponding fluctuation laws. The difficulty created by prescribed
degrees is also reflected in CLTs for the giant-component size: in the
fixed supercritical regime, the configuration-model result
\cite{BarbourRollin} appeared decades after the Erd\H{o}s--R\'enyi
result \cite{Pittel1990}.

The guiding principle behind our proof is spatial correlation decay,
which allows the membership of an edge in the MWM to be approximated
using only the weights in a finite neighborhood. For exponential
weights, \cite{Gamarnik} established correlation decay through recursive
distributional equations on regular and Poisson Galton--Watson trees.
Exponential correlation decay was subsequently proved in
\cite{lamsen2026} uniformly over bounded-degree tree neighborhoods,
without assuming regularity. That paper also conjectured that
correlation decay holds uniformly over bounded-degree graphs with
i.i.d.\ nonnegative continuous edge weights. Beyond exponential weights,
qualitative correlation decay was established in \cite{KangLiu2026}
for a class of continuous edge-weight laws on unimodular
Galton--Watson trees whose degree distribution has a finite second
moment. Using the representation
\[ \W_G=\sum_{e\in E(G)}w_e\mathbf{1}_{\{e\in\M_G\}}, \]
correlation decay allows us to approximate $\W_G$ by a sum of local,
weakly dependent contributions.

For deterministic graphs, combining correlation decay with suitable
moment and variance bounds is now a standard route to a CLT via
Chatterjee's normal approximation method \cite{Chatterjee14}.
In our setting, this yields a CLT for deterministic bounded-degree graphs with
linearly many edges and diverging girth. The same normal approximation
method underlies Cao's result \cite{Cao} for sparse
Erd\H{o}s--R\'enyi graphs. Indeed, the optimal weight can be realized
on the deterministic complete graph $K_n$ using independent weights
$B_eX_e$, where
$B_e\sim\operatorname{Bernoulli}(\lambda/n)$,
$X_e\sim\operatorname{Exp}(1)$, and all these variables are independent.

In the configuration model, however, the prescribed degrees introduce
dependence between edges through the random pairing of half-edges.
Conditioning on the graph allows us to apply the preceding method,
but the resulting quenched CLT does not by itself yield an annealed
CLT. One must also analyze the fluctuations of
$\E_w\W_{\G_n}$ and the concentration of $\Var_w(\W_{\G_n})$ as the
graph varies, where $\E_w$ and $\Var_w$ denote expectation and variance
over the edge weights alone.

The main new ingredient is a perturbative correlation decay estimate
showing that the quenched mean and its local approximation respond
nearly identically to the deletion of a single edge. Combined with an
Efron--Stein-type variance bound, this estimate allows us to approximate
the fluctuations of the quenched mean by those of a local graph
statistic and apply the CLT of Barbour and R\"ollin
\cite{BarbourRollin} for the configuration model, as explained next.

\subsection{Proof Outline}  \label{subsec:outline_intro}

To prove Theorem~\ref{thm: annealed_CLT}, we decompose the fluctuations of
$\W_{\G_n}$ into those arising from the edge weights and those arising from
the randomness of the graph:
\begin{equation}\label{eq:main_decomp}
\frac{\W_{\G_n}-\E\W_{\G_n}}{\sqrt{\Var(\W_{\G_n})}}
= \sqrt{\frac{\Var_w(\W_{\G_n})}{\Var(\W_{\G_n})} } \frac{\W_{\G_n} -\E_w\W_{\G_n}}{\sqrt{\Var_w(\W_{\G_n})}} +  \sqrt{\frac{\Var_{\G_n}(\E_w\W_{\G_n})}{\Var(\W_{\G_n})}}\frac{\E_w\W_{\G_n}-\E_{\G_n}\E_w\W_{\G_n}}
{\sqrt{\Var_{\G_n}(\E_w\W_{\G_n})}}
\end{equation}
We introduce some notations for the two terms in this decomposition. For any deterministic
simple graph $G$ satisfying $\Var_w(\W_G)>0$, define
\[ X(G) = \frac{\W_G-\E_w\W_G}{\sqrt{\Var_w(\W_G)}}.\]
For the
random graph $\G_n$, define the standardized quenched mean 
\[ Y_n^{\mathrm s}   = \frac{\E_w\W_{\G_n}-\E_{\G_n}\E_w\W_{\G_n}}
{\sqrt{\Var_{\G_n}(\E_w\W_{\G_n})}},\]
with the convention that $Y_n^{\mathrm s} = 0$ if $\Var_{\G_n}(\E_w\W_{\G_n})=0$. Here, the superscript $\mathrm s$ stands for the simple-graph model and distinguishes this quantity from the analogous standardized quenched mean for the configuration model, denoted later by $Y_n^{\mathrm c}$. Also, define
\begin{align*}
Q_n^{\mathrm s}:=\frac{\Var_w(\W_{\G_n})}{\E_{\G_n}\Var_w(\W_{\G_n})}, \ \ \ \alpha_n
:=\sqrt{\frac{\E_{\G_n}\Var_w(\W_{\G_n})}{\Var(\W_{\G_n})}}, \ \ \  
\beta_n := \sqrt{ \frac{\Var_{\G_n}(\E_w\W_{\G_n})} {\Var(\W_{\G_n})}}
\end{align*}
Thus, $Q_n^{\mathrm s}$ is the normalized quenched variance. 
By the law of total variance, $\alpha_n,\beta_n\in[0,1]$ and
$\alpha_n^2+\beta_n^2=1$.  With this notation,
\eqref{eq:main_decomp} becomes 
\begin{equation*}
\frac{\W_{\G_n}-\E\W_{\G_n}}{\sqrt{\Var(\W_{\G_n})}}
= \alpha_n\sqrt{Q_n^{\mathrm s}}\,X(\G_n) +
\beta_nY_n^{\mathrm s}.
\end{equation*}
To show the CLT for $\W_{\G_n}$, it suffices to verify the following three conditions.

\begin{enumerate}[
    label=\textbf{Condition \Roman*.},
    ref=Condition~\Roman*,
    align=left,
    widest=III,
    leftmargin=*,
    labelsep=0.75em,
    itemsep=1em]
\item \label{cond:quenched_clt}
As $n\to\infty,$
\[ \E_{\G_n}
d_{\mathrm{KS}}\left(\mathcal L\bigl(X(\G_n)\mid \G_n\bigr),\mathrm N(0,1)\right)=o(1).\]
\item  \label{cond:variance_conc} $Q_n^{\mathrm s} \stackrel{p}{\to} 1$ as $n \to \infty.$
 \item \label{cond:quenched_mean_clt} For every subsequence $(n_k)$ such that $\beta_{n_k}\to \beta>0$, one has  
 \begin{equation*}
   Y_{n_k}^{\mathrm s} \stackrel{d}{\to} \mathrm N(0,1)  \text{ as } k \to \infty.
 \end{equation*}
\end{enumerate}
It is easy to check that \ref{cond:quenched_clt}, \ref{cond:variance_conc}, and \ref{cond:quenched_mean_clt} imply the desired CLT for $\W_{\G_n},$ see Section~\ref{subsec:outline}.

\ref{cond:quenched_clt} concerns the fluctuations of $\W_{\G_n}$
around its quenched mean, with the graph held fixed.  For deterministic bounded-degree graphs with linearly many edges and diverging girth, the correlation decay
established in \cite{lamsen2026} for exponential edge weights, combined with
Chatterjee's normal approximation method \cite{Chatterjee14}, yields a uniform
quenched CLT; see Section~\ref{sec:chatterjee_CLT}. Removing a negligible number of edges
to eliminate short cycles then gives Condition~I for $\G_n$.
This reduction is carried out in Section~\ref{sec:verification}.

\ref{cond:variance_conc} states that the quenched variance $\Var_w(\W_{\G_n})$ concentrates around its mean. We first establish the corresponding concentration for
the configuration model, using correlation decay together with an
Efron--Stein-type variance bound for graph statistics
(Lemma~\ref{cor: conf_var_bound}), and then transfer it to the simple-graph model.
The details are given in Section~4.4.

The main technical difficulty is verifying
\ref{cond:quenched_mean_clt}, which concerns the fluctuations of the
quenched mean $\E_w\W_{\G_n}$ under the randomness of the graph.
Using Janson's switching technique \cite{Janson}, we first reduce the CLT for $\E_w\W_{\G_n}$ to the corresponding CLT for $\E_w\W_{\C_n}$ on the configuration model $\C_n$, whose distribution is more tractable.  The transfer  is stated in
Lemma~\ref{lem:G_to_C-CLT} and proved in
Section~\ref{subsec:G_to_C-CLT}.

After averaging over the edge weights, the quenched mean becomes a
statistic of the graph alone and can be written as a sum of vertex
contributions:
\begin{equation}\label{eq:graph_statistics}
  \E_w\W_{\C_n} = \sum_{v\in V(\C_n)} h(\C_n,v),
\quad h(G,v):=\frac12\sum_{e\ni v}\E_w\!\left[w_e\ind_{\{e\in\M_G\}}\right].  
\end{equation}
Our main input is the CLT of Barbour--R\"ollin
\cite{BarbourRollin} for sums of {\em local} graph statistics, stated in
Theorem~\ref{thm: BarbourRollin}. Since $h(\C_n,v)$ is not local, we
approximate it by replacing the full graph $\C_n$ with the radius-$R$
neighborhood $\bB_v^R(\C_n)$ of $v$, where
$1\ll R\ll\log n$ grows sufficiently slowly with $n$.

A naive term-by-term application of correlation decay gives only the bound
\[ \big| \E_w\W_{\C_n}  - \sum_{v\in V(\C_n)}  h(\bB_v^R(\C_n),v) \big|
= O_p\bigl(ne^{-cR}\bigr) \]
for some $c>0$. When $R=o(\log n)$, this bound is too large to yield an approximation
error of order $o_p(\sqrt n)$, the relevant fluctuation scale of
$\E_w\W_{\C_n}$.

Instead, we
prove the perturbative correlation decay estimate in
Theorem~\ref{thm:local_perturbative_bound}, which shows that the quenched
mean and its local approximation respond nearly identically to the deletion
of a single edge, uniformly over the bounded graphs. Combining this estimate with the
Efron--Stein-type variance bound in
Lemma~\ref{cor: conf_var_bound}, we control the approximation error
in variance. Schematically, the resulting estimate is
\[ \Var_{\C_n}\Big( \E_w\W_{\C_n} - \sum_{v\in V(\C_n)}  h(\bB_v^R(\C_n),v) \Big)=o(n),\]
see Proposition~\ref{eq: var_k_approx}. Consequently, after centering,
the local statistic approximates the quenched mean at the required
$\sqrt n$ scale. The Barbour--R\"ollin CLT can then be applied to the
local statistic, yielding the desired CLT for $\E_w\W_{\C_n}$. The
details are given in Section~\ref{sec: final_step}. Strictly speaking, the proof implements this argument on a suitably pruned version of
$\C_n$, see Definition~\ref{df:bad_edge}, for which the required neighborhoods are trees (and in particular, contain no self-loops and multi-edges). The pruning error is negligible on the $\sqrt n$ scale, and the edge-centered
statistic can be expressed as a sum of local vertex contributions.

The proof of Theorem~\ref{thm:local_perturbative_bound} is considerably
more involved than the argument for standard correlation decay and is
given in Section~\ref{sec: local_approximation}.

\subsection{Notations}

Let $G=(V(G),E(G))$ be a finite simple graph. For $A\subseteq V(G)$,
let $G\setminus A$ denote the induced subgraph of $G$ on the vertex set
$V(G)\setminus A$. For $F\subseteq E(G)$, let $G\setminus F:=(V(G),E(G)\setminus F).$
For a singleton, we write $G\setminus v$ and $G\setminus e$ in place of
$G\setminus\{v\}$ and $G\setminus\{e\}$, respectively.

For $u,v\in V(G)$, let $d_G(u,v)$ denote their graph distance. For $v\in V(G)$ and $e,f\in E(G)$, define
\[d_G(v,e):=\min_{x\in e}d_G(v,x), \qquad
d_G(e,f):=d_{L(G)}(e,f), \]
where $L(G)$ denotes the line graph of $G$.  When the
ambient graph is clear, we write $d(e,f)$ instead of $d_G(e,f)$.  Let $\mathcal{G}_{n, D}$ be the space of all simple graphs on $n$ vertices with maximum degree at most~$D$.

For $v\in V(G)$, $e\in E(G)$, and $r\ge0$, let $\bB_v^r(G)$ and $\bB_e^r(G)$ be the subgraphs of $G$ induced by the vertices at distance at most $r$ from $v$ and $e$, respectively, and let $\partial\bB_v^r(G)$ and $\partial\bB_e^r(G)$ be the corresponding sets of vertices at distance exactly $r$.

We write $\prob_w$ for probability with respect to the weights $(w_e)$
only, and $\prob_{\G_n}$ and $\prob_{\C_n}$ for probability with respect
to the random graph $\G_n$ and the configuration model $\C_n$,
respectively, where $\C_n$ is defined in
Section~\ref{subsec:config}. We use $\prob$ for probability with respect
to both the graph and the weights. Similarly, $\E_w$, $\E_{\G_n}$, and
$\E_{\C_n}$ denote the corresponding expectations, while $\E$ denotes
expectation with respect to all sources of randomness. We use the same subscript conventions for variance $\Var$ and covariance $\Cov$. For a random variable $X$, we denote its law by $\mathcal L(X)$. Let $d_{\mathrm{KS}}(\mu, \nu)$ and  $d_{\mathrm{W}}(\mu, \nu)$   be the Kolmogorov-Smirnov distance and the Wasserstein distance between the probability measures $\mu$ and $\nu$ on $\mathbb{R}$, respectively. 

Throughout the paper, the expression $f(n, \ell) = O(g(n, \ell))$ or $f(n, \ell) \leq O(g(n, \ell))$ means that there exists a constant $C > 0$ independent of $n$ and $\ell$ such that $f(n, \ell) \leq C g(n, \ell)$ for all $n$, $\ell$. If we write $O_\varepsilon$, then the constant $C$ could possibly depend on $\varepsilon$. $f(n) = o(g(n))$ means $\lim_{n \to \infty} \frac{f(n)}{g(n)} = 0$. Finally, $C$, $c$, $C_1$, $C_2$ etc.\ denote positive constants that may be different at each occurrence; they possibly depend on $D$ or some other parameters but not on $n$.

\section{Correlation decay for maximum weight matching on locally tree-like graphs}

\begin{df}[Bonus]
Let $G$ be a finite simple graph with edge weights $(w_e)_{e\in E(G)}$, and let $H$ be a subgraph of $G$. 
	 For $v\in V(H)$ and $e\in E(H)$, define
\[\B(v,H):=\W_H-\W_{H\setminus v}, \qquad \B(e,H):=\W_H-\W_{H\setminus e}. \]
These are called the bonus at vertex $v$ in $H$ and the bonus at edge $e$ in $H$, respectively.
\end{df}
The vertex bonuses satisfy the  cavity recursion (see \cite[Eq.~(2.2)]{lamsen2026}):
\begin{equation}\label{eq:bonus_cavity_recursion}
     \B(v, H) = \max_{u: (uv)\in E(H) }\big( w_{(uv)} -  \B(u, H \setminus v) \big)_+ .
\end{equation}
The edge bonus can be expressed in terms of vertex bonuses. Let $e=(uv)\in E(H)$.
Since every matching of $H$ either avoids $e$, or uses $e$ and then avoids both
$u$ and $v$, we have
\[\W_H =  \max\left (  \W_{H\setminus e},  w_{(uv)}+\W_{H\setminus\{u,v\}} \right).
\]
Therefore
\begin{align}\label{eq:edge_bonus_vertex_bonus}
\B(e,H) &= \left( w_{(uv)} +\W_{H\setminus\{u,v\}} -\W_{H\setminus e}\right)_+ \nonumber\\ &=\left(w_{(uv)}-\bigl(\W_{H\setminus u}-\W_{H\setminus\{u,v\}}\bigr)-
\bigl(\W_{H\setminus e}-\W_{H\setminus u}\bigr) \right)_+ \nonumber\\
&=\left( w_{(uv)} - \B(v,H\setminus u) -  \B(u,H\setminus e)\right)_+.
\end{align}
Conversely, vertex bonuses can be written in terms of edge bonuses. Let
$v_1,\ldots,v_d$ be the neighbors of $v$, and write $e_i=(vv_i)$. For each $i$,
let $G_i$ be the subgraph obtained from $G$ by deleting all edges incident to
$v$ except $e_i$. Then
\begin{equation}
\label{eq:vertex_to_edge_bonus}
    \B(v,G)=\max_{1\le i\le d} \B(e_i,G_i).
\end{equation}
Indeed, $  G_i\setminus e_i$ is the graph $G\setminus v$ together with an isolated copy of $v$, and hence $\W_{G_i\setminus e_i}=\W_{G\setminus v}.$
Also, $ \W_{G_i} = \max\left (  \W_{G\setminus v}, w_{e_i}+\W_{G\setminus\{v,v_i\}}\right).$
Thus
\[\B(e_i,G_i) =\left( w_{e_i}+\W_{G\setminus\{v,v_i\}}-\W_{G\setminus v} \right)_+ = \big (  w_{e_i} - \B(v_i, G \setminus v) \big)_+. \]
Taking the maximum over $i$ gives \eqref{eq:vertex_to_edge_bonus}, by \eqref{eq:bonus_cavity_recursion}.

The key ingredient in our proof is spatial correlation decay for MWM on bounded-degree graphs. In the locally tree-like setting with exponential edge weights, this was proved in \cite{lamsen2026}. The precise statement is as follows.
\begin{thm}[Correlation decay \cite{lamsen2026}, Theorem~1.3]
	\label{thm:vertex_bonus}
    Let $G$ be a finite connected graph whose degree is bounded above by $D$. Suppose that the edge weights are i.i.d.\ exponentially distributed with parameter $1$. Let $e\in E(G)$ and $r\geq 3$ such that $\bB_e^r(G)$ is a tree. Then there exist constants $c, C > 0$ depending only on $D$ such that
\[\E_w \sup_{A\subseteq \partial \bB_e^r(G)} |\ind_{\{e\in \M_{\bB_e^r(G)}\}} - \ind_{\{e\in \M_{\bB_e^r(G)\setminus A}\}}| \leq Ce^{-cr}. \]
    \end{thm}
Let $\mathsf{A}$ be the set of
boundary vertices of $\bB_e^r(G)$ that are matched by $\M_G$ to edges outside
$\bB_e^r(G)$. The restriction of $\M_G$ to $\bB_e^r(G)$ is a
maximum weight matching of $\bB_e^r(G) \setminus \mathsf{A}$. Therefore, $ \ind_{\{e\in \M_G\}} = \ind_{\{e\in \M_{\bB_e^r(G) \setminus \mathsf{A}}\}}$
almost surely. Consequently, by Theorem~\ref{thm:vertex_bonus}, for all $r \ge 3$,
\begin{align} \label{eq:corr_decay_ind}
    \prob_w\big( \ind_{\{e\in \M_G\}}\ne \ind_{\{e\in \M_{\bB_e^r(G)}\}}\big)
    &\le \E_w\sup_{A\subseteq \partial \bB_e^r(G)} \big| \ind_{\{e\in \M_{\bB_e^r(G)}\}} - \ind_{\{e\in \M_{\bB_e^r(G)\setminus A}\}} \big|  \le Ce^{-cr}.
\end{align}
The above correlation decay result readily yields the locality of the edge and vertex bonuses. 
\begin{lem}
	\label{lem: edge_bonus}
    Let $G$ be a finite simple graph of maximum degree at most $D$ with i.i.d.\ $\mathrm{Exp}(1)$ edge weights.  Then there exist constants  $c, C>0$, depending only on  $D$, such that the following estimates hold.
    \begin{itemize}
        \item[(i)]  Let  $e \in E(G)$ and  $r\geq 1$ be such that $\bB_e^r(G)$ is a tree.
        \[\E_w |\B(e, G) - \B(e, \bB_e^r(G))|^2 \leq Ce^{-cr}. \]
    \item[(ii)]   Let  $v \in V(G)$ and  $r\geq 1$ be such that $\bB_v^{r+1}(G)$ is a tree.
    \[ \E_w |\B(v, G) - \B(v, \bB_v^{r+1}(G))|^2 \leq Ce^{-cr}.\]
    \end{itemize}
\end{lem}
\begin{proof}
We first prove (i). Let $e=(uv)$. The cases
$r=1,2$ can be absorbed into the constant $C$, since
$0\le \B(e,H)\le w_e$ for every graph $H$ containing $e$. Hence assume
$r\ge 3$.

It follows from \eqref{eq:edge_bonus_vertex_bonus} that for any subgraph $H$
containing $e$, almost surely, 
\begin{equation} \label{eq:edge_MWM_rep}
    \{e\in \M_H\} =\{w_e > \B(v,H\setminus u)+\B(u,H\setminus e)\}. 
\end{equation}
We use the following elementary estimate. If $X\sim \mathrm{Exp}(1)$ and
$a,b\ge0$, then
\[ \E |(X-a)_+-(X-b)_+|^2
    \le 2\prob(\ind_{\{X>a\}}\ne \ind_{\{X>b\}}).\]
Applying this conditionally  with $H=G$ and $H=\bB_e^r(G)$ and using \eqref{eq:edge_MWM_rep}, we obtain
\[\E_w |\B(e,G)-\B(e,\bB_e^r(G))|^2 \le
2\prob_w\big( \ind_{\{e\in \M_G\}} \ne \ind_{\{e\in \M_{\bB_e^r(G)}\}} \big).\]
This proves (i) by applying \eqref{eq:corr_decay_ind}.

We now prove (ii). Let $v_1,\ldots,v_d$ be the neighbors of $v$, and write
$e_i=(vv_i)$. For each $i$, let $G_i$ be the graph obtained from $G$ by
deleting all edges incident to $v$ except $e_i$. Similarly, let $H_i$ be the
graph obtained from $\bB_v^{r+1}(G)$ by deleting all edges incident to $v$
except $e_i$. By \eqref{eq:vertex_to_edge_bonus},
\[   \B(v,G)=\max_{1\le i\le d} \B(e_i,G_i), \qquad   \B(v,\bB_v^{r+1}(G)) =
    \max_{1\le i\le d} \B(e_i,H_i).\]
Since the edge bonus of $e_i$ depends only on the connected component
containing $e_i$, and since $\bB_v^{r+1}(G)$ is a tree, this component in
$H_i$ is exactly $\bB_{e_i}^r(G_i)$. Thus $\B(e_i,H_i)=\B(e_i,\bB_{e_i}^r(G_i)).$
Moreover, $\bB_{e_i}^r(G_i)$ is a tree. Hence, by part (i), $ \E_w
    \left|
        \B(e_i,G_i)-\B(e_i,\bB_{e_i}^r(G_i))
    \right|^2
    \le Ce^{-cr}$ for each $i$.
Consequently, 
\[
\begin{aligned}
    \E_w  \left| \B(v,G)-\B(v,\bB_v^{r+1}(G))  \right|^2
    &\le \sum_{i=1}^d   \E_w  \left|   \B(e_i,G_i)-\B(e_i,\bB_{e_i}^r(G_i)) \right|^2 
    \le DCe^{-cr}.
\end{aligned}
\]
Renaming the constant $C$ proves (ii).
\end{proof}

\subsection{Local approximation and perturbative correlation decay}\label{sec: local_approx}

Theorem~\ref{thm:vertex_bonus} gives a correlation decay estimate for individual edge indicators, but this is not sufficient for the proof of Theorem~\ref{thm: annealed_CLT}. We will need a perturbative version of correlation decay, which compares the effect of deleting a single edge on the MWM weight with its effect on a local approximation.

Let $H$ be a subgraph of $G$. For $R \ge 1,$ define the local approximation of $\W_H$ by 
\begin{equation}\label{eq:MWM_local_H}
\W_H^{\mathrm{loc}, R} = \sum_{ e \in E(H) } w_e \ind_{\{e\in \M_{\bB_e^R(H)}\}}.    \end{equation}

\begin{thm}\label{thm:local_perturbative_bound} 
There exist an integer $A_0 \ge 2$ and constants $C >0$, and $c>0$, depending only on $D$, such that the following holds. Let $R\ge 1$ and $A\ge A_0$, and let $G$ be a finite simple graph with maximum degree at most $D$ such that $\bB_e^{AR}(G)$ is a tree for every edge $e\in E(G)$. Then
for any $f \in E(G),$
\begin{equation} \label{eq:local_perturbative_bound}
 \big| \big( \E_w \W_G  -  \E_w \W_{G \setminus f} \big) - \big(\E_w \W^{\mathrm{loc},R}_G  -  \E_w \W^{\mathrm{loc}, R}_{G \setminus f} \big)  \big|   \le C e^{-cR}.
\end{equation}
\end{thm}
The proof of Theorem~\ref{thm:local_perturbative_bound} relies on the following lemma and two propositions. For notational brevity, for any subgraph $H$ of $G$, define 
\[ X_e(H) = \ind_{ \{e\in \M_{H}\}}, \qquad X_e^R(H) = \ind_{ \{e\in \M_{\bB_e^R(H)}\}}, \]
with the convention that $X_e(H) =  X_e^R(H) = 0 $ if $e \not \in E(H).$
We can write
\[ \W_H = \sum_{ e \in E(H) } w_e X_e(H), \qquad \W^{\mathrm{loc},R}_H = \sum_{ e \in E(H) } w_e X^R_e(H).  \]
\begin{lem}
\label{lemma:near_delete_edge_tail}
There exists $\delta_0=\delta_0(D)\in(0,1)$ such that the
following holds. Let $0<\delta\le \delta_0$. Then there exist constants $c, C >0$, depending only on $\delta$ and $D$, such that for every
finite simple graph $G$ of maximum degree at most $D$, every edge $f\in E(G)$,
and every $R\ge1$ for which $\bB_f^R(G)$ is a tree,
\[\Big| \sum_{e:\,d(e,f)\le \delta R} \big( \E_w w_e X_e(G) -
\E_w w_e X_e^R(G)\big) \Big| \le C e^{-cR}.\]
\end{lem}
\begin{prop}
\label{prop:annular_delete_edge}
Fix $\delta \in (0, 1)$. There exist constants 
$c, C>0$, depending only on $\delta$ and $D$, such that the following holds.
Let $G$ be a finite tree of maximum degree at most $D$ and let $f\in E(G)$. 
Then, for every $R\ge 1$,
\[ \Big| \sum_{e:\delta R\le d(e,f)\le R} \left( \E_w w_e X_e^R(G)
- \E_w w_e X_e^R(G\setminus f) \right) \Big| \le C e^{-cR}. \]
\end{prop}
\begin{prop}
\label{prop:far_delete_edge_tail}
Fix $\delta \in (0, 1)$. There exist an integer $A_0 \ge 2$ and constants $C >0$, and $c>0$, depending only on $\delta$ and $D$, such that the following holds.  Let $R\ge 1$ and $A\ge A_0$, and let $G$ be a finite simple graph with maximum degree at most $D$ such that $\bB_e^{AR}(G)$ is a tree for every edge $e\in E(G)$. Then, for
every edge $f\in E(G)$,
\begin{equation}\label{eq:far_delete_edge_tail}
  \Big| \sum_{e:\,d(e,f)\ge \delta R} \big(\E_w w_e X_e(G) - \E_w w_e X_e(G\setminus f) \big) \Big| \le C e^{-cR}.  
\end{equation}
\end{prop}

We first explain how Theorem~\ref{thm:local_perturbative_bound} follows from these three results. Fix $0<\delta<1$. Using the triangle inequality and the fact that
$X_e^R(G)=X_e^R(G\setminus f)$ whenever $d(e,f)>R$, we can bound the left
hand side of \eqref{eq:local_perturbative_bound} by
\begin{gather}
\Big| \sum_{e:\,d(e,f)\le \delta R} \left( \E_w w_e X_e(G)-\E_w w_e X_e^R(G) \right) \Big|
+
\Big| \sum_{e:\,d(e,f)\le \delta R} \left(
\E_w w_e X_e(G\setminus f)-\E_w w_e X_e^R(G\setminus f) \right) \Big|
\label{eq:perturb_bd1}
\\[0.5em] {}+
\Big| \sum_{e:\,\delta R<d(e,f)\le R} 
\left( \E_w w_e X_e^R(G)-\E_w w_e X_e^R(G\setminus f) \right) \Big|
\label{eq:perturb_bd2}
\\[0.5em] {}+
\Big| \sum_{e:\,d(e,f)>\delta R}
\left( \E_w w_e X_e(G)-\E_w w_e X_e(G\setminus f)\right)\Big|.
\label{eq:perturb_bd3}
\end{gather}
Choose $\delta>0$ sufficiently small and assume $A\ge 1$. Since $\bB_f^R(G)$ is a tree,  the hypotheses of
Lemma~\ref{lemma:near_delete_edge_tail} are satisfied for the first term in
\eqref{eq:perturb_bd1}. Hence this term is bounded by $Ce^{-cR}$.

The second term is bounded by the same argument. Indeed, although
$f\notin E(G\setminus f)$, the proof of
Lemma~\ref{lemma:near_delete_edge_tail} applies verbatim, since it uses $f$
only to measure distance in the ambient graph $G$ and to ensure that the
relevant local balls are contained in the tree $\bB_f^R(G)$. Hence
\eqref{eq:perturb_bd1} is bounded by $Ce^{-cR}$.

For the same choice of $\delta$, \eqref{eq:perturb_bd2} is bounded above by
$Ce^{-cR}$ for every $A\ge 2$, by Proposition~\ref{prop:annular_delete_edge}.
To apply Proposition~\ref{prop:annular_delete_edge}, note that although the
ambient graph $G$ need not be a tree, it may be replaced by $\bB_f^{2R}(G)$, which is 
indeed a tree by the hypothesis of
Theorem~\ref{thm:local_perturbative_bound}. Moreover, for any
edge $e$ with $d(e,f)\le R$, we have
$\bB_e^R(G)\subseteq \bB_f^{2R}(G)$, and hence $X_e^R(G)=X_e^R\bigl(\bB_f^{2R}(G)\bigr).$

Finally, again for the same
$\delta$, Proposition~\ref{prop:far_delete_edge_tail} implies that, for all
sufficiently large $A$, \eqref{eq:perturb_bd3} is bounded above by
$Ce^{-cR}$. Combining these three estimates proves
Theorem~\ref{thm:local_perturbative_bound}.

The proofs of Lemma~\ref{lemma:near_delete_edge_tail}, Propositions~\ref{prop:annular_delete_edge} and \ref{prop:far_delete_edge_tail} are  lengthy, so we postpone them to Section~\ref{sec: local_approximation}.

\subsection{CLT for MWM on locally tree-like graphs}\label{sec:chatterjee_CLT}
Fix $D\ge2$ and $\rho>0$, and let $\mathcal K_n$ denote the collection of graphs on $n$ vertices with maximum degree at most $D$ and number of edges at least $\rho n$.
For $\ell\ge1$, define $\mathcal K_n^{\square \ell} :=\bigl\{ G\in\mathcal K_n:\operatorname{girth}(G)\ge\ell \bigr\}.$

We have the following uniform CLT for MWM.  Let $(\ell_n)_{n \ge 1}$ be any divergent sequence of positive integers. 
\begin{equation}\label{eq:CLT_locally_tree}
    \max_{ G \in \mathcal{K}_{n}^{\square \ell_n} } d_{\mathrm{KS}} \Big( \mathcal{L}\Big( \frac{\W_G - \E_w \W_G}{\sqrt{\Var_w(\W_G)}} \Big), \mathrm N(0, 1)\Big) \to 0, \quad \text{ as } n \to \infty. 
\end{equation} 
The CLT is obtained by combining Chatterjee's normal approximation result \cite{Chatterjee14} with the correlation decay estimate of Lemma~\ref{lem: edge_bonus}.
\begin{thm}[Chatterjee's perturbative normal approximation
{\cite[Corollary~3.2]{Chatterjee14}}]
\label{cor:chatterjee_normal_approx}
Let $I$ be a finite index set, let $X=(X_i)_{i\in I}$ be a family of
independent random variables, and let $\widetilde X$ be an independent
copy of $X$. For $A\subseteq I$, let $X^A$ be obtained from $X$ by
replacing $X_i$ with $\widetilde X_i$ for every $i\in A$. Define, for $i\notin A$, $\Delta_i^A f :=
f(X^A)-f(X^{A\cup\{i\}})$ and $\Delta_i f:=\Delta_i^\emptyset f.$
Suppose that $\sigma^2:=\Var(f(X))>0$ and that there are constants
$c(i,j)\ge0$ such that
\[ \Cov\left( \Delta_i f\,\Delta_i^A f, \Delta_j f\,\Delta_j^B f \right)
\le c(i,j) \]
for all $A\subseteq I\setminus\{i\}$ and
$B\subseteq I\setminus\{j\}$. Then
\[ d_{\mathrm{KS}}\Big( \mathcal L\Big(\frac{f(X)-\E f(X)}{\sigma}\Big), \mathrm N(0,1)\Big)
\le \frac{\sqrt{2}}{\sigma} \Big(\sum_{i,j\in I}c(i,j)\Big)^{1/4}
+ \frac{1}{\sigma^{3/2}} \Big(\sum_{i\in I}\E|\Delta_i f|^3\Big)^{1/2}. \]
\end{thm}

Fix $G\in\mathcal{K}_{n}^{\square \ell_n}$  , and write $m:=|E(G)|$ and  $\sigma_G^2:=\Var_w(\W_G)$.
Here $m\asymp n$, and Lemma~\ref{lem: variance_lower_bound} gives
$\sigma_G^2\asymp m$, uniformly in $G$.

Let $(\widetilde w_e)_{e\in E(G)}$ be an independent copy of the edge
weights. For $S\subseteq E(G)$, let $\W_G^S$ denote the MWM weight after
replacing $w_e$ by $\widetilde w_e$ for $e\in S$, and, for $e\notin S$,
set
\[ \Delta_e : = \W_G  - \W_G^{\{e\}},  \quad \Delta^S_{e}:= \W_G^S-\W_G^{S\cup\{e\}}.
\]
Choose $a>0$ such that $2a\log D<1$, and set $r_n:=\left\lfloor \min \big(  (\ell_n-4)/2,  a\log n \big) \right\rfloor.$
Then $r_n\to\infty$, $D^{2r_n+1}=o(n)$, and since
$2r_n+3<\ell_n$,
$\bB_e^{r_n}(G)$ is a tree for every $e\in E(G)$.

If $\B^S(e, G) $ denotes the bonus of the edge $e$ after
replacing $w_f$ by $\widetilde w_f$ for every edge $f\in S$, then  
\[\Delta^S_{e} = \B^S(e, G) - \B^{S \cup \{e\}}(e, G). \]
Hence, by the correlation decay of edge-bonus as given in Lemma~\ref{lem: edge_bonus}(i), H\"older's inequality, and the bounded moments of the
exponential weights,  the product $\Delta_{e}\Delta^S_{e}$ can be replaced in $L^2$ by
its analogue computed on $\bB_e^{r_n}(G)$, i.e., by 
\[ \Big( \B(e, \bB_e^{r_n}(G)) - \B^{\{e\}}(e, \bB_e^{r_n}(G)) \Big) \Big( \B^S(e, \bB_e^{r_n}(G)) - \B^{S \cup \{e\}}(e, \bB_e^{r_n}(G)) \Big), \]
with error at most
$Ce^{-cr_n}$, uniformly in $e$ and $S$. The localized variables
corresponding to $e$ and $e'$ are independent whenever their
$r_n$-neighborhoods are disjoint. Hence we may choose
$c_G(e,e')$ so that
\[ c_G(e,e')
\le C\ind_{\{ \bB_e^{r_n}(G)\cap\bB_{e'}^{r_n}(G)\ne\emptyset\}} + Ce^{-cr_n},\]
and consequently
\[ \sum_{e,e'\in E(G)}c_G(e,e') \le C\left(  mD^{2r_n+1}+m^2e^{-cr_n} \right).\]
Moreover, since $|\Delta_{e}|
\le |w_e-\widetilde w_e|$, we have $\sum_{e\in E(G)}\E_w|\Delta_{e}|^3
\le Cm.$ Theorem~\ref{cor:chatterjee_normal_approx} now gives
\[
d_{\mathrm{KS}} \Big(
    \mathcal L \Big ( \frac{\W_G-\E_w\W_G}{\sqrt{\Var_w(\W_G)}}
    \Big),\mathrm N(0,1)\Big) \le
C\Big(  \frac{D^{2r_n+1}}{m}+e^{-cr_n} \Big)^{1/4} + Cm^{-1/4} =o(1),\]
uniformly in $G$. This proves \eqref{eq:CLT_locally_tree}.

\section{Maximum weight matching on configuration model}
One of the main ideas in the proof of Theorem~\ref{thm: annealed_CLT}
is to first prove the relevant central limit theorem for the configuration
model (which we will define below)  and then transfer it to the uniformly random simple graph.

A multigraph is a triple $G=(V(G), \cE(G),\partial_G)$, where $\partial_G:\cE(G)\to \{\{u,v\}:u,v\in V(G)\}$
    assigns to each edge copy its unordered pair of endpoints. Distinct elements of $\cE(G)$ are regarded as distinct edge copies, even if they have the same endpoints. In particular, loops are allowed, corresponding to $\partial_G(e)=\{v\}$ for some $v\in V(G)$.  The degree of a vertex $v\in V(G)$ is the number of edge ends incident to $v$, with each loop counted twice.  
    We write $\mathcal{MG}_n$ for the set of all finite multigraphs on $n$ vertices, and
$\mathcal{MG}_{n,D}$ for the subset of multigraphs with maximum degree at most~$D$.

We extend the preceding simple-graph notation to multigraphs by interpreting edges as edge copies and distances in the underlying simple graph. Let $G = (V(G), \cE(G), \partial_G)$ be a multigraph. For $A \subseteq V(G)$, we write $G \setminus A$ for the induced submultigraph of $G$ on the vertex set $V(G)\setminus A$. Similarly, for $F \subseteq \cE(G)$, we write $G \setminus F := \bigl(V(G), \cE(G)\setminus F,\ \partial_G|_{\cE(G)\setminus F}\bigr).$

For a multigraph $G=(V(G),\cE(G),\partial_G)$, vertex distances are
measured in the underlying simple graph obtained by deleting loops and
identifying parallel edges. For $v\in V(G)$ and $e\in\cE(G)$, define $d_G(v,e) := \min_{x\in\partial_G(e)}d_G(v,x).$ For $r\ge0$, let $\bB_v^r(G)$ and $\bB_e^r(G)$ denote the induced
submultigraphs of $G$ on the vertex sets $\{x\in V(G):d_G(x,v)\le r\}$ and $\{x\in V(G):d_G(x,e)\le r\},$
respectively.

A matching in  a multigraph $G=(V(G),\cE(G),\partial_G)$  is a set $M\subseteq \cE(G)$ such that $|\partial_G(e)|=2$ for all $e \in M$ and $\partial_G(e)\cap \partial_G(f)=\emptyset$ for all distinct $e,f\in M$.
Equivalently, a matching in a multigraph is a collection of edge copies such that no vertex is incident to more than one chosen edge. In particular, loops are automatically excluded. Note that this extends the usual notion of a matching in a simple graph.  Given a multigraph $G$ with edge-weights $(w_e)_{e\in \cE(G)}$, the weight of a matching $M$ is  $\sum_{e\in M} w_e$. A maximum weight matching (MWM) of $G$ is a matching $\M_G$ of maximal weight, and its weight is denoted by $\W_G := \sum_{e\in \M_G} w_e.$ 

\subsection{Configuration model}\label{subsec:config}

Let $\mathbf{d}= (d_1, d_2, \ldots, d_n)$ be a degree sequence satisfying \eqref{eq: degree_seq_cond}, where we suppress the dependence on $n$ for ease of notation. A standard method to construct a random simple graph $\G_n$ with this degree sequence is via the configuration model introduced in \cite{Bollobas}, which we now describe. Assume that $N := \sum_{i=1}^n d_i$ is even.  Assign $d_i$ half-edges to vertex $i$, and let $\mathbb{H}=\{h_1,\ldots,h_N\}$
denote the set of all half-edges. Write $\mathcal{P}_{\mathbb H}$ for the set of pairings of $\mathbb H$, that is, the set of partitions of $\mathbb H$ into $N/2$ blocks of size $2$. We identify each pairing with the corresponding fixed-point-free involution $\sigma$ on $\mathbb H$, that is, a permutation satisfying
\[ \sigma^2=\mathrm{id} \qquad\text{and}\qquad
\sigma(h)\neq h \quad \text{for all } h\in\mathbb H.\]

Define a map $\chi:\mathcal{P}_{\mathbb H}\to\mathcal{MG}_n$ as follows. For a pairing $\sigma\in\mathcal{P}_{\mathbb H}$, let $\chi(\sigma)$ be the multigraph with vertex set $V=\{1,\ldots,n\}$ obtained by placing the half-edges $h_1,\ldots,h_{d_1}$ at vertex $1$, the half-edges $h_{d_1+1},\ldots,h_{d_1+d_2}$ at vertex $2$, and so on. We join the half-edges $h_i$ and $h_j$ to form an edge in $\chi(\sigma)$ if and only if $\sigma(h_i)=h_j$. Note that $\chi(\sigma)$ may contain self-loops and multiple edges.

When $\bs$ is drawn uniformly from $\mathcal{P}_{\mathbb H}$, the resulting random multigraph $\C_n:=\chi(\bs)$
is called the configuration model. It is well known that
\[ \G_n \stackrel{d}{=} \C_n \mid \{\text{$\C_n$ is simple}\},
\]
where $\G_n$ denotes the uniformly chosen simple graph on $n$ vertices with degree sequence $\mathbf{d}$. Moreover, under \eqref{eq: degree_seq_cond}, the probability that $\C_n$ is simple is bounded away from $0$.

One of the main tools for studying the configuration model is the switching operation. For $1\le x,y\le N$, define the switch $\sw^{h_x,h_y}:\mathcal{P}_{\mathbb H}\to\mathcal{P}_{\mathbb H}$
as follows. Let $\sigma\in\mathcal{P}_{\mathbb H}$. If $\{h_x,\sigma(h_x)\} = \{h_y,\sigma(h_y)\}$, then we set $\sw^{h_x,h_y}(\sigma):=\sigma.$
Otherwise, if $\{h_x,\sigma(h_x)\} \ne \{h_y,\sigma(h_y)\}$, we define $\sigma':=\sw^{h_x,h_y}(\sigma)$ by deleting the pairs $\{h_x,\sigma(h_x)\}$ and $\{h_y,\sigma(h_y)\}$ from $\sigma$, and replacing them with the pairs $\{h_x,h_y\}$ and $\{\sigma(h_x),\sigma(h_y)\}$. Equivalently,
\begin{equation*}
\sigma'(h_x)=h_y, \quad  \sigma'(\sigma(h_x))=\sigma(h_y), \quad
\sigma'(h_y)=h_x, \quad \sigma'(\sigma(h_y))=\sigma(h_x),
\end{equation*}
and
\[\sigma'(h_z)=\sigma(h_z)
\quad\text{whenever}\quad
h_z\notin\{h_x,h_y,\sigma(h_x),\sigma(h_y)\}. \]

\begin{figure}[h!]
\centering
\includegraphics[scale=.55]{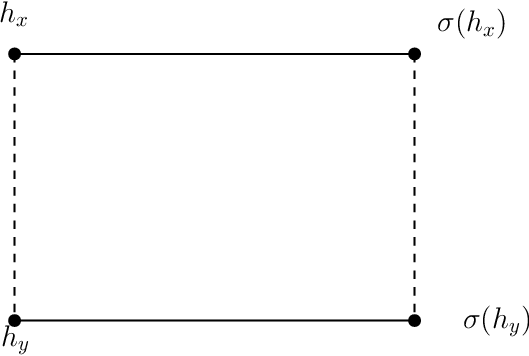}
\caption{The solid lines represent the pairings in $\sigma$, while the dashed lines represent the pairings in $\sw^{h_x,h_y}(\sigma)$.}
\label{fig:universe}
\end{figure}

Note that a switch always preserves the degree sequence of the resulting multigraph. We shall also use the following elementary fact.

\medskip

\noindent\textbf{Claim.} For any $1\le x,y\le N$, if $\bs\sim \mathrm{Unif}(\mathcal{P}_{\mathbb H})$, then $\sw^{h_x,\bs(h_y)}(\bs)\sim \mathrm{Unif}(\mathcal{P}_{\mathbb H}).$
\begin{proof}
For fixed $x,y$, define
\[ T_{x,y}(\sigma):=\sw^{h_x,\sigma(h_y)}(\sigma), \qquad \sigma\in\mathcal{P}_{\mathbb H}.
\]
It suffices to show that $T_{x,y}$ is a bijection on $\mathcal{P}_{\mathbb H}$.
If $x=y$ or $h_x=\sigma(h_y)$, then the two arguments
$h_x$ and $\sigma(h_y)$ of the switch belong to the same pair of
$\sigma$. Hence, by definition, $T_{x,y}(\sigma)=\sigma.$

Now assume that $x\ne y$ and $h_x\ne\sigma(h_y)$. Then also
$h_y\ne\sigma(h_x)$, so the pairs containing $h_x$ and $h_y$
are distinct. Write
\[ \sigma=\{\{h_x,h_{x'}\},\{h_y,h_{y'}\}\}\cup\omega, \]
where $h_x,h_y,h_{x'},h_{y'}$ are all distinct, and $\omega$ is
the union of the remaining pairs. Then
\[ T_{x,y}(\sigma) = \{\{h_x,h_{y'}\},\{h_y,h_{x'}\}\}\cup\omega.\]
Applying $T_{x,y}$ once more returns the original pairing $\sigma$.
Hence $T_{x,y}$ is an involution, and therefore a bijection on $\mathcal{P}_{\mathbb H}$.
\end{proof}

\subsection{Concentration for configuration model}
We first establish an Efron-Stein-type variance bound for functions of the configuration model $\C_n$. This is an edge-deletion reformulation of the Poincaré inequality for the random-transposition walk on the symmetric group. Related Efron–Stein formulations for uniform permutations appear in \cite[Lemma 2.1]{chengoldsteinrollin} and \cite[Lemma 4.4]{BarbourRollin}. We include a derivation for completeness.

\begin{lem}[Efron-Stein bound for configuration model]
    \label{cor: conf_var_bound}
	Let $f:\mathcal{MG}_n \to\mathbb{R}$ be a real-valued function defined on multigraphs on $n$ vertices. There exists an absolute constant $C$ such that
    \[ \Var(f(\C_n)) \leq C\E \sum_{e \in \cE(\C_n)} \big(f(\C_n) - f(\C_n \setminus e )\big)^2 + \frac{C}{N}\E  \sum_{e,e' \in \cE(\C_n)}  \big(f(\C_n \setminus  e ) - f(\C_n \setminus \{e, e'\})\big)^2,\]
 where $N = 2 | \cE(\C_n)|$.
\end{lem}

We begin with a variance bound for functions on $S_N$, the symmetric group on $N$ elements. This is useful because a uniformly random pairing can be represented using a uniformly random permutation. The random-transposition walk on $S_N$, viewed as a reversible Markov chain, has spectral gap $2/(N-1)$; see \cite{DS} and \cite[Example~3.12]{BobkovTetali}. The corresponding Poincaré inequality gives the following bound.

\begin{lem}
	\label{lem: var_sym_group}
	There exists an absolute constant $C>0$ such that for any function $g:S_N\to \mathbb R$,
	\[ \Var_{\bp} (g(\bp)) \leq \frac{C}{N}\sum_{1 \le i, j \le N} \E_{\bp} (g(\bp) - g(\bp\circ(i\;j)))^2,\]
	 where  $\bp$ is uniformly distributed on $S_N$.
\end{lem} 
We now deduce Lemma~\ref{cor: conf_var_bound} from Lemma~\ref{lem: var_sym_group}.

\begin{proof}[Proof of Lemma \ref{cor: conf_var_bound}]
    We will abuse the notation a little by writing $\mathcal{P}_\mathbb{H} \subseteq S_{\mathbb{H}}$, where $S_\mathbb{H}$ is the permutation group of all symbols in $\mathbb{H}$. Define $\psi: S_\mathbb{H} \to \mathcal{P}_\mathbb{H}$ by
    \[ \psi(\pi) = (\pi(h_1) \; \pi(h_2))\circ \cdots \circ (\pi(h_{N-1})\;\pi(h_N)). \]
    In other words, $\psi(\pi)$ corresponds to the pairing $\{\pi(h_1), \pi(h_2)\}, \ldots, \{\pi(h_{N-1}), \pi(h_N)\}$. Note that
    \[ |\psi^{-1}(\sigma)| = 2^{\frac{N}{2}} \Big(\frac{N}{2}\Big)!\quad \text{for all $\sigma \in \mathcal{P}_\mathbb{H}$.} \]
    Hence, if $\bp \sim \mathrm{Unif}(S_\mathbb{H})$, then $\psi(\bp) \sim \mathrm{Unif}(\mathcal{P}_\mathbb{H})$.

   Let $g:\mathcal{P}_\mathbb{H} \to \mathbb{R}$ be any function. Applying Lemma~\ref{lem: var_sym_group} to the function $g \circ \psi:S_\mathbb{H} \to \mathbb{R}$, we obtain 
    \begin{equation} \label{eq:vbd_1}
         \Var (g(\bs)) \leq \frac{C}{N} \sum_{1 \le i, j \le N} \E_{\bp} \big(g(\psi(\bp)) - g(\psi(\bp \circ (h_i\; h_j)))\big)^2,
   \end{equation}
    where $\bs := \psi(\bp) \sim \mathrm{Unif}(\mathcal{P}_\mathbb{H})$.

    Fix $i \ne j$ and $\pi \in S_{\mathbb{H}}$. So, $\pi(h_i) \ne \pi(h_j)$. Let $\sigma := \psi(\pi)$ and $\hat \sigma := \psi(\pi\circ (h_i\; h_j))$. 
    If $\pi(h_i)=\sigma(\pi(h_j))$, then $\hat\sigma=\sigma$, since
    $\pi(h_i)$ and $\pi(h_j)$ are already paired in $\sigma$.
    In this case,
    \[  \hat\sigma=\sw^{\pi(h_i),\sigma(\pi(h_j))}(\sigma),
    \]
    because the two arguments of the switch coincide. Assume otherwise, that is, $\pi(h_i) \neq \sigma(\pi(h_j))$.
    Then in $\sigma$, we have the pairs $ \{\pi(h_i), \sigma(\pi(h_i))\} \text{ and } \{\pi(h_j), \sigma(\pi(h_j))\},$
    whereas $\hat \sigma$ has the pairs
    \[  \{\pi(h_j), \sigma(\pi(h_i))\} \text{ and } \{\pi(h_i), \sigma(\pi(h_j))\}.
    \]
    Consequently, if $\pi(h_i) \neq \sigma(\pi(h_j))$, then $ \hat \sigma = \sw^{\pi(h_i), \sigma(\pi(h_j))}(\sigma).$
    For $i = j$, we trivially have $\hat \sigma = \sw^{\pi(h_i), \sigma(\pi(h_j))}(\sigma) = \sigma$. Therefore, in all cases $1 \le i, j \le N,$
\[  \hat \sigma = \sw^{\pi(h_i), \sigma(\pi(h_j))}(\sigma). \]
    Since $\pi$ is a bijection on $\mathbb{H}$, we have 
    \begin{align*}
        \sum_{i, j} \big(g(\psi(\pi)) - g(\psi(\pi \circ (h_i\; h_j)))\big)^2 
        &=\sum_{i, j} \big(g(\sigma) - g(\sw^{\pi(h_i), \sigma(\pi(h_j))}(\sigma))\big)^2\\
        &=\sum_{x, y}\big(g(\sigma) - g(\sw^{h_x, \sigma(h_y)}(\sigma))\big)^2.
    \end{align*}
    Combined with \eqref{eq:vbd_1}, this yields 
    \begin{equation}
    \label{eq: efron_stein_pairing}
    \Var (g(\bs)) \leq \frac{C}{N}\sum_{x, y}\E \big(g(\bs) - g(\sw^{h_x, \bs(h_y)}(\bs))\big)^2.
    \end{equation}
    Finally, we would like to recast \eqref{eq: efron_stein_pairing} as a variance bound for functions on multigraphs instead of pairings. Let $f$ be a real-valued function defined on multigraphs, and take $g = f\circ \chi$. 

    Take $1 \le x, y \le N$ and $\sigma \in \mathcal{P}_{\mathbb{H}}$.
    Let $\sigma' := \sw^{h_x, \sigma(h_y)}(\sigma)$. Suppose that $x \ne y$ and $h_x \ne \sigma(h_y)$. Consider the following `path' from $\sigma$ to $\sigma'$. From $\sigma$, remove the pairs $\{ h_x, \sigma(h_x)\}$ and $\{ h_y, \sigma(h_y)\}$ in succession. Then add back the pairs $\{ h_y, \sigma(h_x)\}$ and $\{ h_x, \sigma(h_y) \}$ in succession to obtain $\sigma'$. Adding these two pairs is the reversal of deleting the pairs $\{ h_x, \sigma'(h_x)\}$ and $\{ h_y, \sigma'(h_y)\}$ from $\sigma'$ in that order. These operations can also be performed on the space of multigraphs via the map $\chi$ as follows:
    \begin{align*}
    G &\rightarrow G \setminus  e_\sigma(h_x, \sigma(h_x))    \rightarrow G \setminus \big \{ e_\sigma(h_x, \sigma(h_x)), e_\sigma(h_y, \sigma(h_y))\big \} \\
    &= G' \setminus \big \{ e_{\sigma'}(h_x, \sigma'(h_x)), e_{\sigma'}(h_y, \sigma'(h_y))\big \} \rightarrow G' \setminus  e_{\sigma'}(h_x, \sigma'(h_x)) \rightarrow G',
    \end{align*} 
    where $G:=\chi(\sigma)$, $G':=\chi(\sigma')$, and for $\tau\in\{\sigma,\sigma'\}$,
    the symbol $e_\tau(a,\tau(a))$ denotes the unique edge copy in
    $\cE(\chi(\tau))$ corresponding to the paired half-edges $\{a,\tau(a)\}$.

    For notational brevity, let us denote $e_\tau(x) = e_\tau(h_x, \tau(h_x))$. Then
    \begin{align*}
     (f(G) - f(G'))^2   &\leq  4 \Big[ \big(f(G) - f( G\setminus e_\sigma(x) ) \big)^2 + \big (f(G\setminus e_\sigma(x) ) - f(G\setminus\{e_\sigma(x), e_\sigma(y)\})\big )^2\\
    &\quad + \big(f(G') - f( G'\setminus e_{\sigma'}(x)) \big)^2 + \big (f(G'\setminus e_{\sigma'}(x)) - f(G'\setminus\{e_{\sigma'}(x), e_{\sigma'}(y)\})\big )^2 \Big].
        \end{align*}
    This inequality is also true when $x = y$ or $h_x =  \sigma(h_y)$, since in these cases $\sigma' = \sigma$ and hence $G' = G$, so the left-hand side is zero.  
    Recall that when $\bs \sim \mathrm{Unif}(\mathcal{P}_\mathbb{H})$, then so is $\bs^{x, y} := \sw^{h_x, \bs(h_y)} (\bs)$ for any $x, y$.  
    Write $\bG = \chi(\bs)$ and $\bG^{x,y} = \chi(\bs^{x, y})$. Then for any $x,y$, we have 
    \begin{align*}
     \E \big(f(\bG^{x, y}) - f( \bG^{x, y} \setminus e_{\bs^{x, y}}(x) ) \big)^2   &=  \E \big(f(\bG) - f( \bG \setminus e_{\bs}(x) ) \big)^2.
    \end{align*}
    Similarly, for any $x, y$, 
    \begin{align*}
    \E \big( f( \bG^{x, y} \setminus e_{\bs^{x, y}}(x) ) &-  f( \bG^{x, y} \setminus\{e_{\bs^{x, y}}(x), e_{\bs^{x, y}}(y)\}) \big)^2 \\
    &=  \E \big(f( \bG \setminus e_{\bs}(x) ) -  f( \bG \setminus\{e_{\bs}(x),e_{\bs}(y) \}) \big)^2.
    \end{align*}
    The above discussion coupled with \eqref{eq: efron_stein_pairing} yields 
    \begin{align*}
        \Var (f(\bG))
        &\le \frac{C'}{N} \sum_{x, y} \E \Big[ \big( f(\bG) -  f( \bG \setminus e_{\bs}(x) ) \big)^2 \\
        &\qquad\qquad\qquad\qquad\qquad\quad + \big(f( \bG \setminus e_{\bs}(x) ) -  f( \bG \setminus\{e_{\bs}(x),e_{\bs}(y) \}) \big)^2 \Big].
    \end{align*}
    Now, for each realization of $\bs$, every edge copy $e\in\cE(\bG)$ is equal to $e_{\bs}(x)$ for exactly two values of $x\in[N]:=\{1, 2, \ldots, N\}$. Hence
    \[ \sum_{x=1}^N \big( f(\bG) -  f( \bG \setminus e_{\bs}(x)) \big)^2
    = 2\sum_{e\in\cE(\bG)} \big(f(\bG)-f(\bG\setminus e )\big)^2.  \]
    Likewise, each ordered pair $(e,e')\in \cE(\bG)\times\cE(\bG)$ is represented by exactly four ordered pairs $(x,y)\in[N]^2$ such that
    $e_{\bs}(x)=e$ and $e_{\bs}(y)=e'$. Therefore,
    \[ \sum_{x,y}
    \big(f( \bG \setminus e_{\bs}(x)) -  f( \bG \setminus\{e_{\bs}(x),e_{\bs}(y) \}) \big)^2
    = 4\sum_{e,e'\in\cE(\bG)} \big(f(\bG\setminus e)-f(\bG\setminus\{e,e'\})\big)^2. \]
    Since $\bG$ has the same distribution as $\C_n$, the theorem follows after adjusting the constant.
\end{proof}

\subsection{Variance bounds on the weight of the maximum matching}
In this subsection, we gather several variance bounds that will be used in the proof of Theorem~\ref{thm: annealed_CLT}. 

We will first show that for a simple bounded-degree graph $G$,  $\Var(\W_G)$ is comparable to $|E(G)|$. 
\begin{lem}
	\label{lem: variance_lower_bound}
	There exists a constant $c>0$ depending only on $D$ such that for any graph $G \in \mathcal{G}_{n, D}$, 
    \[ c|E(G)| \le  \Var_w (\W_{G}) \le |E(G)|. \]
\end{lem}
\begin{proof}
 The upper bound follows easily from the Efron-Stein inequality. Indeed, for each $e \in E(G)$, let $\W_G^{\{e\}}$ be the maximum matching weight after replacing $w_e$ by an independent copy $\widetilde w_e$. By the Efron-Stein inequality,
	\[ \Var_w(\W_G) \le \frac12 \sum_{e \in E(G)} \E_w \bigl(\W_G-\W_G^{\{e\}}\bigr)^2.\]
	Replacing  $w_e$ by $\widetilde w_e$ changes the weight of any matching by at most $|w_e-\widetilde w_e|$, so $|\W_G- \W_G^{\{e\}}| \le |w_e- \widetilde w_e|$. Hence
	\begin{equation}
    \label{eq: Efron-Stein}
	\Var_w(\W_G) \le \frac12 \sum_{e \in E(G)} \E_w (w_e-\widetilde w_e)^2 = |E(G)|.
	\end{equation}
 
 The proof of the lower bound is similar to that of \cite[Lemma~3.2]{Cao}, which in turn is based on an idea in \cite{Chatterjee19}. We include a proof for the convenience of the reader.

 For the lower bound, it suffices to find constants $c_1,c_2>0$ such that if $|E(G)|$ is sufficiently large, then for all $a<b$ with $b-a\leq c_1\sqrt{|E(G)|}$, one has $\prob_w(a\leq \W_{G}\leq b)\leq 1-c_2$.
	
	Fix $\alpha>0$, and set $\varepsilon = \alpha |E(G)|^{-1/2}$. For $e\in E(G)$, define $\widehat{w}_e = w_e/(1-\varepsilon)$. Let $\widehat{\W}_{G}$ be the corresponding total weight of the MWM using the weights $(\widehat{w}_e)$. By \cite[Lemma~1.2]{Chatterjee19} and \cite[Corollary~1.8]{Chatterjee19}, there exists a constant $C_1>0$ such that for $|E(G)|$ sufficiently large,
	\begin{align}
    \label{eq: fluc_W_G}
	\prob_w(a\leq \W_{G}\leq b) \leq & \frac{1}{2}\big(1+\prob_w(|\W_{G} - \widehat{\W}_{G}| \leq b-a) + C_1\alpha \big).
	\end{align}
	Now, choose $\alpha>0$ sufficiently small such that $C_1\alpha \leq 1/2$.
	On the other hand, since $\widehat{\W}_{G} = \W_{G}/(1-\varepsilon)$, one has $|\W_{G} - \widehat{\W}_{G}| = \W_{G}\cdot \frac{\varepsilon}{1-\varepsilon}$. We can always upper bound the weight of any fixed maximal matching in $G$ (a maximal matching is a matching in $G$ such that it is not a proper subset of any matching in $G$) by $\W_G$. Such a matching has size at least $(2D)^{-1}|E(G)|$. Together with a Chernoff bound, one can show that whenever $c_1$ is sufficiently small and $a<b$ satisfy $b-a \leq c_1\sqrt{|E(G)|}$,
	\begin{equation}
    \label{eq: var_lower_bound_intermediate}
	\prob_w\Big(\W_{G}\cdot \frac{\varepsilon}{1-\varepsilon} \leq b-a\Big) \leq \frac{1}{4}.
	\end{equation}
	Combining \eqref{eq: fluc_W_G}, \eqref{eq: var_lower_bound_intermediate} (and recalling that $C_1\alpha \leq 1/2$), for $c_1$ sufficiently small and for all $a<b$ with $b-a\leq c_1\sqrt{|E(G)|}$, $\prob_w(a\leq \W_{G} \leq b) \leq 7/8$. This proves Lemma~\ref{lem: variance_lower_bound}.
\end{proof}
We will need several variance estimates for the MWM weight, depending on which sources of randomness are being averaged over: the edge weights alone, the random graph alone, or both the graph and the weights together.
\begin{lem}
	\label{lem: var_bounds}
	There exist constants $c, C>0$ depending only on $D$ such that the following statements hold for all $n \ge 1$.
	\begin{enumerate}
		\item[(a)] For any $G \in \mathcal{MG}_{n,D}$, $\Var_w(\W_G) \le Cn$.
		\item[(b)] $cn \le \E_{\C_n}\Var_w(\W_{\C_n}) \le Cn$.
		\item[(c)] $\Var_{\C_n}(\E_w \W_{\C_n}) \le Cn$.
		\item[(d)] $cn \le \Var(\W_{\C_n}) \le Cn$.
	\end{enumerate}
\end{lem}

\begin{proof}
	\begin{enumerate}

        \item[(a)] Applying Efron--Stein to the weights of the edge copies and using
$|\W_G-\W_G^{\{e\}}|\le |w_e-\widetilde w_e|$, we obtain
\[\Var_w(\W_G) \le \frac12\sum_{e\in\mathcal E(G)}
\E_w(w_e-\widetilde w_e)^2 =|\mathcal E(G)| \le \frac{Dn}{2}.
\]
 \item[(b)] The upper bound follows immediately from (a). For the lower bound, by Lemma~\ref{lem: variance_lower_bound} and the fact that $N\geq c_0n$ for some constant $c_0>0$ (which follows from \eqref{eq: degree_seq_cond}), one has
        \[ \E_{\C_n} \Var_w(\W_{\C_n}) \geq \E_{\C_n} [\Var_w(\W_{\C_n}) \ind_{\{\C_n \text{ is simple}\}}] \geq cc_0n \prob_{\C_n}(\C_n \text{ is simple}).
        \]
        The desired lower bound then follows from \cite[Theorem~7.12]{vdHofstadvol1} (or \cite[Eq.~(1.3)]{Janson}), which implies that
		\begin{equation}
			\label{eq: g_n_simple}
		\liminf_{n\to\infty} \prob_{\C_n}(\text{$\C_n$ is simple})>0.
		\end{equation}

        \item[(c)] Apply Lemma~\ref{cor: conf_var_bound} to the function $f(G) := \E_w \W_G$. For any $G \in \mathcal{MG}_{n,D}$ and any $e,e' \in \cE(G)$,
		$0 \le f(G)-f(G\setminus e) \le \E_w w_e = 1$,
		and similarly
		$0 \le f(G\setminus e)-f(G\setminus\{e,e'\}) \le \E_w w_{e'} = 1$.
		Therefore Lemma~\ref{cor: conf_var_bound} yields
		\[ \Var_{\C_n}(\E_w \W_{\C_n})
		\le C_0 \E_{\C_n}|\cE(\C_n)| + \frac{C_0}{N}\E_{\C_n}|\cE(\C_n)|^2.
		\]
		Since $|\cE(\C_n)| = N/2 \le Dn/2$, the right-hand side is at most $Cn$ for a suitable constant $C.$

\item[(d)] This follows from (b), (c), and the variance decomposition
		\[ \Var(\W_{\C_n}) = \Var_{\C_n}(\E_w \W_{\C_n}) + \E_{\C_n}\Var_w(\W_{\C_n}). \qedhere
		\]
	\end{enumerate}
\end{proof}

We also need to compare the variance of a function on $\G_n$ with that on $\C_n$.
\begin{lem}
	\label{lem: var_cn_gn}
	Let $f_n: \mathcal{MG}_{n, D} \to \mathbb{R}$ be  a function. Then there exists a constant $C>0$ such that for all $n \ge 1$,
	\[
	\Var_{\G_n}(f_n(\G_n))\leq C\Var_{\C_n}(f_n(\C_n)).
	\]
\end{lem}
\begin{proof}
	Recall that $\prob(\G_n\in \cdot) = \prob(\C_n\in \cdot \mid \C_n \text{ is simple})$. Hence
	\begin{align*}
		\Var_{\C_n}(f_n(\C_n)) &\geq \Var_{\C_n}(f_n(\C_n)\mid \C_n \text{ is simple}) \prob(\C_n \text{ is simple})\\
		&=\Var_{\G_n}(f_n(\G_n)) \prob(\C_n \text{ is simple}).
	\end{align*}
	Lemma~\ref{lem: var_cn_gn} then follows from \eqref{eq: g_n_simple}.
\end{proof}

We next use correlation decay to show that the quenched variance of the MWM weight is stable under the deletion of an edge with a tree-like neighborhood.
\begin{lem}
    \label{lem: new_lemma}
    Let $G$ be a simple graph with $n$ vertices and maximum degree bounded above by $D$. Fix $\ell\geq 1$. There exist constants $C>0$ and $\gamma \in (0, 1)$ depending only on $D$ such that for any edge $e$ such that $\bB_e^\ell(G)$ is a tree, we have
	\[	\left|\Var_w(\W_{G}) - \Var_w(\W_{G\setminus e })\right| \le \gamma^{\ell} \sqrt{n} + CD^\ell.\]
\end{lem}
\begin{proof}
Using $\W_{G} = \W_{G\setminus e} + \B(e, G)$, we can write 
	\[\Var_w(\W_G) - \Var_w(\W_{G\setminus e}) = 2\Cov_w(\B(e, G), \W_G) - \Var_w(\B(e, G)).
	\]
	Since $0\leq \B(e, G)\leq w_e$, one has $\Var_w(\B(e, G)) \le \E w_e^2 = 2$. Hence, it remains to bound $|\Cov_w(\B(e, G), \W_G)|$.

    Let $S$ be the  outer edge boundary of $\bB_e^\ell(G)$, equivalently $S$ is the set of edges of $G$ with exactly one endpoint in
$V(\bB_e^\ell(G))$. Note that $|S|$ and $ |V(\bB_e^\ell(G))|$ are both $O(D^\ell)$. By Lemma~\ref{lem: edge_bonus}, since $e$ has a tree neighborhood, we have
	\[\E_w(\B(e,G) - \B(e, \bB_e^\ell(G)))^2 = O((1-\delta)^\ell),\]
	uniformly in $G$ and $e$, where $\delta\in (0,1)$ depends only on $D$.

	We approximate $\W_G$ by $\W_{\bB_e^\ell(G)} + \W_{U_e^\ell}$, where $U_e^\ell:=G\setminus V(\bB_e^\ell(G)).$
	In other words, we approximate the MWM on $G$ by considering the MWM after setting all the edge weights on $S$ as $0$. Note that
	\[0\leq \W_G - (\W_{\bB_e^\ell(G)} + \W_{U_e^\ell}) \leq \sum_{f\in S} w_f.\]
	Therefore, uniformly in $G$ and $e$, one has
	\[ \E_w(\W_G - (\W_{\bB_e^\ell(G)} + \W_{U_e^\ell}))^2 \leq \E_w\Big(\sum_{f\in S} w_f\Big)^2 = O(D^{2\ell}).\]
	Hence Cauchy--Schwarz implies
	\begin{align*}
		&|\Cov_w(\B(e, G), \W_G) - \Cov_w(\B(e, \bB_e^\ell(G)), \W_{\bB_e^\ell(G)} + \W_{U_e^\ell})|\\
		\leq&\;  O((1-\delta)^{\ell/2})\sqrt{\Var_w(\W_G)} +\big | \Cov_w(\B(e, \bB_e^\ell(G)), \W_G - (\W_{\bB_e^\ell(G)} + \W_{U_e^\ell})) \big|\\
		\leq&\; O((1-\delta)^{\ell/2}\sqrt{n}) + O(D^\ell).
	\end{align*}
	On the other hand, by Lemma~\ref{lem: var_bounds}(a),
	\begin{align*}
		|\Cov_w(\B(e, \bB_e^\ell(G)), \W_{\bB_e^\ell(G)} + \W_{U_e^\ell})| &= |\Cov_w(\B(e, \bB_e^\ell(G)), \W_{\bB_e^\ell(G)})|\\
		&\leq \sqrt{\Var_w(\B(e, \bB_e^\ell(G)))\Var_w(\W_{\bB_e^\ell(G)})}
		=O(D^{\ell/2}).
	\end{align*}
    Combining, we have
    \[  |\Cov_w(\B(e, G), \W_G)| = O((1-\delta)^{\ell/2}\sqrt{n}) + O(D^\ell). \]
    This proves the lemma.
\end{proof}

\subsection{CLT for local graph statistics of configuration models}
A Gaussian CLT was established in \cite{BarbourRollin}   for the sum of functions indexed by vertices of the configuration model which only depend on a local neighborhood of the index vertices.  Before stating their result, we introduce some notations. 
Let $h$ be a non-negative function defined on the space of finite connected rooted multigraphs. Suppose that $h$ depends only on the $k$-neighborhood of the root. 
Define 
\begin{equation}\label{eq:local_stat_rep}
H_{n} = \sum_{v\in V(\C_n) } h(\bB_v^k(\C_n), v).   
\end{equation}
\begin{thm}[\cite{BarbourRollin}, Theorem~2.1, simplified]
	\label{thm: BarbourRollin}
	Let $\mathbf{d}$ be a degree sequence satisfying \eqref{eq: degree_seq_cond}, and let $h$ and $H_n$ be as above. Suppose that $k \leq a\log  n$ for some small constant $a>0$ depending on $D$ and $N\geq n$. Then there exists a constant $C>0$ depending on $D$ such that for $n$ sufficiently large,
	\[d_{\mathrm{W}}\Bigg(\mathcal{L} \Big ( \frac{H_{n} - \E_{\C_n} H_{n}}{\sqrt{\Var_{\C_n}(H_{n})}} \Big), \  \mathrm N(0, 1) \Bigg) \leq C D^{Ck} \Bigg(\frac{\|h\|_\infty^3  n}{\big(\Var_{\C_n}(H_{n})\big)^{3/2}} + \frac{\|h\|_\infty^2  n^{1/2}}{\Var_{\C_n}(H_{n})}\Bigg).
	\]
\end{thm}

\section{Proof of Theorem~\ref{thm: annealed_CLT}}

\subsection{Reduction to Conditions I, II, and III}  \label{subsec:outline}
We begin by justifying that \ref{cond:quenched_clt}, \ref{cond:variance_conc}, and \ref{cond:quenched_mean_clt}  imply Theorem~\ref{thm: annealed_CLT}. Let $Z\sim \mathrm N(0,1)$ be independent of everything else, and define
\[ U_n :=\alpha_n \sqrt{Q_n^{\mathrm s}}X(\G_n)+\beta_n Y_n^{\mathrm s},
\qquad V_n:=\alpha_n \sqrt{Q_n^{\mathrm s}} Z+\beta_n Y_n^{\mathrm s}. \]
We claim that
\begin{equation}\label{UV_close}
d_{\mathrm{KS}}(\mathcal{L} (U_n), \mathcal{L}(V_n))=o(1),
\end{equation}
and
\begin{equation}\label{V_normal}
  V_n\stackrel{d}{\to} \mathrm N(0,1).
\end{equation}
Together, \eqref{UV_close} and \eqref{V_normal} immediately imply that
\[ U_n = \frac{\W_{\G_n} - \E \W_{\G_n}}{\sqrt{\Var(\W_{\G_n})}} \stackrel{d}{\to} \mathrm N(0,1), \]
which yields Theorem~\ref{thm: annealed_CLT}. \\

\noindent\textbf{Proof of \eqref{UV_close}.}  
For any $t\in\mathbb{R}$,
\[ \big|\prob(U_n\le t)-\prob(V_n\le t)\big| \le \E_{\G_n} \Big[\big|\prob(U_n\le t\mid \G_n)-\prob(V_n\le t\mid \G_n)\big|\Big]. \]
Taking the supremum over $t\in\mathbb{R}$, we obtain
\[ d_{\mathrm{KS}}(\mathcal{L} (U_n), \mathcal{L}(V_n))
\le \E_{\G_n} \big[ d_{\mathrm{KS}}\big(\mathcal{L}(U_n\mid \G_n),\mathcal{L}(V_n\mid \G_n)\big)\big].\]

Now condition on $\G_n$. Since $Q_n^{\mathrm s}$ and $Y_n^{\mathrm s}$ are functions only of $\G_n$, they become deterministic under this conditioning. Hence, for some constants
$c_n:=\alpha_n \sqrt{Q_n^{\mathrm s}} \ge 0$ and $b_n:=\beta_n Y_n^{\mathrm s},$
we have
\[ U_n\mid \G_n \;\stackrel{d}{=}\; c_n X(\G_n)+b_n, \qquad
V_n\mid \G_n \;\stackrel{d}{=}\; c_n Z+b_n.\]
Since Kolmogorov distance is invariant under a common shift and a common nonzero scaling, it follows that
\[ d_{\mathrm{KS}}\big(\mathcal{L}(U_n\mid \G_n),\mathcal{L}(V_n\mid \G_n)\big)
= d_{\mathrm{KS}}\big(\mathcal{L}(X(\G_n)\mid \G_n),\mathcal{L}(Z)\big) \]
on the event $\{c_n>0\}$. On the event $\{c_n=0\}$, both conditional laws are equal to the point mass at $b_n$, so the left-hand side is $0$. Thus, in all cases,
\[ d_{\mathrm{KS}}\big(\mathcal{L}(U_n\mid \G_n),\mathcal{L}(V_n\mid \G_n)\big)
\le d_{\mathrm{KS}}\big(\mathcal{L}(X(\G_n)\mid \G_n),\mathcal{L}(Z)\big).\]

Consequently,
\[ d_{\mathrm{KS}}(\mathcal{L} (U_n), \mathcal{L}(V_n))
\le \E_{\G_n}\big[ d_{\mathrm{KS}}\big(\mathcal{L}(X(\G_n)\mid \G_n),\mathcal{L}(Z)\big) \big]  = o(1), \]
 by Condition I. This proves the claim \eqref{UV_close}.

\noindent\textbf{Proof of \eqref{V_normal}.}  Since $0\le \alpha_n,\beta_n\le 1$, the sequence $(\alpha_n,\beta_n)$ is contained in the compact set $[0,1]^2$. Let $(n_k)$ be any subsequence. By compactness, there exists a further subsequence, which we denote again by $(n_k)$, and numbers $\alpha,\beta\in[0,1]$ such that
$\alpha_{n_k}\to \alpha$ and $\beta_{n_k}\to \beta$.
Passing to the limit in $\alpha_{n_k}^2+\beta_{n_k}^2=1$, we obtain
$\alpha^2+\beta^2=1$. Since $Q_n^{\mathrm s}\to 1$ in probability by Condition II, it follows that
$\alpha_{n_k} \sqrt{Q_{n_k}^{\mathrm s}}\to \alpha$ in probability. If $\beta>0$, by Condition III,
\[Y_{n_k}^{\mathrm s}\stackrel{d}{\to} \xi, \quad \text{ where} \ \   \xi \sim \mathrm N(0,1). \]
Hence, by Slutsky's theorem, and using that $Z$ is independent of everything else, we have
\[ \big(\alpha_{n_k} \sqrt{Q_{n_k}^{\mathrm s}},\, \beta_{n_k},\, Y_{n_k}^{\mathrm s},\, Z\big) \stackrel{d}{\to} (\alpha,\beta,\xi,Z), \]
where $\xi$ and $Z$ are independent standard normal random variables.
An application of the continuous mapping theorem then gives
\[ V_{n_k} = \alpha_{n_k} \sqrt{Q_{n_k}^{\mathrm s}} Z+\beta_{n_k} Y_{n_k}^{\mathrm s}
\stackrel{d}{\to} \alpha Z+\beta \xi \sim \mathrm N(0,\alpha^2+\beta^2)=\mathrm N(0,1). \]
If $\beta=0,$ then $\beta_{n_k} Y_{n_k}^{\mathrm s} \stackrel{p}{\to} 0$ since the sequence $Y_n^{\mathrm s} $ is tight. Therefore, by Slutsky, $V_{n_k} \stackrel{d}{\to}  \mathrm{N}(0, 1).$ Thus every subsequence of $(V_n)$ has a further subsequence converging in distribution to $\mathrm N(0,1)$, which proves \eqref{V_normal}. \\

Thus it remains to show that  Conditions I, II, and III hold.

\subsection{Verification of Conditions I, II, and III}\label{sec:verification}
We need the following elementary bound whose proof is omitted. 
\begin{lem}
	\label{lem: exp}
	Let $X_1, \ldots, X_m$ be $\mathrm{Exp}(1)$ random variables and set  $Y = \max (X_1,\ldots, X_m)$. Then for any $p \ge 1$, 
  $\prob\left(Y > 2\log{m}\right) \leq m^{-1}$ and  $\E Y^p \le (2\log{m})^p + p!$.
\end{lem}

\noindent \textbf{Proof of Condition I.} Choose a sequence $1\ll \ell_n\ll \log n$, and for a multigraph $H$, let
$\Theta_k(H)$ denote the number of cycles of length at most $k$.
Since $\G_n$ has the law of $\C_n$ conditional on $\C_n$ being simple,
\eqref{eq: g_n_simple} implies that, for some constant $C>0$, $\E_{\G_n}\Theta_{\ell_n} (\G_n) \le C\E_{\C_n}\Theta_{\ell_n} (\C_n).$

A direct cycle-counting argument for the configuration model, as in the
calculation leading to \cite[Eq.~(2.14)]{BollobasBook}, gives
\[\E_{\C_n}\Theta_{\ell_n}(\C_n) \le
\sum_{k=1}^{\ell_n} \frac{1}{2k} \Big( \frac{\sum_{i=1}^n d_i(d_i-1)}   {2m -2\ell_n}\Big)^k\le C_D^{\ell_n}\]
for all sufficiently large $n$, where $m = |\cE(\C_n)|$ and $C_D$
depends only on $D$. Since $\ell_n=o(\log n)$, we have
$C_D^{\ell_n}=n^{o(1)}$. In particular, for all sufficiently large $n$, $\E_{\G_n} \Theta_{\ell_n}(\G_n)\le n^{1/8}$ and hence, by Markov's inequality,
\[ \prob_{\G_n}(\Theta_{\ell_n} (\G_n) >n^{1/4})\le n^{-1/8}. \]

Set $\mathcal{A}_n:=\{\Theta_{\ell_n}(\G_{n}) \le n^{1/4}\}.$
On $\mathcal{A}_n$, choose one edge from each cycle of length at most $\ell_n$
according to a fixed deterministic rule, let $F_n$ be the set of selected
edges, and define $\widehat{\G}_n:=\G_n\setminus F_n.$ Then $|E (\widehat{\G}_n) |\ge n/2- n^{1/4}$ and
$\operatorname{girth}(\widehat{\G}_n)>\ell_n$. In particular,
$\widehat{\G}_n\in\mathcal K_{n}^{\square \ell_n}$ for sufficiently large $n$ with $\rho=1/4$. Moreover,
\[0\le \W_{\G_n}-\W_{\widehat{\G}_n} \le\sum_{e\in F_n}w_e \le |F_n|\max_{e\in E(\G_n)} w_e.\]
Hence, by Lemma~\ref{lem: exp}, on $\mathcal{A}_n$,
\[\Var_w\bigl(\W_{\G_n}-\W_{\widehat{\G}_n}\bigr) \le C n^{1/2}(\log n)^2.\]
Recall that $\sigma^2_G = \Var_w(\W_{G})$. Since $|E(\G_n)|\ge n/2$, Lemma~\ref{lem: variance_lower_bound} gives
$\sigma_{\G_n}^2 \ge cn$. Consequently, uniformly on $\mathcal{A}_n$,
\begin{equation}
\label{eq:short_cycle_L2_error}
\frac{ \Var_w\bigl(\W_{\G_n}-\W_{\widehat{\G}_n}\bigr) }{ \Var_w(\W_{\G_n}) }
\le C n^{-1/2}(\log n)^2 =:q_n=o(1).
\end{equation}
Moreover,  
\[\Big|\frac{\sigma_{\widehat{\G}_n}}{\sigma_{\G_n}}-1\Big|\le
\frac{\sqrt{\Var_w(\W_{\G_n}-\W_{\widehat{\G}_n})}}{\sigma_{\G_n}}\le \sqrt{q_n}=o(1)\]
uniformly on $\mathcal{A}_n$. Here, we used the bound 
    \begin{equation}\label{eq:sd_triangle}
    \big| \sqrt{\Var(X)} - \sqrt{\Var(Y)} \big| \le \sqrt{\Var(X- Y)}
    \end{equation}
    for any random variables $X$ and $Y$ with finite second moments, which is a consequence of the triangle inequality.  Define $\widetilde X_n :=
X(\widehat{\G}_n) \sigma_{\widehat{\G}_n}/\sigma_{\G_n}.$
Then, conditionally on $\G_n$,
\[\E_w\big|X(\G_n)-\widetilde X_n\big|^2
=\frac{\Var_w(\W_{\G_n}-\W_{\widehat{\G}_n})}{\Var_w(\W_{\G_n})}\le q_n.\]
By \eqref{eq:CLT_locally_tree},
\[\delta_n:=\max_{G\in\mathcal K_{n}^{\square \ell_n}}
d_{\mathrm{KS}}\left(\mathcal L\bigl(X(G)\bigr),
\mathrm N(0,1)\right) \to 0.\]
Since
$\sigma_{\widehat{\G}_n}/\sigma_{\G_n}=1+o(1)$ uniformly on $\mathcal{A}_n$, it follows
that
\[d_{\mathrm{KS}}\big(\mathcal L\bigl(\widetilde X_n\mid\G_n\bigr), \mathrm N(0,1)\big)
\le \delta_n+o(1)\]
uniformly on $\mathcal{A}_n$.

For any random variables $U,V$ and any $\varepsilon>0$,
\[d_{\mathrm{KS}}\bigl(\mathcal L(U),\mathrm N(0,1)\bigr)\le
d_{\mathrm{KS}}\bigl(\mathcal L(V),\mathrm N(0,1)\bigr)
+ \prob(|U-V|>\varepsilon)+ \frac{\varepsilon}{\sqrt{2\pi}}.
\]
Applying this conditionally on $\G_n$ with
$U=X(\G_n)$, $V=\widetilde X_n$, and
$\varepsilon=q_n^{1/3}$, and using \eqref{eq:short_cycle_L2_error}, we obtain
\[
d_{\mathrm{KS}}\left( \mathcal L\bigl(X(\G_n)\mid\G_n\bigr), \mathrm N(0,1) \right)
\le \delta_n+o(1) \]
uniformly on $\mathcal{A}_n$. Since the Kolmogorov distance is at most $1$, by dominated convergence, 
\begin{align*}
&\E_{\G_n}\left[ d_{\mathrm{KS}}\left(
\mathcal L\bigl(X(\G_n)\mid\G_n\bigr), \mathrm N(0,1) \right) \right] \le
\prob_{\G_n}(\mathcal{A}_n^c)+\delta_n+o(1)=o(1).
\end{align*}
This proves Condition~I. \\

\noindent \textbf{Proof of Condition II.}  The proof of Condition II is deferred to Subsection~\ref{subsec:conc_quenched_var}, where we establish the following result, showing that the quenched variance $\Var_w(\W_{\G_n})$ concentrates around its mean.
\begin{prop}
	\label{lem: variance_conc}
	One has
	\[ Q_n^{\mathrm s}  \stackrel{p}{\rightarrow} 1 \quad \text{ as $n\to\infty$.} \]
\end{prop}

\noindent \textbf{Proof of Condition III.} We will verify Condition III in two steps. First step, we reduce the CLT from $\G_n$ to $\C_n$. 

By Lemma~\ref{lem: variance_lower_bound}, 
\begin{align*}
\Var(\W_{\G_n}) \ge \E_{\G_n} \Var_w(\W_{\G_n}) \ge c   \E_{\G_n} |E(\G_n)| \ge cn/2. 
\end{align*}
So, if  $(n_k)$ is a subsequence such that $\liminf_k \beta_{n_k} >0, $ then  there exists a constant $c_0>0$ such that for all $k \ge 1$,
\begin{equation}\label{eq:var_lb_1}
    \Var_{\G_{n_k}}(\E_w \W_{\G_{n_k}}) \ge c_0 n_k.
\end{equation}
\begin{lem}\label{lem:G_to_C-CLT}
Let $(n_k)$ be a subsequence such that \eqref{eq:var_lb_1} holds.
 Suppose that 
\[ Y_{n_k}^{\mathrm c}:= \frac{\E_w \W_{\C_{n_k}} - \E \W_{\C_{n_k}}}{\sqrt{\Var_{\C_{n_k}}(\E_w \W_{\C_{n_k}})}}   \stackrel{d}{\to} \mathrm N(0,1) \text{ as }  \  k\to\infty. \]
Then  $Y_{n_k}^{\mathrm s} \stackrel{d}{\to} \mathrm N(0,1).$
\end{lem}
Lemma~\ref{lem:G_to_C-CLT}  will be proved in Subsection~\ref{subsec:G_to_C-CLT}. 

Second, we will prove a CLT for $Y_{n_k}^{\mathrm c}$, which will imply Condition~III and complete the proof of Theorem~\ref{thm: annealed_CLT}.
\begin{prop} \label{lem:C-CLT}
 Let $(n_k)$ be a subsequence such that \eqref{eq:var_lb_1} holds.  Then 
  \[ Y_{n_k}^{\mathrm c} \stackrel{d}{\to} \mathrm N(0,1) \text{ as $k\to\infty$.} \]
\end{prop}

We can represent $Y_n^{\mathrm c}$ as a graph statistics of $\C_n$ as in \eqref{eq:local_stat_rep}. We would like to invoke Theorem~\ref{thm: BarbourRollin} to prove a central limit theorem for $Y_n^{\mathrm c}$. However, we cannot directly apply Theorem~\ref{thm: BarbourRollin}, because $\E_w \W_{\C_n}$ is not a sum of functions that depend only on local neighborhoods in $\C_n$. Yet, thanks to the correlation decay, we will be able to approximate $\E_w \W_{\C_n}$ by such a sum, which allows us to 
apply Theorem~\ref{thm: BarbourRollin}. The details are carried in Subsection~\ref{sec: final_step}.

\subsection{Maximum weight matching on the pruned configuration model}

\begin{df}[$\ell$-good]\label{df:bad_edge} 
Let $G=(V(G),\cE(G),\partial_G)$ be a multigraph and let $\ell\ge1$. An edge copy $e\in\cE(G)$ is called $\ell$-good if $\bB_e^\ell(G)$ is a simple tree. Define 
\[ \Ggood{G}{\ell}  := \big( V(G), \left\{e\in\cE(G):e\text{ is $\ell$-good}\right\}, \left.\partial_G\right|_{\{e\in\cE(G):\,e\text{ is $\ell$-good}\}} \big). \] 
Thus, $\Ggood{G}{\ell} $ is the spanning submultigraph of $G$ obtained by deleting all edges that are not $\ell$-good and  $\cE(\Ggood{G}{\ell} )$ is the set of $\ell$-good edge copies of $G$.  In particular, $G^{[\ell]}$ is simple, and $\bB_e^\ell(G^{[\ell]})$ is a tree for every $e\in\cE(\Ggood{G}{\ell})$. The edge weights on $\Ggood{G}{\ell} $ are inherited from $G$. 
\end{df}

The MWM weight on the configuration model $\C_n$ and that on its pruned version $\Ggood{\C_n}{\ell}$ satisfy
    \begin{equation}\label{eq:mwm_diff_bd_1}
        0 \leq \W_{\C_n} - \W_{\Ggood{\C_n}{\ell}} \leq \sum_{e\in \cE(\C_n) \setminus \cE(\Ggood{\C_n}{\ell})} w_e.
    \end{equation}

  From (the proof of) \cite[Lemma~4.2]{vdHofstad} (see also Remark~4.3 there), one has
\begin{equation}
    \label{eq: L_n_size}
	\E_{\C_n}  \bigl|\cE(\C_n) \setminus \cE(\Ggood{\C_n}{\ell})\bigr| = O(D^{2\ell}), \qquad \E_{\C_n}  \bigl| \cE(\C_n) \setminus \cE(\Ggood{\C_n}{\ell})\bigr|^2 = O(D^{4\ell}).
	\end{equation}

\begin{lem}\label{lem:delete_e_tree}
Let $G\in\mathcal{MG}_n$, let $\ell\ge 1$, and let
$e\in \cE(\Ggood{G}{\ell})$. Then
\[\Ggood{(G\setminus e)}{\ell}= \Ggood{G}{\ell}\setminus e.\]
\end{lem}

\begin{proof}
Set $H:=G\setminus e$. If
$f\in \cE(\Ggood{G}{\ell}) \setminus\{e\}$, then
$\bB_f^\ell(H)$ is a connected subgraph of the simple tree
$\bB_f^\ell(G)$. Hence
\[ \cE(\Ggood{G}{\ell}) \setminus\{e\}\subseteq \cE(\Ggood{H}{\ell}). \]

For the reverse inclusion, write $\partial_G(e)=\{u,v\}$. Since
$\bB_e^\ell(G)$ is a simple tree, $d_H(u,v)>2\ell+1$, 
otherwise, a shortest path from $u$ to $v$ in $H$, together with $e$,
would form a cycle in $\bB_e^\ell(G)$.

Let $f\in  \cE(\Ggood{H}{\ell})$. If
$\bB_f^\ell(G)=\bB_f^\ell(H)$, then $f\in \cE(\Ggood{G}{\ell})$. Otherwise,
after interchanging $u$ and $v$ if necessary, set
\[ a:=d_H(u,f)\le\ell-1, \qquad r:=\ell-a-1. \]
The preceding distance bound implies $d_H(v,f)>\ell$. Every shortest
path from $f$ to a vertex of
$V(\bB_f^\ell(G))\setminus V(\bB_f^\ell(H))$ must therefore traverse
$e$ from $u$ to $v$, and consequently, 
\[ V\bigl(\bB_f^\ell(G)\bigr) = V\bigl(\bB_f^\ell(H)\bigr) \cup V\bigl(\bB_v^r(H)\bigr). \]
These two vertex sets are disjoint, since a common vertex would give $d_H(u,v)\le a+1+\ell+r=2\ell.$
Moreover, no edge of $H$ joins them, since such an edge would give
\[d_H(u,v)\le a+1+\ell+1+r=2\ell+1. \]
Now $\bB_f^\ell(H)$ is a simple tree, while $\bB_v^r(H)$ is a connected
subgraph of the simple tree $\bB_e^\ell(G)$. Thus
$\bB_f^\ell(G)$ is obtained by joining these two simple trees by the
single edge $e$, and hence is itself a simple tree. Therefore
$f\in  \cE(\Ggood{G}{\ell})$.

It follows that $\cE(\Ggood{H}{\ell})=\cE(\Ggood{G}{\ell}) \setminus\{e\}$
and hence $\Ggood{(G\setminus e)}{\ell}= \Ggood{G}{\ell}\setminus e$.
\end{proof}

\begin{lem}\label{eq: expected_var_diff}
For any $\ell \ge 1$,  
    	\[\E_{\C_n} \big|\Var_w(\W_{\C_n}) - \Var_w(\W_{\Ggood{\C_n}{\ell}})\big| = O\big(D^{2\ell}\sqrt{n}\log{n}\big).\]
\end{lem}

\begin{proof}
    By \eqref{eq:mwm_diff_bd_1} and Lemma~\ref{lem: exp}, we have
	\[\Var_w(\W_{\C_n} - \W_{\Ggood{\C_n}{\ell}}) \leq \E_w\Big(\sum_{e\in \cE(\C_n) \setminus \cE(\Ggood{\C_n}{\ell})} w_e\Big)^2 \leq |\cE(\C_n) \setminus \cE(\Ggood{\C_n}{\ell})|^2 \cdot O((\log{n})^2).	\]
	Hence, by \eqref{eq:sd_triangle},
	\[\Big|\sqrt{\Var_w(\W_{\C_n})} - \sqrt{\Var_w(\W_{\Ggood{\C_n}{\ell}})}\Big| \leq \sqrt{\Var_w\big(\W_{\C_n} - \W_{\Ggood{\C_n}{\ell}}\big)} \leq |\cE(\C_n) \setminus \cE(\Ggood{\C_n}{\ell})| O(\log{n}).\]
    Thus
	\begin{align*}
	\big|\Var_w(\W_{\C_n}) - \Var_w(\W_{\Ggood{\C_n}{\ell}})\big| &\leq  \Big(\sqrt{\Var_w(\W_{\C_n})} + \sqrt{\Var_w(\W_{\Ggood{\C_n}{\ell}})}\Big)|\cE(\C_n) \setminus \cE(\Ggood{\C_n}{\ell})| O(\log{n}) \\
 &\le |\cE(\C_n) \setminus \cE(\Ggood{\C_n}{\ell})|O(\sqrt{n}\log{n})
	\end{align*}
	where in the last inequality, we used the fact that there exists a nonrandom constant $C>0$ such that almost surely
	\[ \Var_w(\W_{\C_n}) \leq Cn, \quad  \Var_w(\W_{\Ggood{\C_n}{\ell}}) \leq Cn,\]
as a result of Lemma~\ref{lem: var_bounds}(a). The lemma now follows from \eqref{eq: L_n_size}.
\end{proof}

\subsection{Concentration of the quenched variance} \label{subsec:conc_quenched_var}
 This subsection is devoted to the proof of Proposition~\ref{lem: variance_conc}. The proof will be divided into several steps. We first claim that it suffices to show the statement if we replace $\G_n$ by $\C_n$, i.e., 
    \begin{equation} \label{eq: variance_conc_conf}
	\frac{\Var_w(\W_{\C_n})}{\E_{\C_n} \Var_w(\W_{\C_n})} \stackrel{p}{\rightarrow} 1 \text{ as $n\to\infty$.}
	\end{equation}
    By Lemma~\ref{lem: var_bounds}(a) and (b), the ratio in \eqref{eq: variance_conc_conf} is bounded above. Hence, by the bounded convergence theorem,  \eqref{eq: variance_conc_conf} implies that
    \[ \Var_{\C_n}\Big(\frac{\Var_w(\W_{\C_n})}{\E_{\C_n}  \Var_w(\W_{\C_n})}\Big) \to 0\text{ as $n\to\infty$.}\]
    Using Lemma~\ref{lem: var_bounds}(b) again, we then obtain
    \[ \Var_{\C_n}(\Var_w(\W_{\C_n})) = o(n^2). \]
    Now Lemma~\ref{lem: var_cn_gn}, applied to the function $f(G) :=\Var_w(\W_{G}) $ for a multigraph $G$,  implies that
    \begin{equation} \label{eq: var_var}
	\Var_{\G_n}(\Var_w(\W_{\G_n})) = o(n^2).
	\end{equation}
	Finally, note that for any $\varepsilon>0$, by Chebyshev's inequality,
	\[\prob_{\G_n} (|\Var_w(\W_{\G_n}) - \E_{\G_n} \Var_w(\W_{\G_n})| > \varepsilon \E_{\G_n} \Var_w(\W_{\G_n})) \leq \frac{\Var_{\G_n}(\Var_w(\W_{\G_n}))}{\varepsilon^2(\E_{\G_n} \Var_w(\W_{\G_n}))^2}.\]
By \eqref{eq: var_var} and Lemma~\ref{lem: variance_lower_bound}, the right-hand side tends to $0$ as $n\to\infty$, implying  Proposition~\ref{lem: variance_conc}. Thus, it remains to show \eqref{eq: variance_conc_conf}.

Let $R$ be a positive integer satisfying $ 1 \ll R \ll \log n$. Recall that  (see Definition~\ref{df:bad_edge}) $\Ggood{\C_n}{R}$ is the simple graph obtained from $\C_n$ by removing all edges that are not $R$-good. Next, we will argue that to show \eqref{eq: variance_conc_conf}, it suffices to prove the following lemma. 
 \begin{lem}\label{lem: var_conc}
    For $1\ll R\ll \log{n}$, one has
    	\begin{equation} \label{eq: var_conc}
		\Var_{\C_n}(\Var_w(\W_{\Ggood{\C_n}{R}})) = o(n^2).
	\end{equation}
  \end{lem}
Let's argue why \eqref{eq: var_conc} implies \eqref{eq: variance_conc_conf}. Note that
it follows from Lemma~\ref{eq: expected_var_diff} and the fact that $R  = o(\log n)$ that for any fixed $\varepsilon>0$
\begin{equation}\label{eq:var_approx_2R}
    \E_{\C_n} \big|\Var_w(\W_{\C_n}) - \Var_w(\W_{\Ggood{\C_n}{R}})\big| = O_\varepsilon(n^{1/2+ \varepsilon}).
\end{equation}
 Further, by Lemma~\ref{lem: variance_lower_bound} and \eqref{eq: L_n_size}, there exists $c>0$ depending only on $D$ such that $\E_{\C_n} \Var_w(\W_{\Ggood{\C_n}{R}})  \geq cn$. This, together with \eqref{eq: var_conc} and Chebyshev's inequality, implies that
	\begin{equation}
		\label{eq: var_w_tilde_conv}
		\frac{\Var_w(\W_{\Ggood{\C_n}{R}})}{\E_{\C_n} \Var_w(\W_{\Ggood{\C_n}{R}})} \stackrel{p}{\rightarrow} 1.
	\end{equation}
We will use the following elementary fact:

If $(S_n)$, $(T_n)$ are two sequences of random variables with finite expectations such that (a) $S_n/ \E S_n \stackrel{p}{\rightarrow} 1,$ (b)  there exists $c>0$ such that  $\E T_n\geq cn$ for all $n$, and (c) $(S_n - T_n)/n \stackrel{L^1}{\rightarrow} 0$, then $ T_n /\E T_n \stackrel{p}{\rightarrow} 1.$

We apply this with  $S_n = \Var_w(\W_{\Ggood{\C_n}{R}})$ and $T_n = \Var_w(\W_{\C_n})$.
Then \eqref{eq: var_w_tilde_conv}, \eqref{eq:var_approx_2R}, and
Lemma~\ref{lem: var_bounds}(b), which gives $\E_{\C_n}T_n = \E_{\C_n}\Var_w(\W_{\C_n})
\ge cn,$
imply \eqref{eq: variance_conc_conf}.

  \begin{proof}[Proof of Lemma~\ref{lem: var_conc}] By an application
	of Lemma~\ref{cor: conf_var_bound} to the function $f(G) = \Var_w(\W_{\Ggood{G}{R}}) $ for $G \in \mathcal{MG}_n$,  it suffices to show that
	\begin{equation} \label{eq: var_diff_bound}
        \E_{\C_n} \sum_{e\in \cE(\C_n)}\big(\Var_w(\W_{\Ggood{\C_n}{R}}) - \Var_w(\W_{\Ggood{(\C_n \setminus e)}{R}})\big)^2 = o(n^2)
    \end{equation}
    and
    \begin{equation} \label{eq: var_diff_bound_2}
        \E_{\C_n} \sum_{e, e'\in \cE(\C_n), e\neq e'} \big(\Var_w(\W_{\Ggood{(\C_n \setminus e)}{R}}) - \Var_w(\W_{\Ggood{(\C_n \setminus \{e, e'\})}{R}})\big)^2 = o(n^3).
    \end{equation}
First consider \eqref{eq: var_diff_bound}.  We decompose the sum into two corresponding to  $e \in \cE(\Ggood{\C_n}{R})$ and $e  \not \in \cE(\Ggood{\C_n}{R})$. 
	We bound the second term as follows.
    \begin{align*}
        & \E_{\C_n}\sum_{e \not\in \cE(\Ggood{\C_n}{R})}\big(\Var_w(\W_{\Ggood{\C_n}{R}}) - \Var_w(\W_{\Ggood{(\C_n \setminus e)}{R}})\big)^2 \\
        \leq &\; \E_{\C_n}\Big[ \big|\cE(\C_n) \setminus \cE(\Ggood{\C_n}{R})\big| \cdot \max_{e \in \cE(\C_n) \setminus \cE(\Ggood{\C_n}{R})}\big(\Var_w(\W_{\Ggood{\C_n}{R}}) - \Var_w(\W_{\Ggood{(\C_n \setminus e)}{R}}) \big)^2\Big].
    \end{align*}
    Fix any $G\in \mathcal{MG}_{n,D}$ and $ e\in \cE(G)$. 
    Note that if we remove an edge copy $e$ from $G$, it is possible to turn a not $R$-good  edge copy in $\bB_e^{R}(G)$  into a good one, and thus $|\W_{\Ggood{G}{R}} - \W_{\Ggood{(G \setminus e)}{R}}|$ is bounded above by the sum of edge weights inside $\bB_e^{R}(G)$, which, in turn, can be bounded above by $\max_{f\in \cE(G)} w_f   \cdot O( D^{R}) $. Therefore, by Lemma~\ref{lem: exp} and \eqref{eq:sd_triangle}, 
	\begin{align*}
	\big|\sqrt{\Var_w(\W_{\Ggood{G}{R}})} - \sqrt{\Var_w(\W_{\Ggood{(G \setminus e)}{R}})}\big| &\le \sqrt{\Var_w \big(\W_{\Ggood{G}{R}} - \W_{\Ggood{(G \setminus e)}{R}}\big)} \\
    &\le  O( D^{R}) \big( \E_w \max_{f \in \cE(G)} w_f^2 \big)^{1/2}  = O( D^{R}\log{n}).
	\end{align*}
 Combining with Lemma~\ref{lem: var_bounds}(a), we see that
	\[\big|\Var_w(\W_{\Ggood{G}{R}}) - \Var_w(\W_{\Ggood{(G \setminus e)}{R}})\big| = O( D^{R}\sqrt{n} \log{n}).\]
 The above bound is uniform over $G \in \mathcal{MG}_{n,D}$ and $e \in \cE(G)$. Consequently, 
	\begin{equation*}
	\max_{e \in \cE(\C_n)} \big|\Var_w(\W_{\Ggood{\C_n}{R}}) - \Var_w(\W_{\Ggood{(\C_n \setminus e)}{R}})\big| \leq C_1D^{R}\sqrt{n}\log{n}.
	\end{equation*}
    Together with \eqref{eq: L_n_size}, noting $R \ll \log{n}$, we have, for any fixed $\varepsilon >0$, 
	\begin{equation}
	    \label{eq: NT_bound}
        \E_{\C_n}\sum_{e  \not \in \cE(\Ggood{\C_n}{R})}\big(\Var_w(\W_{\Ggood{\C_n}{R}}) - \Var_w(\W_{\Ggood{(\C_n \setminus e)}{R}})\big)^2  =  O( D^{R}\sqrt{n}\log{n})^2 \cdot O(D^{2R}) = O_\varepsilon(n^{1 + \varepsilon}).
	\end{equation}

    Next, we handle the term  involving $e \in \cE(\Ggood{\C_n}{R})$. By Lemma~\ref{lem:delete_e_tree}, $\Ggood{(\C_n \setminus e)}{R}= \Ggood{\C_n}{R}\setminus e$.
  Applying Lemma~\ref{lem: new_lemma} to the simple graph $\Ggood{\C_n}{R} \in \mathcal{G}_{n, D}$ and using $1\ll R\ll \log{n}$, we obtain
        \begin{align*}
\max_{e \in \cE(\Ggood{\C_n}{R})} \big|\Var_w(\W_{\Ggood{\C_n}{R}}) - \Var_w(\W_{\Ggood{(\C_n \setminus e)}{R}})\big| = o(\sqrt{n}).
    \end{align*}
    Therefore, $\E_{\C_n}\sum_{e \in \cE(\Ggood{\C_n}{R})}\big(\Var_w(\W_{\Ggood{\C_n}{R}}) - \Var_w(\W_{\Ggood{(\C_n \setminus e)}{R}})\big)^2$ can be bounded above by 
    \begin{equation} \label{eq: T_bound}
        \begin{split}
         N\E_{\C_n} \max_{e \in \cE(\Ggood{\C_n}{R})}\big(\Var_w(\W_{\Ggood{\C_n}{R}}) - \Var_w(\W_{\Ggood{(\C_n \setminus e)}{R}})\big)^2 = o(n^2).
        \end{split}
    \end{equation}
    Combining \eqref{eq: NT_bound}, \eqref{eq: T_bound},  we have
    \begin{align*}
    \E_{\C_n}\sum_{e \in \cE(\C_n)}\big(\Var_w(\W_{\Ggood{\C_n}{R}}) - \Var_w(\W_{\Ggood{(\C_n \setminus e)}{R}})\big)^2 = o(n^2).
    \end{align*}
    The estimate \eqref{eq: var_diff_bound_2} can be shown similarly, and we will omit the details. This completes the proof of the lemma. 
    \end{proof}

\subsection{Transfer of the CLT from $\C_n$ to $\G_n$}\label{subsec:G_to_C-CLT}

In this section, we prove Lemma~\ref{lem:G_to_C-CLT}. To simplify notation, we write $(n)$ in place of the subsequence $(n_k)$.  We first describe the result in \cite{Janson} that allows us to `transfer' the central limit theorem for $\C_n$ to that for the configuration model $\G_n$.

Consider the configuration model $\C_n=G_0:=\chi(\bs_0),$
where $\bs_0$ is uniform on $\mathcal P_{\mathbb H}$. We remove loops and
parallel edges by performing a sequence of switches. At step $i$, suppose that
$G_i:=\chi(\bs_i)$ is not simple. Choose the first half-edge $h_x$, according
to the fixed lexicographic ordering of $\mathbb H$, such that the edge copy
corresponding to $\{h_x,\bs_i(h_x)\}$
is either a loop or belongs to a family of parallel edge copies in $G_i$.
Choose $h_y$ uniformly from
$\mathbb H\setminus\{h_x,\bs_i(h_x)\}$, and set
\[ \bs_{i+1}:=\sw^{h_x,h_y}(\bs_i),
\qquad G_{i+1}:=\chi(\bs_{i+1}). \]
This switch removes the selected bad edge copy, although it may create new bad
edge copies. We repeat the procedure until the first time $\vartheta_n$ for which $G_{\vartheta_n}$
is simple, thereby obtaining the sequence $G_0,G_1,\ldots,G_{\vartheta_n}$.

It is shown in \cite{Sjostrand} that this procedure terminates almost surely
whenever there exists a simple graph with degree sequence $\mathbf d$.
Moreover, under \eqref{eq: degree_seq_cond},  \cite[Theorem~3.2]{Janson} shows that, with
high probability, no new bad edge copies are created and
$\lim_{K \to \infty} \sup_n \prob_{\C_n}(\vartheta_n \geq K) = 0$. Following Janson, we call the resulting simple graph the
\emph{switched configuration model} and write
\[ \Gamma(\C_n):=G_{\vartheta_n}. \]
The law of $\Gamma(\C_n)$ need not coincide exactly with that of $\G_n$.
Nevertheless, under \eqref{eq: degree_seq_cond}, \cite{Janson} shows that the
two distributions are asymptotically close in a sense sufficient to transfer a
central limit theorem from $\C_n$ to $\G_n$, as stated below.

\begin{lem}[\cite{Janson}, Corollary~2.3, simplified]\label{lem: Janson_cor}
	Suppose that the degree sequence $\mathbf{d}$ satisfies \eqref{eq: degree_seq_cond}. Let $f_n: \mathcal{MG}_n \to \mathbb{R}$ be such that
	\[ f_n(\C_n) \stackrel{d}{\to} \mathrm N(0,1), \qquad f_n(\Gamma(\C_n)) - f_n(\C_n) \stackrel{p}{\rightarrow} 0, \quad \text{as } n \to \infty.\]
	 Then as $n \to \infty,$
	\[ f_n(\G_n) \stackrel{d}{\to} \mathrm N(0,1).\]
\end{lem}
We will apply Lemma~\ref{lem: Janson_cor} with 
\[ f_n(G) = (\E_w \W_G  -  \E \W_{\C_n})/ \sqrt{ \Var_{\C_n} (\E_w \W_{\C_n} ) }, \quad G \in \mathcal{MG}_n.\]
Note that each switch modifies at most four edge copies in the multigraph, in the sense that it deletes at most two edge copies and creates at most two new ones.  Since the edge weights are independent of $\C_n$ and of the switching procedure, changing one edge copy changes the quantity $\E_w \W$ by at most $\E w_e = 1$. Therefore, $|\E_w\W_{\C_n} - \E_w\W_{\Gamma(\C_n)}|
\le 4\vartheta_n$.
 On the other hand,  by Lemma~\ref{lem: var_cn_gn} and assumption \eqref{eq:var_lb_1}, there exists a positive constant $c$ such that 
\begin{equation} \label{eq:var_lb_2}
    \Var_{\C_n}(\E_w \W_{\C_n}) \ge c n.
\end{equation}
Therefore,
\[ f_n (\C_n)   - f_n (\Gamma(\C_n)) =  \big( \E_w \W_{\C_n} - \E_w \W_{\Gamma(\C_n)} \big)/\sqrt{\Var_{\C_n}(\E_w \W_{\C_n})} \stackrel{p}{\to} 0. \]
On the other hand, by hypothesis, $f_n (\C_n) \stackrel{d}{\to} \mathrm N(0, 1)$.
Hence,   Lemma~\ref{lem: Janson_cor} yields 
\begin{equation}\label{eq: CLT_Z_n}
	f_n(\G_n) = \frac{\E_w\W_{\G_n} - \E \W_{\C_n}}{\sqrt{\Var_{\C_n}(\E_w \W_{\C_n})}} \stackrel{d}{\to} \mathrm N(0,1).
\end{equation}
Note that $\E f_n(\C_n)=0$ and $\Var(f_n(\C_n))=1$. Also, $f_n(\C_n)\stackrel{d}{\to} \mathrm N(0,1)$. By \cite[Theorem~5.5.9]{Gut}, the sequence $(f_n(\C_n)^2)_{n\geq 1}$ is uniformly integrable. As the law of $f_n(\G_n)$ is the conditional law of $f_n(\C_n)$ given that $\C_n$ is simple, it follows from  \eqref{eq: g_n_simple} that $(f_n(\G_n)^2)_{n\geq 1}$ is also uniformly integrable. Since $f_n(\G_n)\stackrel{d}{\to} \mathrm N(0,1)$ by \eqref{eq: CLT_Z_n}, another application of \cite[Theorem~5.5.9]{Gut} gives
\[ \E f_n(\G_n) \to 0, \quad \Var(f_n(\G_n)) \to 1 \quad \text{as } n\to\infty. \]
This is equivalent to saying
\[ \frac{\E\W_{\G_n} - \E \W_{\C_n}}{\sqrt{\Var_{\C_n}(\E_w \W_{\C_n})}} \to 0, \quad \frac{\Var_{\G_n}(\E_w\W_{\G_n})}{\Var_{\C_n}(\E_w\W_{\C_n})} \to 1\quad \text{as $n\to\infty$.}\]
These, together with \eqref{eq: CLT_Z_n}, imply $Y_n^{\mathrm s}  \stackrel{d}{\to} \mathrm N(0,1).$
This completes the proof of Lemma~\ref{lem:G_to_C-CLT}. \qed

\subsection{CLT for quenched mean} 
\label{sec: final_step}
In this section, we will prove Proposition~\ref{lem:C-CLT}.  To lighten the notation, we will again use the full sequence  $(n)$ in place of the subsequence $(n_k)$. As before, let $R$ be a positive integer satisfying $ 1 \ll R \ll \log n$. Let $A$ be a large positive integer chosen according to Theorem~\ref{thm:local_perturbative_bound} and set $L = AR.$

We obtain the simple graph $\Ggood{\C_n}{L}$ from $\C_n$ by retaining only the $L$-good edge copies, recall Definition~\ref{df:bad_edge}.  Below, we argue that 
it suffices to prove the CLT for  $\E_w \W_{\Ggood{\C_n}{L}}$.

By \eqref{eq: L_n_size}, we have 
\begin{equation} \label{eq: var_remove_bad_edges} 
 \Var_{\C_n} \big(\E_w \W_{\C_n} - \E_w \W_{\Ggood{\C_n}{L}} \big ) \leq \E_{\C_n} \big|\cE(\C_n) \setminus \cE(\Ggood{\C_n}{L}) \big|^2 = O(D^{4L}).   
\end{equation}
Recall from \eqref{eq:var_lb_2}, we have $\Var_{\C_n}(\E_w \W_{\C_n}) \ge c n$.
Now since $R \ll \log{n}$,  \eqref{eq: var_remove_bad_edges}  together with this variance lower bound imply that 
\[  \Var_{\C_n}\big(\E_w \W_{\C_n} - \E_w \W_{\Ggood{\C_n}{L}}\big) = o(n) = o(\Var_{\C_n}(\E_w \W_{\C_n}))\]
and consequently, as $n \to \infty$, 
\[ \frac{\E_w \W_{\C_n} - \E \W_{\C_n}}{\sqrt{\Var_{\C_n}(\E_w \W_{\C_n})}} - 
\frac{\E_w \W_{\Ggood{\C_n}{L}} - \E \W_{\Ggood{\C_n}{L}} }{\sqrt{\Var_{\C_n}(\E_w \W_{\C_n})}} \stackrel{p}{\to} 0. \]
On the other hand, using \eqref{eq:sd_triangle}, 
\begin{equation} \label{eq: var_comparison_bad}
\Bigg|\sqrt{\frac{\Var_{\C_n}(\E_w \W_{\Ggood{\C_n}{L}})}{\Var_{\C_n}(\E_w \W_{\C_n})}} - 1\Bigg| \leq \sqrt{\frac{\Var_{\C_n}(\E_w \W_{\C_n} - \E_w \W_{\Ggood{\C_n}{L}})}{\Var_{\C_n}\big(\E_w \W_{\C_n}\big)}} = o(1).
\end{equation}
Hence if 
\begin{equation}\label{eq:conv_R_approx}
   \frac{\E_w \W_{\Ggood{\C_n}{L}} - \E \W_{\Ggood{\C_n}{L}}}{\sqrt{\Var_{\C_n}(\E_w \W_{\Ggood{\C_n}{L}})}} \stackrel{d}{\to} \mathrm N(0,1),
\end{equation}
then so does $(\E_w \W_{\C_n} - \E \W_{\C_n}) / \sqrt{\Var_{\C_n}(\E_w \W_{\C_n})}$ as claimed in Proposition~\ref{lem:C-CLT}.

It remains to prove \eqref{eq:conv_R_approx}. To this end, we further approximate the quenched mean  $\E_w \W_{\Ggood{\C_n}{L}}$  by the local statistic  $\sZ_R(\Ggood{\C_n}{L})$, where   for a finite simple graph $H$, we define (recall \eqref{eq:MWM_local_H})
\[ \sZ_R(H):= \E_w \W_H^{\mathrm{loc}, R} = \sum_{f\in \cE(H)} \E_w\big[
w_f\,\ind_{\{f\in \M_{\bB_f^R(H)}\}}\big].\]

The following result justifies the effectiveness of this approximation.
\begin{prop}\label{eq: var_k_approx}
  Let $ 1 \ll R \ll \log n$. Then  there exists an integer $A \ge 1$ large enough such that as $n \to \infty$, 
 \begin{equation*}
	\Var_{\C_n}\big(\E_w \W_{\Ggood{\C_n}{L}} -  \sZ_R(\Ggood{\C_n}{L})\big) = o(n). 
\end{equation*}
\end{prop}
Assuming the above bound, let us proceed to the rest of the proof of Proposition~\ref{lem:C-CLT}. 
From \eqref{eq: var_comparison_bad} and \eqref{eq:var_lb_2}, it follows that 
\begin{equation} \label{var_Gn_ub_lb_linear}
  \Var_{\C_n} (\E_w \W_{\Ggood{\C_n}{L}}) \ge cn,
\end{equation}
for some positive constant $c$.
By Proposition~\ref{eq: var_k_approx} and \eqref{var_Gn_ub_lb_linear}, we deduce that
\begin{equation}
	\label{eq: k_approx_close}
\frac{\E_w \W_{\Ggood{\C_n}{L}} - \E \W_{\Ggood{\C_n}{L}}}{\sqrt{\Var_{\C_n} (\E_w \W_{\Ggood{\C_n}{L}}) }} - \frac{\sZ_R(\Ggood{\C_n}{L})- \E_{\C_n}  \sZ_R(\Ggood{\C_n}{L}) }{\sqrt{\Var_{\C_n} (\E_w \W_{\Ggood{\C_n}{L}}) }} \stackrel{p}{\rightarrow} 0 \text{ as $n\to\infty$.}
\end{equation}
We can write
\begin{align}
	\label{eq: R_appro_decomposition}
	\frac{\sZ_R(\Ggood{\C_n}{L}) - \E_{\C_n}  \sZ_R(\Ggood{\C_n}{L}) }{\sqrt{\Var_{\C_n} (\E_w \W_{\Ggood{\C_n}{L}}) }}  &= \frac{\sZ_R(\Ggood{\C_n}{L}) - \E_{\C_n}  \sZ_R(\Ggood{\C_n}{L}) }{\sqrt{\Var_{\C_n}(\sZ_R(\Ggood{\C_n}{L}))}}  \cdot \sqrt{\frac{\Var_{\C_n}(\sZ_R(\Ggood{\C_n}{L}))}{\Var_{\C_n} (\E_w \W_{\Ggood{\C_n}{L}}) }}.
\end{align}
By  Proposition~\ref{eq: var_k_approx} and \eqref{eq:sd_triangle}, we obtain
\begin{align}
	\label{eq: popcorn}
	\big|\sqrt{\Var_{\C_n}(\sZ_R(\Ggood{\C_n}{L}))} - \sqrt{\Var_{\C_n}(\E_w \W_{\Ggood{\C_n}{L}})}\big| \leq \sqrt{\Var_{\C_n}(\sZ_R(\Ggood{\C_n}{L}) - \E_w \W_{\Ggood{\C_n}{L}})} = o(\sqrt{n}).
\end{align}
Together, \eqref{var_Gn_ub_lb_linear} and  \eqref{eq: popcorn} imply
\[ \Var_{\C_n}(\sZ_R(\Ggood{\C_n}{L})) = (1+ o(1))\Var_{\C_n}(\E_w \W_{\Ggood{\C_n}{L}}),\]
and as a result, from \eqref{eq: R_appro_decomposition} we have 
\[	\frac{\sZ_R(\Ggood{\C_n}{L}) - \E_{\C_n}  \sZ_R(\Ggood{\C_n}{L})}{\sqrt{\Var_{\C_n} (\E_w \W_{\Ggood{\C_n}{L}})}}   = (1 + o(1)) \frac{ \sZ_R(\Ggood{\C_n}{L}) - \E_{\C_n}  \sZ_R(\Ggood{\C_n}{L}) }{\sqrt{\Var_{\C_n}(\sZ_R(\Ggood{\C_n}{L}))}} .\]
This, coupled with \eqref{eq: k_approx_close}, asserts that to prove the CLT for $\E_w \W_{\Ggood{\C_n}{L}}$ given in \eqref{eq:conv_R_approx}, it suffices to establish  the CLT for $\sZ_R(\Ggood{\C_n}{L})$ as the following
\begin{equation}\label{CLT_local_approx}
\frac{ \sZ_R(\Ggood{\C_n}{L}) - \E_{\C_n}  \sZ_R(\Ggood{\C_n}{L}) }{ \sqrt{\Var_{\C_n}(\sZ_R(\Ggood{\C_n}{L}))} }
\stackrel{d}{\to} \mathrm N(0,1).
\end{equation}
The advantage of considering $\sZ_R(\Ggood{\C_n}{L})$ instead of
$\E_w\W_{\Ggood{\C_n}{L}}$ is that the former can be written as a local
graph statistic, allowing us to apply Theorem~\ref{thm: BarbourRollin}.
Indeed, setting $k:=L+R+1$, we may write in the form of \eqref{eq:local_stat_rep}
\[ \sZ_R(\Ggood{\C_n}{L}) =
\sum_{v\in V(\C_n)} h\bigl(\bB_v^k(\C_n),v\bigr), \]
where, for a finite rooted multigraph $(H,o)$,
\[ h(H,o) := \frac12 \sum_{f\in\cE(H): o\in\partial_H(f)} \ind_{\{f\in\cE(\Ggood{H}{L})\}}
\E_w\big[ w_f \ind_{\{f\in\M_{\bB_f^R(\Ggood{H}{L})}\}} \big]. \]
Thus, $h\bigl(\bB_v^k(\C_n),v\bigr)$ depends only on the
$k$-neighborhood of $v$ in $\C_n$.

Since $R=o(\log n)$ and $L=AR$, we have $k=L+R+1\le a\log n$
for all sufficiently large $n$, where $a>0$ is the constant appearing in
Theorem~\ref{thm: BarbourRollin}. Moreover, $\|h\|_\infty\le D/2$. By
Proposition~\ref{eq: var_k_approx} and \eqref{var_Gn_ub_lb_linear}, for some $c' > 0$,
\[ \Var_{\C_n}\bigl(\sZ_R(\Ggood{\C_n}{L})\bigr) \ge c'n. \]
Therefore, Theorem~\ref{thm: BarbourRollin} yields
\[ d_{\mathrm W}\Big( \mathcal L\Big( \frac{
\sZ_R(\Ggood{\C_n}{L}) - \E_{\C_n}\sZ_R(\Ggood{\C_n}{L})}{ \sqrt{ \Var_{\C_n}\bigl(\sZ_R(\Ggood{\C_n}{L})\bigr)}} \Big),
\mathrm N(0,1)\Big) \le C\frac{D^{Ck}}{\sqrt n} =o(1). \]
This proves \eqref{CLT_local_approx}. It remains only to prove
Proposition~\ref{eq: var_k_approx}.

\begin{proof}[Proof of Proposition~\ref{eq: var_k_approx}]
Apply Lemma \ref{cor: conf_var_bound} to the function
\[ f(G) := \E_w\W_{\Ggood{G}{L}} - \sZ_R(\Ggood{G}{L}),
\qquad G\in\mathcal{MG}_{n,D}. \]
Since $N\asymp n$, it suffices to show that
\begin{equation} \label{eq: var_diff_bound_a1}
\E_{\C_n} \sum_{e\in\cE(\C_n)} \big[
f(\C_n)-f(\C_n\setminus e) \big]^2 =o(n)
\end{equation}
and
\begin{equation} \label{eq: var_diff_bound_a2}
\E_{\C_n} \sum_{e,e'\in\cE(\C_n)}
\big[ f(\C_n\setminus e) - f(\C_n\setminus\{e,e'\}) \big]^2 =o(n^2).
\end{equation}

We first prove \eqref{eq: var_diff_bound_a1}, splitting the sum according
to whether $e\in\cE(\Ggood{\C_n}{L})$.

Suppose first that $e\in\cE(\Ggood{\C_n}{L})$. By
Lemma~\ref{lem:delete_e_tree},
\[ \Ggood{(\C_n\setminus e)}{L} = \Ggood{\C_n}{L}\setminus e. \]
Therefore,
\begin{align*}
f(\C_n)-f(\C_n\setminus e) = \big( \E_w\W_{\Ggood{\C_n}{L}} -
\E_w\W_{\Ggood{\C_n}{L}\setminus e} \big) - \big( \sZ_R(\Ggood{\C_n}{L})
- \sZ_R(\Ggood{\C_n}{L}\setminus e) \big).
\end{align*}
The graph $\Ggood{\C_n}{L}$ is simple, and
$\bB_f^L(\Ggood{\C_n}{L})$ is a tree for every
$f\in\cE(\Ggood{\C_n}{L})$. Since $L=AR$,
Theorem~\ref{thm:local_perturbative_bound} gives
\[ |f(\C_n)-f(\C_n\setminus e)| \le Ce^{-cR}. \]
Consequently,
\begin{equation}
\label{eq:efron_stein_term1_good}
\sum_{e\in\cE(\Ggood{\C_n}{L})}
\bigl[f(\C_n)-f(\C_n\setminus e)\bigr]^2 \le Cne^{-2cR}
= o(n).
\end{equation}
Now suppose that $e\notin\cE(\Ggood{\C_n}{L})$. We use the crude bounds
\[ \big| \E_w\W_{\Ggood{\C_n}{L}} - \E_w\W_{\Ggood{(\C_n\setminus e)}{L}} \big|
\le CD^L, \]
\[ \big| \sZ_R(\Ggood{\C_n}{L}) - \sZ_R(\Ggood{(\C_n\setminus e)}{L}) \big|
\le CD^{L+R} \le CD^{2L}. \]
Indeed, deleting $e$ can change the $L$-good status only of edge copies
whose $L$-neighborhood contains $e$, and there are at most $CD^L$ such
edge copies. Each such change can affect at most $CD^R$ terms in the
local statistic $\sZ_R$.

Thus,
\[ |f(\C_n)-f(\C_n\setminus e)|^2 \le CD^{4L}. \]
Using \eqref{eq: L_n_size}, we obtain
\begin{align}\label{eq:efron_stein_term1_bad}
\E_{\C_n} \sum_{e\notin\cE(\Ggood{\C_n}{L})} \bigl[f(\C_n)-f(\C_n\setminus e)\bigr]^2 \le
CD^{4L} \E_{\C_n} \big| \cE(\C_n)\setminus\cE(\Ggood{\C_n}{L}) \big| 
=O(D^{6L}) = o(n).
\end{align}
Combining \eqref{eq:efron_stein_term1_good} and
\eqref{eq:efron_stein_term1_bad} proves
\eqref{eq: var_diff_bound_a1}.

The proof of \eqref{eq: var_diff_bound_a2} is analogous. Fix
$e\in\cE(\C_n)$ and apply the preceding argument to $\C_n\setminus e$.
The term corresponding to $e'=e$ is zero. Uniformly over $e$, the
contribution from $e'\in
\cE\bigl(\Ggood{(\C_n\setminus e)}{L}\bigr)$
is at most $Cne^{-2cR}=o(n)$. For the remaining edge copies, the crude
bound above gives
\[ CD^{4L} \big| \cE(\C_n\setminus e) \setminus \cE\bigl(\Ggood{(\C_n\setminus e)}{L}\bigr)
\big|.\]
Deleting an edge cannot turn an $L$-good edge into an $L$-bad edge, and
hence
\[\big| \cE(\C_n\setminus e) \setminus \cE\bigl(\Ggood{(\C_n\setminus e)}{L}\bigr)
\big| \le \big| \cE(\C_n) \setminus \cE(\Ggood{\C_n}{L}) \big|.\]
Summing over $e$ and using \eqref{eq: L_n_size}, the total contribution
of the $L$-bad edges is $O(nD^{6L}) =o(n^2)$,
while the contribution of the $L$-good edges is $O(n^2e^{-2cR}) =o(n^2)$.
This proves \eqref{eq: var_diff_bound_a2} and completes the proof.
\end{proof}

\section{Proofs of Lemma~\ref{lemma:near_delete_edge_tail}, Propositions~\ref{prop:annular_delete_edge} and \ref{prop:far_delete_edge_tail}}
\label{sec: local_approximation}
Throughout this section, all graphs are finite, simple, and deterministic, and
all probabilities and expectations are taken with respect to the random edge
weights. To simplify notation, we write $\prob$ and $\E$ in place of
$\prob_w$ and $\E_w$, respectively. We begin with proving Lemma~\ref{lemma:near_delete_edge_tail}.
\subsection{Proof of Lemma~\ref{lemma:near_delete_edge_tail}}
We claim that there exist positive constants
$c_1$ and $C_1$, depending only on $D$, such that whenever
$\bB_e^r(H)$ is a tree with $r\ge 1$,
\begin{equation} \label{eq:weighted_edge_localization}
| \E w_e X_e(H) - \E w_e X_e^r(H)  | \le C_1e^{-c_1r},
\end{equation}
where we recall that  $X_e^r(H) = X_e(\bB_e^r(H)).$
By Theorem~\ref{thm:vertex_bonus}, 
$\prob(
X_e(H)\ne X_e^r(H)) \le Ce^{-cr}$,
where the constants only depend on $D$. Therefore, by Cauchy-Schwarz,
\[ \begin{aligned}
\big| \E w_e X_e(H) - \E w_e X_e^r(H)  \big| 
&\le \E\big[ w_e|X_e(H)-X_e^r(H)| \big] \\
&\le (\E w_e^2)^{1/2} \big( \prob(X_e(H)\ne X_e^r(H) )
\big)^{1/2} \le \sqrt{2C} e^{-cr/2}.
\end{aligned}
\]
Renaming constants gives \eqref{eq:weighted_edge_localization}.

We now prove the lemma. Choose $\delta_0>0$ small enough that $\delta_0\le \tfrac16$ and  $\delta_0\log D<\tfrac{c_1}{8}$.
Fix $0<\delta\le\delta_0$. Take $r:=\lfloor (1-2\delta)R\rfloor$.
For large $R$ we have $r\ge 1$. If $d(e,f)\le\delta R$, then
\[ \bB_e^r(G)\subseteq \bB_f^R(G). \]
Indeed, if $v\in V(\bB_e^r(G))$, then
$d_G(v,e)\le r$, and hence $d_G(v,f)\le r+d(e,f)\le (1-\delta)R<R$.
Thus $v\in V(\bB_f^R(G))$. Since $\bB_f^R(G)$ is a tree, it follows
that $\bB_e^r(G)$ is also a tree.

Let $H_e:=\bB_e^R(G).$
Since $r\le R$, we have $\bB_e^r(H_e)=\bB_e^r(G)$.
Therefore \eqref{eq:weighted_edge_localization}, applied first to $H=G$ and
then to $H=H_e$, gives
\[ \left| \E w_e X_e(G) - \E w_e X_e^R(G) \right|
\le C_1e^{-c_1r}+C_1e^{-c_1r} 
\le C_2e^{-c_1r}.
\]
For large $R$, since $\delta\le1/6$, $r\ge (1-3\delta)R\ge R / 2$.
Hence
\begin{equation}\label{eq:per_edge_near_bound}
  \left| \E w_e X_e(G) - \E w_e X_e^R(G) \right| \le C_2e^{-c_1R/2}.  
\end{equation}
It remains to sum over the near edges. Since $G$ has maximum degree at most
$D$, $\#\{e:d(e,f)\le \delta R\} \le C_DD^{\delta R}$.
Using \eqref{eq:per_edge_near_bound},
\[\Big| \sum_{e:\,d(e,f)\le \delta R}
\big( \E w_e X_e(G) - \E w_e X_e^R(G) \big) \Big|
\le C_DD^{\delta R} C_2e^{-c_1R/2}  
= C_3 \exp\left\{ -\left(\frac{c_1}{2}-\delta\log D\right)R \right\}.\]
By the choice of $\delta_0$, $\tfrac{c_1}{2}-\delta\log D
\ge \tfrac{3c_1}{8}>0.$
Thus the last display is bounded by $Ce^{-cR}$ for some $c>0$ depending only
on $\delta$ and $D$. This proves the claim for large $R$. Enlarging $C$ handles
the remaining finite range of $R$. This concludes the proof of Lemma~\ref{lemma:near_delete_edge_tail}.

Before proving Propositions~\ref{prop:annular_delete_edge} and \ref{prop:far_delete_edge_tail}, we introduce the following important definition. 

\begin{df}[Alternating path]
Let $G$ be a finite weighted simple graph whose maximum weight matching is unique, and let
$f\in E(G)$. Define $\I_f(G)$ to be the connected component of
$\M_G\triangle \M_{G\setminus f}$ that contains $f$, with the convention that
$\I_f(G)=\emptyset$ if $f\notin \M_G\triangle \M_{G\setminus f}$. 
\end{df}
Since the symmetric
 difference of two matchings is a disjoint union of alternating paths and alternating cycles, $\I_f(G)$ is either empty, an alternating path, or an alternating cycle.

We will repeatedly use the following consequence of uniqueness. Since the
edge weights are continuous, almost surely,
\begin{equation}
\label{eq:symmetric_difference_equals_If}
\M_G\triangle\M_{G\setminus f} = \I_f(G).
\end{equation}
Indeed, suppose that $\mathcal C$ is a connected component of
$\M_G\triangle\M_{G\setminus f}$ that does not contain $f$. Since
$f\notin E(\mathcal C)$, exchanging the two alternating edge sets on
$\mathcal C$ preserves feasibility both in $G$ and in $G\setminus f$.
If the total weight of $\M_G\cap E(\mathcal C)$ is larger than that of
$\M_{G\setminus f}\cap E(\mathcal C)$, this exchange produces a matching
of $G\setminus f$ heavier than $\M_{G\setminus f}$. The reverse strict
inequality similarly contradicts the optimality of $\M_G$. If the two
weights are equal, the exchange produces a second maximum weight matching,
contrary to uniqueness. Thus no such component $\mathcal C$ exists,
which proves \eqref{eq:symmetric_difference_equals_If}.

\begin{lem}[Probability bound for a specific alternating path on a tree]
\label{lem:product_upper_bound_tree}
Let $G$ be a finite tree, and let $e,f\in E(G)$ be distinct edges. Write
\[ (x_{-1} x_0) = e_0=f\sim e_1 = (x_0 x_1) \sim \cdots \sim (x_{r-1} x_r) = e_r=e\]
for the unique edge path from $f$ to $e$.

Delete $f$ from $G$, and let $T$ be the component of $G\setminus f$ that contains $e$.
Root $T$ at the endpoint $x_0$ of $e_0$ that is closer to $e$. 
For each $0\le j\le r-1$, let $h_{j,1}=e_{j+1}=(x_jx_{j+1}),$ and $y_{j,1}:=x_{j+1},$
and write the remaining child edges of $x_j$ in $T$ as $h_{j,\ell}=(x_jy_{j,\ell})$ for $2\le \ell\le m_j$, where $m_j$ is the number of children of $x_j$.
For $u\in V(T)$, let $T_u$ denote the descendant subtree of $T$ rooted at $u$. Set
\[\B_{j,\ell}:=\B(y_{j,\ell},T_{y_{j,\ell}}), \qquad
\xi_{j,\ell}:=\ind_{\{w_{h_{j,\ell}}>\B_{j,\ell}\}}, \]
and define $\xi_j:=\xi_{j,1}$ and $Z_j:=\xi_{j,2} + \xi_{j,3} + \cdots+\xi_{j,m_j}.$
Then
\[ \prob\bigl( e\in \I_f(G) \bigr) \le \E\Bigl[ e^{-\B(x_0,T)} \prod_{j=0}^{r-1}\frac{\xi_j}{1+Z_j} \Bigr]. \]
\end{lem}
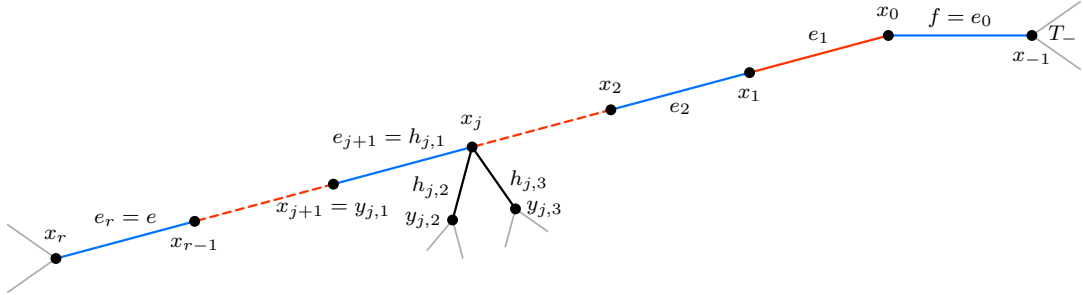
\begin{figure}[ht]
    \centering
    \begin{tikzpicture}[
        line cap=round,
        v/.style={circle,fill=black,inner sep=1.45pt},
        edge/.style={line width=0.85pt},
        faint/.style={line width=0.65pt,gray!65},
        blueedge/.style={edge,draw=blue!55!cyan},
        rededge/.style={edge,draw=red!60!orange},
        reddash/.style={
            line width=0.85pt,
            densely dashed,
            draw=red!60!orange
        },
        lab/.style={font=\scriptsize,inner sep=1pt},
        elab/.style={font=\scriptsize,fill=white,inner sep=1pt}
    ]

    \pgfmathsetmacro{\L}{1.90}   
    \def\A{195}                  

    \coordinate (x0)   at (0,0);
    \coordinate (xm1)  at ($(x0)+(0:\L)$);

    \coordinate (x1)   at ($(x0)+(\A:\L)$);
    \coordinate (x2)   at ($(x1)+(\A:\L)$);
    \coordinate (xj)   at ($(x2)+(\A:\L)$);
    \coordinate (xjp1) at ($(xj)+(\A:\L)$);
    \coordinate (xrm1) at ($(xjp1)+(\A:\L)$);
    \coordinate (xr)   at ($(xrm1)+(\A:\L)$);

    \coordinate (yj2) at ($(xj)+(255:1.00)$);
    \coordinate (yj3) at ($(xj)+(305:1.00)$);

    \coordinate (tma) at ($(xm1)+(35:0.78)$);
    \coordinate (tmb) at ($(xm1)+(-35:0.78)$);

    \coordinate (y21) at ($(yj2)+(230:0.52)$);
    \coordinate (y22) at ($(yj2)+(285:0.52)$);
    \coordinate (y31) at ($(yj3)+(255:0.52)$);
    \coordinate (y32) at ($(yj3)+(325:0.52)$);

    \coordinate (xra) at ($(xr)+(145:0.78)$);
    \coordinate (xrb) at ($(xr)+(215:0.78)$);

    \draw[blueedge] (x0) -- (xm1);
    \node[elab,above=2pt] at ($(x0)!0.50!(xm1)$) {$f=e_0$};

    \draw[faint] (xm1) -- (tma);
    \draw[faint] (xm1) -- (tmb);
    \node[lab,right=5pt] at (xm1) {$T_-$};

    \draw[rededge] (x0) -- (x1);
    \node[elab,above=3pt] at ($(x0)!0.50!(x1)$) {$e_1$};

    \draw[blueedge] (x1) -- (x2);
    \node[elab,below=3pt] at ($(x1)!0.50!(x2)$) {$e_2$};

    \draw[reddash] (x2) -- (xj);

 \draw[blueedge] (xj) -- (xjp1);
\node[elab,above=3pt,xshift=-5pt,yshift=2pt]
    at ($(xj)!0.50!(xjp1)$)
    {$e_{j+1}=h_{j,1}$};

    \draw[reddash] (xjp1) -- (xrm1);

    \draw[blueedge] (xrm1) -- (xr);
    \node[elab,above=4pt] at ($(xrm1)!0.50!(xr)$) {$e_r=e$};

    \draw[faint] (xr) -- (xra);
    \draw[faint] (xr) -- (xrb);

\draw[edge] (xj) -- (yj2);
\node[elab,left=1pt,xshift=-2pt]
    at ($(xj)!0.55!(yj2)$)
    {$h_{j,2}$};

\draw[edge] (xj) -- (yj3);
\node[elab,right=1pt,xshift=3pt]
    at ($(xj)!0.55!(yj3)$)
    {$h_{j,3}$};

    \draw[faint] (yj2) -- (y21);
    \draw[faint] (yj2) -- (y22);
    \draw[faint] (yj3) -- (y31);
    \draw[faint] (yj3) -- (y32);

    \node[v] at (x0) {};
    \node[lab,above=5pt] at (x0) {$x_0$};

    \node[v] at (xm1) {};
    \node[lab,below=5pt] at (xm1) {$x_{-1}$};

    \node[v] at (x1) {};
    \node[lab,below=5pt] at (x1) {$x_1$};

    \node[v] at (x2) {};
    \node[lab,above=5pt] at (x2) {$x_2$};

    \node[v] at (xj) {};
    \node[lab,above=5pt] at (xj) {$x_j$};

    \node[v] at (xjp1) {};
    \node[lab,below=6pt] at (xjp1) {$x_{j+1}=y_{j,1}$};

    \node[v] at (xrm1) {};
    \node[lab,below=5pt] at (xrm1) {$x_{r-1}$};

    \node[v] at (xr) {};
    \node[lab,above=5pt] at (xr) {$x_r$};

    \node[v] at (yj2) {};
    \node[lab,left=3pt] at (yj2) {$y_{j,2}$};

    \node[v] at (yj3) {};
    \node[lab,right=3pt] at (yj3) {$y_{j,3}$};

    \end{tikzpicture}
    \caption{Alternating path. Along the path, the displayed edges
    alternate between $\M_G$ (blue) and $\M_{G\setminus f}$ (red).}
    \label{fig:influential_path}
\end{figure}

\begin{proof}
Let $T_-$ be the other component of $G\setminus f $, rooted at the other
endpoint $x_{-1}$ of $f$.  Since $G\setminus f $ is the disjoint union of $T$ and $T_-$, the restriction of $\M_{G\setminus f }$ to $T$ is $\M_{T}$. In particular,
 $\{e\in \M_{G\setminus f }\} = \{e\in \M_{T}\}. $

 On the event
$e\in  \I_f(G)$, the component of $\M_G \triangle \M_{G\setminus f}$ containing $f$ is
nonempty. Since $f\notin E(G\setminus f)$, this implies $f\in \M_G$. 

Moreover, by the definition of $\I_f(G)$, on the event
$e\in\I_f(G)$, we have  edges between $f$ and $e$ belong to $\M_G $ and $\M_{G\setminus f}$ alternately, i.e., $e_0, e_2, e_4, \ldots \in \M_G $ and $e_1, e_3, e_5, \ldots \in \M_{G \setminus f}$.

By the
tree recursion at the two endpoints of $f=e_0$,
\[ w_f>\B(x_0,T)+\B(x_{-1},T_{-}) . \]
(See \eqref{eq:edge_MWM_rep}.) Fix $0\le j\le r-1$. Let $\mathsf{N}$ be the one of the two matchings
$\M_G$ and $\M_{G\setminus f}$ that contains $e_{j+1}$. Since the path is
alternating, $\mathsf{N}$ does not contain the adjacent edge $e_j=(x_{j-1}x_j)$. Thus, in $\mathsf{N}$, the vertex $x_j$ is matched downward to $x_{j+1}$.

The restriction of $\mathsf N$ to the descendant subtree $T_{x_j}$ must be
optimal given the boundary condition that $x_j$ is not matched to its parent.
Otherwise, replacing this restriction by a better matching inside $T_{x_j}$
would strictly improve the corresponding global maximum matching.

Define
\[ Y_{j,\ell}:=(w_{h_{j,\ell}}-\B_{j,\ell})_+ . \]
By the bonus recursion \eqref{eq:bonus_cavity_recursion} at the root $x_j$ of $T_{x_j}$, the optimal downward
choice is a child edge maximizing $Y_{j,\ell}$. 
Since $\mathsf{N}$ uses $h_{j,1}=e_{j+1}$, uniqueness of the maximum
weight matching implies that $Y_{j,1}>0$ and $Y_{j,1}>Y_{j,\ell}$ for $2\le \ell\le m_j.$
Hence, $A_j:= \big \{ Y_{j,1}>\max_{2\le \ell\le m_j} Y_{j,\ell} \big\}$
holds, where the maximum over an empty set is interpreted as $0$.

So, we deduce that 
\[ \{e\in \I_f(G) \} \subseteq \{w_f> \B(x_0,T)+\B(x_{-1},T_{-}) \}\cap \bigcap_{j=0}^{r-1}A_j . \]
Conditioning on all weights except $w_f$ gives
\begin{equation*}
\prob\bigl(e\in \I_f(G) \bigr)
\le \E\Big[ e^{-\B(x_0,T)-\B(x_{-1},T_{-})}; \bigcap_{j=0}^{r-1}A_j \Big]
\le \E\Big[ e^{-\B(x_0,T)}; \bigcap_{j=0}^{r-1}A_j \Big].
\end{equation*}
Fix $s\in\{0,\ldots,r-1\}$. Let $\cF_s$ be the $\sigma$-field generated by all edge weights off
the child edges of $x_s$, together with the following partial information about
the edge weights on the child edges of $x_s$: we reveal the active set
\[ \{\ell:w_{h_{s,\ell}}>\B_{s,\ell}\}, \]
and the unordered multiset of positive residuals
\[ w_{h_{s,\ell}}-\B_{s,\ell}, \qquad
\ell:w_{h_{s,\ell}}>\B_{s,\ell}. \]
By the memoryless property of the exponential distribution, conditional on
$\cF_s$, the labels of these positive residuals are exchangeable. Hence the
label of the largest positive residual is uniform over the active labels.
Therefore
\[\prob(A_s\mid \cF_s)=\frac{\xi_s}{1+Z_s}. \]
Indeed, if $h_{s,1}$ is not active, both sides are zero, while if $h_{s,1}$ is
active, the number of active labels is $1+Z_s$. Moreover,
\[ e^{-\B(x_0,T)} \ind_{\bigcap_{j=0}^{s-1}A_j} \prod_{j=s+1}^{r-1}\frac{\xi_j}{1+Z_j}\]
is $\cF_s$-measurable. Indeed, for $j>s$, the quantities $\xi_j$ and $Z_j$
depend only on child edges of $x_j$ and on descendant subtrees below those
children, which are already included in $\cF_s$.

It remains to check the factors involving the part above $x_s$. By iterating
the bonus recursion from $x_s$ up to $x_0$, the quantities $\B(x_0,T)$ and
$A_0,\ldots,A_{s-1}$ depend on the subtree $T_{x_s}$ only through the single
value
\[ \B(x_s,T_{x_s}) = \max_{\ell}
(w_{h_{s,\ell}}-\B_{s,\ell})_+ . \]
This value is determined by the active set together with the unordered multiset
of positive residuals revealed in $\cF_s$: if the active set is empty, it is
$0$, while otherwise it is the largest element of that unordered multiset.
Therefore the displayed random variable is $\cF_s$-measurable. So,
\[ \E\Big[ e^{-\B(x_0,T)} \ind_{\bigcap_{j=0}^{s-1}A_j} \ind_{A_s}
\prod_{j=s+1}^{r-1}\frac{\xi_j}{1+Z_j} \Big]
= \E\Big[ e^{-\B(x_0,T)} \ind_{\bigcap_{j=0}^{s-1}A_j} \frac{\xi_s}{1+Z_s}
\prod_{j=s+1}^{r-1}\frac{\xi_j}{1+Z_j} \Big]. \]
Applying this identity successively for $s=r-1,r-2,\ldots,0$ yields
\[ \E\left[ e^{-\B(x_0,T)}\ind_{\bigcap_{j=0}^{r-1}A_j} \right]
= \E\Bigl[ e^{-\B(x_0,T)} \prod_{j=0}^{r-1} \frac{\xi_j}{1+Z_j} \Bigr], \]
which proves the lemma.
\end{proof}
\begin{lem}[Exponential tail on the length of alternating path on a tree]
\label{lem:tree_boundary_tail}
Let $G$ be a finite  tree with maximum degree at most $D$.  Then there exist constants
$C<\infty$ and $\eta\in(0,1)$, depending only on $D$, such that for every $f\in E(G)$ and
every integer $R\ge 4$,
\[ \prob\bigl(\I_f(G)\textnormal{ reaches }\partial \bB_f^R(G)\bigr)
\le C\eta^R . \]
\end{lem}
To prove the lemma, we need the following estimate.
\begin{lem}[One-step $q$-Jacobian contraction]
\label{lem:q_jacobian_contraction}
Let $\mathscr C$ be a finite set with $|\mathscr C|\le D-1$, and let
$\mathscr I\subseteq \mathscr C$. For each $y\in \mathscr C$, let $p_y\in[1/D,1]$, and assume that  $p_y\le 1-1/(2D)$ for all $y\in \mathscr I$.
Let $\{\zeta_y:y\in \mathscr C\}$ be independent Bernoulli random variables with
$\prob(\zeta_y=1)=p_y$. Put
\[ X:=\sum_{y\in \mathscr I}\zeta_y, \quad W:=\sum_{z\in \mathscr C\setminus \mathscr I}\zeta_z, \quad p:=\E\Big[\frac{1}{1+X+W}\Big]. \]
For $y\in \mathscr I$, define
\[ J_y := \frac{p_y}{p} \E\Big[ \frac{1}{(1+W+Z_y)(2+W+Z_y)}
\Big], \qquad Z_y:=\sum_{u\in \mathscr I,  u\ne y}\zeta_u. \]
Then there exists $\eta\in(0,1)$, depending only on $D$, such that
\[ \sum_{y\in \mathscr  I}J_y\le \eta. \]
One may take $\eta:=1-(2D)^{-D}.$
\end{lem}
\begin{proof}
If $\mathscr I=\emptyset$, there is nothing to prove. Assume $\mathscr I\ne\emptyset$.

First observe that
\[ \sum_{y\in \mathscr I} p_y \E\Big[ \frac{1}{(1+W+Z_y)(2+W+Z_y)} \Big]
\le \E\Big[ \frac{\ind_{\{X>0\}}}{1+W+X} \Big]. \]
Indeed, by independence,
\[ \begin{aligned}
\sum_{y\in \mathscr I}
p_y \E\Big[ \frac{1}{(1+W+Z_y)(2+W+Z_y)} \Big]
&= \sum_{y\in \mathscr I} \E\Big[ \frac{\zeta_y}{(W+X)(1+W+X)}
\ind_{\{X>0\}} \Big]  \\
&= \E\Big[ \frac{X}{(W+X)(1+W+X)} \ind_{\{X>0\}} \Big] \\
&\le \E\Big[ \frac{\ind_{\{X>0\}}}{1+W+X} \Big].
\end{aligned}
\]
Therefore
\begin{equation}\label{eq:ppy}
p - \sum_{y\in \mathscr I} p_y \E\Big[
\frac{1}{(1+W+Z_y)(2+W+Z_y)} \Big]
\ge \E\Big[ \frac{\ind_{\{X=0\}}}{1+W} \Big].
\end{equation}
Since $W\le D-1$, 
\[ \E\Big[ \frac{\ind_{\{X=0\}}}{1+W} \Big] \ge \frac1D\prob(X=0). \]
For $y\in \mathscr I$, $1-p_y\ge (2D)^{-1}$. Also, $|\mathscr I|\le D-1$. Hence
\[\prob(X=0) = \prod_{y\in \mathscr I}(1-p_y)
\ge ( 2D)^{-(D-1)}.\]
Thus the LHS of \eqref{eq:ppy} is at least $(2D)^{-D}.$
Since $p\le1$, it follows that
\[ \sum_{y\in \mathscr I}J_y \le \frac{p-(2D)^{-D}}{p}
\le 1-(2D)^{-D}. \]
This proves the lemma.
\end{proof}

\begin{proof}[Proof of Lemma~\ref{lem:tree_boundary_tail}]
Let $\eta:=1-(2D)^{-D}$ and  $f=(o_1 o_2)$. After deleting $f$, root the two components at $o_1$ and $o_2$. It is enough
to bound the event that $\I_f(G)$ reaches level $R$ in one rooted
component. Let this rooted tree be $T$, with root $o$. Every vertex has at most
$D-1$ children.

For a vertex $x$ of  the rooted tree $T$, let $\mathcal C(x) =\mathcal C_T(x) $ be the set of children of $x$, and let
$T_x$ be the descendant subtree rooted at $x$. Set
\[ \B_x:=\B(x,T_x), \qquad p_x:=\E e^{-\B_x},
\qquad q_x:=-\log p_x. \]
For $y\in\mathcal C(x)$, define
\[ \xi_{x,y}:=\ind_{\{w_{(xy)}>\B_y\}}, \qquad
\rho_{x,y} := \frac{\xi_{x,y}} {1+\sum_{z\in\mathcal C(x), z\ne y}\xi_{x,z}}. \]
For a vertex $x$ of $T$, define the functions 
\[ U_x(0):=1, \qquad U_x(n):= \sum_{y\in\mathcal C(x)}\rho_{x,y}U_y(n-1), \qquad n\ge1. \]
As in Lemma~\ref{lem:product_upper_bound_tree}, applied to each edge entering
level $R$ and then summed over those edges, the probability that the alternating path reaches level $R$ in this rooted component is bounded by
\begin{equation}\label{eq:rooted_path_bound_q}
 \E[e^{-\B_o}U_o(R)].   
\end{equation}
Indeed, if $a=(x_{R-1} x_R)$ is an edge entering level $R$ in the rooted
component, then Lemma~\ref{lem:product_upper_bound_tree} gives
\[ \prob(a\in\I_f(G)) \le \E\Big[ e^{-\B_o} \prod_{j=0}^{R-1}\rho_{x_j,x_{j+1}} \Big]. \]
Taking a union bound over all edges entering level $R$, and using the recursive
definition of $U_o(R)$, gives \eqref{eq:rooted_path_bound_q}.

We prove that, for every vertex $x$ and every $n\ge1$,
\begin{equation} \label{eq:weighted_susceptibility_q}
    \E[e^{-\B_x}U_x(n)]\le \eta^{n-1}.
\end{equation}

Set
\[ \tau_x(n):=\E[e^{-\B_x}U_x(n)],
\qquad s_x(n):=\frac{\tau_x(n)}{p_x}. \]
Since $0\le U_x(n)\le1$, we have  $0\le s_x(n)\le1.$
The bonus recursion gives
\[ \B_x=\max_{y\in\mathcal C(x)}(w_{(xy)}-\B_y)_+. \]
For each $y\in\mathcal C(x)$,
\[ \prob(\xi_{x,y}=1) = \prob(w_{(xy)}>\B_y)
= \E e^{-\B_y} = p_y. \]
By the exponential memoryless property, conditional on the active
indicators $(\xi_{x,y})_{y\in\mathcal C(x)}$, the positive residuals
$w_{(xy)}-\B_y$ are independent $\mathrm{Exp}(1)$ random variables.
Therefore,
\[ p_x=\E e^{-\B_x} = \E\Big[ \frac{1}{1+\sum_{y\in\mathcal C(x)}\xi_{x,y}} \Big]. \]
Thus, if $\zeta_y$, $y\in\mathcal C(x)$, are independent Bernoulli variables with
means $p_y$, then
\begin{equation}\label{eq:p_recursion_q}
p_x = \E\Big[ \frac{1}{1+\sum_{y\in\mathcal C(x)}\zeta_y} \Big].  
\end{equation}
Equivalently, $q_x$ is obtained from the recursion
\[ q_x =
\Psi\bigl((q_y)_{y\in\mathcal C(x)}\bigr),
\]
where
\[ \Psi(q_1,\ldots,q_k) := -\log \E\Big[ \frac{1}{1+\zeta_1+\cdots+\zeta_k} \Big],
\qquad \prob(\zeta_i=1)=e^{-q_i}. \]
Let
\[
I_x(n):=
\{y\in\mathcal C(x):T_y\text{ contains a descendant path of length }n-1\}.
\]
Only children in $I_x(n)$ contribute to $U_x(n)$. Conditioning on the descendant subtrees below the children of $x$, and then
integrating over the edge weights from $x$ to its children, gives
\begin{equation}\label{eq:T_recursion_q}
 \tau_x(n) = \sum_{y\in I_x(n)} b_{x,y}\tau_y(n-1),
\end{equation}
where
\begin{equation} \label{eq:bxy_def_q}
    b_{x,y}   :=\E\Big[ \frac{1}{\big(1+\sum_{z\in\mathcal C(x),\,z\ne y}\zeta_z\big)
\big(2+\sum_{z\in\mathcal C(x),\,z\ne y}\zeta_z\big)} \Big].
\end{equation}
Here the variables $\zeta_z$ are independent Bernoulli variables with
$\prob(\zeta_z=1)=p_z$.

We justify the recursion \eqref{eq:T_recursion_q}. Condition first on the
descendant subtrees below the children of $x$. Then the quantities
$\B_z$ and $U_z(n-1)$, $z\in\mathcal C(x)$, are fixed, and only the edge
weights $w_{(xz)}$ remain random.

For this conditional calculation, let $\widehat\zeta_z$,
$z\in\mathcal C(x)$, be conditionally independent Bernoulli variables satisfying
\[ \prob(\widehat\zeta_z=1\mid \mathcal F_x)=e^{-\B_z}, \]
where $\mathcal F_x$ is the sigma-field generated by the descendant subtrees
below the children of $x$. Thus $\widehat\zeta_z$ is the conditional version
of the active indicator $\ind_{\{w_{(xz)}>\B_z\}}$.

Fix $y\in\mathcal C(x)$. Suppose that $y$ is active and exactly $k$ other
children of $x$ are active. Then $\rho_{x,y}=1/(k+1)$. Conditional on this
active set, the positive residuals
\[ w_{(xz)}-\B_z,\qquad z\text{ active}, \]
are independent mean-one exponential random variables, by the memoryless
property. Denote these $k+1$ residuals by $E_1,\ldots,E_{k+1}$.  On this event,
\[ \B_x=\max(E_1,\ldots,E_{k+1}), \qquad \rho_{x,y}=\frac1{k+1}. \]
Therefore, conditionally on $\mathcal F_x$ and on this active set,
\[ \E\big[e^{-\B_x}\rho_{x,y}\big] = \frac1{k+1} \E e^{-\max(E_1,\ldots,E_{k+1})}
= \frac{1}{(k+1)(k+2)}. \]
It follows that, conditionally on $\mathcal F_x$, the conditional expectation
of $ e^{-\B_x}\rho_{x,y}U_y(n-1)$ for $ y \in \mathcal{C}(x)$  is $e^{-\B_y}U_y(n-1)\widehat b_{x,y},$
where
\[\widehat b_{x,y} :=
\E\Big[ \frac{1}{ \big(1+\sum_{z\in\mathcal C(x),\,z\ne y}\widehat\zeta_z\big)
\big(2+\sum_{z\in\mathcal C(x),\,z\ne y}\widehat\zeta_z\big)} \,\Big|\,\mathcal F_x \Big].
\]
The expectation in the definition of $\widehat b_{x,y}$ is only over the
conditional Bernoulli variables $\widehat\zeta_z$, $z\ne y$.

Taking expectation over the descendant subtrees, and using independence of
the descendant subtrees below distinct children, gives
\[ \E\big[e^{-\B_y}U_y(n-1)\widehat b_{x,y}\big] =
\tau_y(n-1)\E\widehat b_{x,y}. \]
After averaging over the sibling subtrees, the conditional Bernoulli variables
$\widehat\zeta_z$ become independent Bernoulli variables with deterministic
means $\E e^{-\B_z}=p_z.$
Hence $\E\widehat b_{x,y}=b_{x,y}$,
where $b_{x,y}$ is the coefficient defined in \eqref{eq:bxy_def_q}. Therefore
\[ \E\big[e^{-\B_y}U_y(n-1)\widehat b_{x,y}\big] =
\tau_y(n-1)b_{x,y}. \]
Summing over $y\in I_x(n)$ gives \eqref{eq:T_recursion_q}.

Now define
\[ J_{x,y}:=\frac{b_{x,y}p_y}{p_x}. \]
Dividing \eqref{eq:T_recursion_q} by $p_x$, we obtain
\[ s_x(n) = \sum_{y\in I_x(n)}J_{x,y}s_y(n-1).\]

Also, for every vertex $v$, $p_v\ge D^{-1}.$
Indeed, by \eqref{eq:p_recursion_q} and the fact that
$|\mathcal C(v)|\le D-1$,
\[ p_v = \E\Big[ \frac{1}{1+\sum_{u\in\mathcal C(v)}\zeta_u} \Big]
\ge \frac{1}{1+|\mathcal C(v)|} \ge \frac1D. \]

For $n\ge2$, every $y\in I_x(n)$ is non-leaf. Hence $p_y\le1-1/(2D)$.
Indeed, if $y$ has a child $z$, then $p_z\ge 1/D$, and the recursion gives
\[ p_y = \E\Big[ \frac{1}{1+\sum_{u\in\mathcal C(y)}\zeta_u} \Big]
\le 1-\frac{p_z}{2} \le 1-\frac{1}{2D}. \]
 Therefore
Lemma~\ref{lem:q_jacobian_contraction}, applied with
$\mathscr I=I_x(n)$, gives $\sum_{y\in I_x(n)}J_{x,y}\le \eta.$
Consequently, for $n\ge2$,
\[ s_x(n) \le \eta\max_{y\in I_x(n)} s_y(n-1). \]
Since $s_x(1)\le1$, induction gives $s_x(n)\le \eta^{n-1}$ for all $n\ge 1$.
Thus
\[ \tau_x(n)=p_xs_x(n)\le \eta^{n-1}, \]
which proves \eqref{eq:weighted_susceptibility_q}.

Returning to the original edge $f=(o_1,o_2)$, the event that
$\I_f(G)$ reaches $\partial\bB_f^R(G)$ implies that the
alternating path reaches level $R$ in at least one of the two rooted
components of $G\setminus f$. Therefore, by \eqref{eq:rooted_path_bound_q}
and \eqref{eq:weighted_susceptibility_q},
\[ \prob\bigl( \I_f(G)\text{ reaches }\partial\bB_f^R(G) \bigr)
\le 2\eta^{R-1}. \]
The lemma follows with $C:=2\eta^{-1}$ and $\eta:=1-(2D)^{-D}$.
\end{proof}

\begin{rem}
The coefficients $J_{x,y}$ are the $q$-Jacobian coefficients for the recursion
$q_x=\Psi((q_z)_{z\in\mathcal C(x)})$. Indeed, by
\eqref{eq:p_recursion_q}, $\tfrac{\partial p_x}{\partial p_y}=-b_{x,y}.$
Since $q_x=-\log p_x$ and $p_y=e^{-q_y}$, the chain rule gives
\[ \frac{\partial q_x}{\partial q_y} = -\frac{b_{x,y}p_y}{p_x}. \]
Thus $J_{x,y}=|\partial q_x/\partial q_y|$.
\end{rem}

\subsection{Proof of Proposition~\ref{prop:annular_delete_edge}}

\begin{df}[$q$-Jacobian coefficients]
\label{df:q_jacobian_coefficients}
Let $S$ be a finite rooted tree. For every vertex $v\in S$, write
$S_v$ for the descendant subtree of $S$ rooted at $v$, and define
\[ p_v^S:=\E e^{-\B(v,S_v)}, \qquad q_v^S:=-\log p_v^S . \]
If $x\in V(S)$ and $y\in\mathcal C_S(x)$ is a child of $x$ in $S$, define the
$q$-Jacobian coefficient from $x$ to $y$ in the tree $S$ by
\[ J_{x,y}^S := \Big| \frac{\partial q_x^S}{\partial q_y^S} \Big|. \]
Equivalently, let $\{\zeta_z^S:z\in\mathcal C_S(x)\}$
be independent Bernoulli random variables with $\prob(\zeta_z^S=1)=p_z^S$.
Then
\begin{equation}\label{eq:Jxy_def} 
 J_{x,y}^S = \frac{p_y^S}{p_x^S} \E\Big[ \frac{1}{ \big(1+\sum_{z\in\mathcal C_S(x),\,z\ne y}\zeta_z^S\big) \big(2+\sum_{z\in\mathcal C_S(x),\,z\ne y}\zeta_z^S\big)} \Big].
\end{equation}
\end{df}

\begin{lem}[stability of $q$-Jacobian coefficients]
\label{lem:q_jacobian_stability}
There exist constants $C$ and $\gamma\in(0,1)$, depending only on $D$,
such that the following holds.

Let $T$ be a finite rooted tree of degree at most $D$, and let $u\in V(T)$.
Let $S\subseteq T_u$ be a rooted subtree with root $u$ which agrees with
$T_u$ to depth at least $\ell\ge1$ below $u$.

For $v\in\mathcal C_T(u)=\mathcal C_S(u)$, let $J_{u,v}^S$ be the $q$-Jacobian coefficient
computed in the rooted tree $S$, and let $J_{u,v}^{T}$ be the $q$-Jacobian
coefficient computed in the ambient rooted tree $T$. Equivalently,
$J_{u,v}^{T}$ is the coefficient computed in the descendant subtree $T_u$.
Then
\[ \sum_{v\in\mathcal C_T(u)} \left|J_{u,v}^S-J_{u,v}^{T}\right| \le
C\gamma^\ell . \]
\end{lem}

\begin{proof}
It suffices to consider $\ell\ge 3$ since the cases $\ell=1,2$ are absorbed into $C$. Write
$\mathcal C_T(u)=\{v_1,\ldots,v_k\}$, $k\le D$. For each $i$,
the rooted trees $S_{v_i}$ and $T_{v_i}$ agree to depth $\ell-1$.
Applying Lemma~\ref{lem: edge_bonus}(ii) yields
\[ \E|\B(v_i,S_{v_i})-\B(v_i,T_{v_i})|\le C\gamma^\ell. \]
Since $x\mapsto e^{-x}$ is $1$-Lipschitz on $[0,\infty)$,
\[ |p_{v_i}^S-p_{v_i}^T|\le C\gamma^\ell . \]
By the explicit formula for $J_{u,v}$, the map
$(p_{v_1},\ldots,p_{v_k})\mapsto J_{u,v_i}$ is uniformly Lipschitz on
$[1/(D+1),1]^k$. Hence, by increasing $C$ if necessary,
\[ |J_{u,v_i}^S-J_{u,v_i}^T| \le C\sum_{j=1}^k |p_{v_j}^S-p_{v_j}^T|
\le C\gamma^\ell . \]
Summing over $i$ proves the claim.
\end{proof}

\begin{lem}
\label{lem:annular_sum}
There exist constants $m_0 \ge 1,  C>0$ and $\theta\in(0,1)$, depending only on $D$, such
that the following holds. Let $T$ be a finite rooted tree of forward degree at
most $D-1$, fix $x\in T$, and let $m\ge m_0$, $n\ge1$.

For each descendant edge $g=(x_{n-1} x_n)$ at depth $n$ below $x$, write
$x=x_0 \sim x_1 \sim \cdots \sim x_n$ for the path from $x$ to $g$, and let
$H_g\subseteq T_x$ be a rooted subtree with root $x$ containing this path. Assume that, for
each $0\le j\le n-1$, the descendant subtree of $H_g$ rooted at $x_j$ agrees
with $T_{x_j}$ to depth at least $m+j$ below $x_j$.

For $y\in\mathcal C_{H_g}(x_j)$, let $(H_g)_y$ be the descendant subtree of
$H_g$ rooted at $y$, and set
\[ \xi_{j,y}^{g} := \ind_{\{w_{(x_jy)}>\B(y,(H_g)_y)\}} .\]
Define the local transition factor along the path to $g$ by
\[ \rho_j^{\rm loc}(g) := \frac{\xi_{j,x_{j+1}}^{g}}
{1+\sum_{y\in\mathcal C_{H_g}(x_j),\,y\ne x_{j+1}}\xi_{j,y}^{g}},
\qquad 0\le j\le n-1.\]
Then
\[ \sum_{g:\,g\textnormal{ descendant edge at depth }n\textnormal{ below }x}
\E\prod_{j=0}^{n-1}\rho_j^{\rm loc}(g)
\le C\theta^n .
\]
\end{lem}
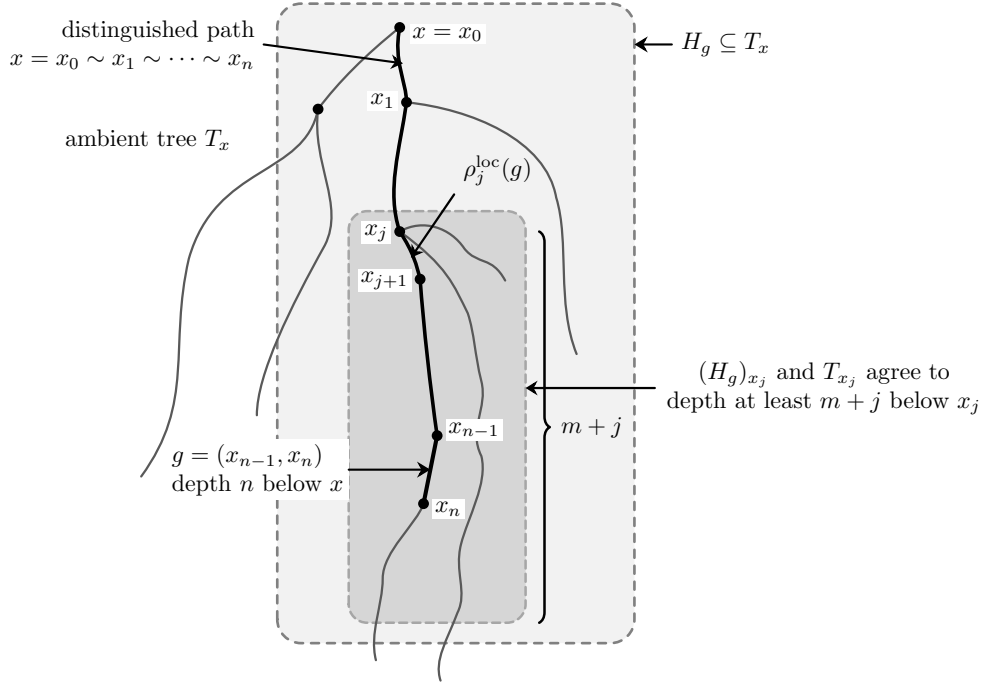
\begin{figure}[ht] \label{fig:lemma57-setup}
\centering
\begin{tikzpicture}[
    scale=0.9,
    transform shape,
    line cap=round,
    line join=round,
    every node/.style={
        font=\small,
        text=black
    },
    treeedge/.style={
        draw=black!65,
        line width=0.85pt
    },
    distinguished/.style={
        draw=black,
        line width=1.35pt
    },
    targetedge/.style={
        draw=black,
        line width=1.55pt
    },
    vertex/.style={
        circle,
        fill=black,
        inner sep=1.6pt
    },
    vlabel/.style={
        inner sep=1.2pt,
        fill=white,
        fill opacity=0.9,
        text opacity=1
    },
    hgboundary/.style={
        draw=black!50,
        dashed,
        line width=1pt,
        rounded corners=9pt,
        fill=black!5,
        fill opacity=0.48
    },
    agreement/.style={
        draw=black!35,
        dash pattern=on 3pt off 2pt,
        line width=0.95pt,
        rounded corners=7pt,
        fill=black!16,
        fill opacity=0.55
    }
]

\coordinate (x0) at (0,6.8);
\coordinate (a) at (-1.2,5.6);
\coordinate (x1) at (0.1,5.7);
\coordinate (xj) at (0,3.8);
\coordinate (xjp) at (0.3,3.1);
\coordinate (xnp) at (0.55,0.8);
\coordinate (xn) at (0.35,-0.2);
\coordinate (gmid) at ($(xnp)!0.5!(xn)$);

\path[hgboundary]
    (-1.80,-2.25) rectangle (3.45,7.15);

\path[agreement]
    (-0.75,-1.95) rectangle (1.85,4.10);

\draw[treeedge]
    (x0)
    .. controls (-0.5,6.4) and (-0.9,6.0) .. (a);

\draw[treeedge]
    (x0)
    .. controls (-0.1,6.3) and (0.1,6.0) .. (x1);

\draw[treeedge]
    (a)
    .. controls (-1.4,4.6) and (-2.7,4.8) .. (-3.1,3.4)
    .. controls (-3.4,2.1) and (-3.0,1.3) .. (-3.8,0.2);

\draw[treeedge]
    (a)
    .. controls (-1.3,4.8) and (-0.8,4.2) .. (-1.1,3.6)
    .. controls (-1.6,2.7) and (-2.1,1.7) .. (-2.1,1.1);

\draw[treeedge]
    (x1)
    .. controls (0.0,5.0) and (-0.2,4.4) .. (xj);

\draw[treeedge]
    (x1)
    .. controls (1.05,5.60) and (2.15,5.15) .. (2.30,4.30)
    .. controls (2.50,3.65) and (2.30,2.75) .. (2.60,2.00);

\draw[treeedge]
    (xj)
    .. controls (0.35,4.00) and (0.75,3.85) .. (0.95,3.55)
    .. controls (1.15,3.30) and (1.35,3.45) .. (1.55,3.08);
\draw[treeedge]
    (xj)
    .. controls (0.7,3.3) and (0.8,3.0) .. (1,2.6)
    .. controls (1.1,2.3) and (1.1,2.0) .. (1.2,1.6)
    .. controls (1.3,1.2) and (1,0.9) .. (1.2,0.5)
    .. controls (1.3,0.1) and (1.0,-0.4) .. (0.9,-0.8)
    .. controls (0.8,-1.2) and (1.0,-1.5) .. (0.85,-1.8)
    .. controls (0.7,-2.2) and (0.5,-2.5) .. (0.6,-2.8);

\draw[treeedge]
    (xn)
    .. controls (0.25,-0.5) and (-0.25,-0.9) .. (-0.25,-1.3)
    .. controls (-0.25,-1.7) and (-0.45,-2.1) .. (-0.35,-2.5);

\draw[distinguished]
    (x0)
    .. controls (-0.1,6.3) and (0.1,6.0) .. (x1);

\draw[distinguished]
    (x1)
    .. controls (0.0,5.0) and (-0.2,4.4) .. (xj);

\draw[distinguished]
    (xj)
    .. controls (0.2,3.5) and (0.25,3.3) .. (xjp);

\draw[distinguished]
    (xjp)
    .. controls (0.35,2.4) and (0.45,1.5) .. (xnp);

\draw[targetedge]
    (xnp) -- (xn);

\draw[
    decorate,
    decoration={brace,amplitude=4pt},
    black,
    line width=0.95pt
]
    (2.05,3.80) -- (2.05,-1.95)
    node[
        midway,
        right=5pt
    ] {$m+j$};

\foreach \p in {x0,a,x1,xj,xjp,xnp,xn}
    \node[vertex] at (\p) {};

\node[
    vlabel,
    anchor=west
]
    at ($(x0)+(0.14,-0.10)$)
    {$x=x_0$};

\node[
    vlabel,
    anchor=east
]
    at ($(x1)+(-0.12,0.02)$)
    {$x_1$};

\node[
    vlabel,
    anchor=east
]
    at ($(xj)+(-0.12,-0.02)$)
    {$x_j$};

\node[
    vlabel,
    anchor=east
]
    at ($(xjp)+(-0.12,-0.02)$)
    {$x_{j+1}$};

\node[
    vlabel,
    anchor=west
]
    at ($(xnp)+(0.12,0.08)$)
    {$x_{n-1}$};

\node[
    vlabel,
    anchor=west
]
    at ($(xn)+(0.12,-0.08)$)
    {$x_n$};

\node[
    anchor=east
]
    at (-2.35,5.15)
    {ambient tree $T_x$};

\node[
    anchor=east,
    align=right
]
    (pathlabel) at (-2.0,6.55)
    {distinguished path\\
     $x=x_0\sim x_1\sim\cdots\sim x_n$};

\draw[
    black,
    line width=0.8pt,
    -{Stealth[length=4.5pt,width=5.5pt]}
]
    (pathlabel.east) -- ($(x0)!0.48!(x1)$);

\node[
    anchor=west
]
    (hglabel) at (4.0,6.55)
    {$H_g\subseteq T_x$};

\draw[
    black,
    line width=0.8pt,
    -{Stealth[length=5pt,width=6pt]}
]
    (hglabel.west) -- (3.45,6.55);

\node[
    anchor=west
]
    (rholabel) at (0.8, 4.7)
    {$\rho_j^{\rm loc}(g)$};

\draw[
    black,
    line width=0.85pt,
    -{Stealth[length=4.5pt,width=5.5pt]}
]
    (rholabel.south west) -- ($(xj)!0.55!(xjp)$);

\node[
    anchor=west,
    align=center,
    text width=4.7cm
]
    (agree-label) at (3.75,1.50)
    {$(H_g)_{x_j}$ and $T_{x_j}$ agree to depth at least
     $m+j$ below $x_j$};

\draw[
    black,
    line width=0.85pt,
    -{Stealth[length=4.5pt,width=5.5pt]}
]
    (agree-label.west) -- (1.85,1.50);

\node[
    anchor=east,
    align=left,
    fill=white,
    inner sep=2pt
]
    (glabel) at (-0.75,0.30)
    {$g=(x_{n-1},x_n)$\\
     depth $n$ below $x$};

\draw[
    black,
    line width=0.8pt,
    -{Stealth[length=5pt,width=6pt]},
]
    (glabel.east) -- (gmid);

\end{tikzpicture}

\caption{Schematic setup of Lemma~\ref{lem:annular_sum}. The outer light-gray
region denotes the subtree $H_g$ rooted at $x$ and contains the distinguished path to
$g=(x_{n-1},x_n)$. For the displayed index $j$, the darker inner region
indicates that $(H_g)_{x_j}$ and $T_{x_j}$ agree to depth at least
$m+j$ below $x_j$. The quantity $\rho_j^{\rm loc}(g)$ is the local
transition factor from $x_j$ to $x_{j+1}$ computed in $H_g$.}
\end{figure}
\begin{proof}
We use the notation of Definition~\ref{df:q_jacobian_coefficients}. Since every
rooted tree considered below has forward degree at most $D-1$, we have
$p_u^S\ge 1/D$, and hence \eqref{eq:Jxy_def} gives $J_{u,v}^S\le D.$

Let $\eta:=1-(2D)^{-D}$. We first need the following contraction estimate for a subset of children. If $I\subseteq \mathcal C_T(u)$ and each $y\in I$ is a non-leaf vertex
in $T_u$, then
\begin{equation}\label{eq: bound_J_u_y_sum}
\sum_{y\in I}J_{u,y}^{T}\le \eta .
\end{equation}
Indeed, $p_z^T\ge 1/D$ for every $z$, while for non-leaf $y$,
\[ p_y^T = \E\Big[ \frac{1}{1+\sum_{z\in\mathcal C_T(y)}\zeta_z} \Big]
\le 1-\frac{1}{2D}. \]
Thus Lemma~\ref{lem:q_jacobian_contraction} applies.

By Lemma~\ref{lem:q_jacobian_stability}, if $S\subseteq T_u$ agrees with
$T_u$ to depth $\ell$ below $u$, then
\[\sum_{v\in\mathcal C_T(u)} |J_{u,v}^S-J_{u,v}^{T}| \le C_0 \gamma^\ell . \]
Choose $m_0$ so that
\[ (D-1)C_0 \gamma^{m_0}\le \frac{1-\eta}{2}, \qquad
\theta:=\frac{1+\eta}{2}. \]
Then, for all $m\ge m_0$ and $s\ge0$,
\begin{equation}\label{eq:para_m0_choice}
  \eta+(D-1)C_0 \gamma^{m+s}\le \theta<1.  
\end{equation}
Fix $n \ge 1$. For a vertex $u$ at depth $s$ below $x$, $1\le s\le n-1$,
let
\[ \mathcal A_s(u) :=
\sum_{g} \prod_{i=s}^{n-1} J_{x_i(g),x_{i+1}(g)}^{H_g}, \]
where the sum is over the edges $g$ at depth $n$ below $x$ whose path passes
through $u$. We claim that
\begin{equation}\label{eq:local_sum_estimate}
\mathcal A_s(u)\le C_1\theta^{n-s-1}, \qquad 1\le s\le n-1. 
\end{equation}
This follows by backward induction on $s$. For $s=n-1$, since  $J_{u,v}^S\le D $, we have
\[ \mathcal A_{n-1}(u) \le \sum_{y\in\mathcal C_T(u)}D \le D(D-1), \]
so the claim holds after choosing $C_1\ge D(D-1)$.

Assume the claim at level $s+1$, where $1\le s\le n-2$.  Let $I(u,s)$ be the set of all children of $u$ which
lead to an edge at level $n$. Each $y\in I(u,s)$ is a non-leaf vertex in $T_u$.
Moreover, for every edge $g$ below $y$, the subtree $(H_g)_u$ agrees
with $T_u$ to depth at least $m+s$ below $u$. Applying Lemma~\ref{lem:q_jacobian_stability} with
$S=(H_g)_u$ and with the ambient descendant subtree $T_u$, and using the
fact that $(H_g)_u$ agrees with $T_u$ to depth at least $m+s$ below
$u=x_s$, we obtain
\[ J_{u,y}^{H_g} \le J_{u,y}^{T}+C_0 \gamma^{m+s}. \]
Using the induction hypothesis,
\[ \begin{aligned}
\mathcal A_s(u) 
&\le \sum_{y\in I(u,s)} \left(J_{u,y}^{T}+C_0 \gamma^{m+s}\right) \mathcal A_{s+1}(y)  \\
&\le C_1\theta^{n-s-2} \left(\eta+(D-1)C_0 \gamma^{m+s}\right) \\
&\le C_1\theta^{n-s-1},
\end{aligned}
\]
where the second inequality follows from \eqref{eq: bound_J_u_y_sum} and the last inequality follows from \eqref{eq:para_m0_choice}. This yields the claim \eqref{eq:local_sum_estimate}.

It remains to compare the random transition product with the deterministic
$q$-Jacobian product. Fix an edge $g$ and compute all quantities inside
$H_g$. For $u\in H_g$, write
\[ \B^{H_g}(u):=\B(u,(H_g)_u), \qquad p_u^{H_g}:=\E e^{-\B^{H_g}(u)}. \]
Since
$\rho_0^{\rm loc}(g)\le \ind_{\{w_{(x_0x_1)}>\B^{H_g}(x_1)\}}$,
conditioning on the descendant subtrees below the children of $x_0$ gives
\[ \E\prod_{j=0}^{n-1}\rho_j^{\rm loc}(g) \le
\E\Big[ e^{-\B^{H_g}(x_1)} \prod_{j=1}^{n-1}\rho_j^{\rm loc}(g) \Big]. \]
For $1\le j\le n-1$ and any nonnegative $F$ depending only on the descendant
tree below $x_{j+1}$ in $H_g$, the same conditioning and exponential
memoryless calculation as in Lemma~\ref{lem:tree_boundary_tail} yields
\begin{equation} \label{id:J_iter}
   \frac{ \E\big[ e^{-\B^{H_g}(x_j)} \rho_j^{\rm loc}(g)F \big]}{p_{x_j}^{H_g}}=
J_{x_j,x_{j+1}}^{H_g} \frac{ \E\big[ e^{-\B^{H_g}(x_{j+1})}F \big] }{ p_{x_{j+1}}^{H_g}}. 
\end{equation}
Indeed, if the chosen child and $k$ sibling children are active, the transition
factor contributes $1/(k+1)$, while
$\E e^{-\max(E_1,\ldots,E_{k+1})}=1/(k+2)$. Averaging over the sibling active
indicators gives exactly the coefficient in \eqref{eq:Jxy_def}.

Iterating the  identity \eqref{id:J_iter}, with $F=
\prod_{i=j+1}^{n-1}\rho_i^{\rm loc}(g)$
at step $j$, gives
\begin{align*}
 \E\Big[ e^{-\B^{H_g}(x_1)} \prod_{j=1}^{n-1}\rho_j^{\rm loc}(g) \Big] 
 &= \Big(\prod_{j=1}^{n-1} \frac{p_{x_j}^{H_g}}{p_{x_{j+1}}^{H_g}}
J_{x_j,x_{j+1}}^{H_g} \Big) \E e^{-\B^{H_g}(x_n)}  \\
&= p_{x_1}^{H_g} \Big( \prod_{j=1}^{n-1} J_{x_j,x_{j+1}}^{H_g} \Big) 
\frac{ \E e^{-\B^{H_g}(x_n)} }{ p_{x_n}^{H_g} } \\
&= p_{x_1}^{H_g} \prod_{j=1}^{n-1} J_{x_j,x_{j+1}}^{H_g},
\end{align*}
where the final equality follows from the fact that $p_{x_n}^{H_g} = \E e^{-\B^{H_g}(x_n)}.$
Since $p_{x_1}^{H_g}\le1$, we obtain
\[ \E\prod_{j=0}^{n-1}\rho_j^{\rm loc}(g) \le \prod_{j=1}^{n-1} J_{x_j,x_{j+1}}^{H_g}, \]
with the product interpreted as $1$ when $n=1$.

If $n=1$, the desired sum is at most $D-1$, and this is absorbed into $C$.
For $n\ge2$, summing the last bound over all target edges at depth $n$ below
$x$ gives
\[ \sum_{g} \E\prod_{j=0}^{n-1}\rho_j^{\rm loc}(g) \le
\sum_{y\in\mathcal C_T(x)}\mathcal A_1(y) \le (D-1)C_1\theta^{n-2}. \]
Increasing $C$ completes the proof.
\end{proof}

\begin{lem}
\label{lem:local_alt_path}
There exist constants $C$ and
$\eta\in(0,1)$, depending only on $D$, such that the following holds. Let $G$ be a finite
tree of maximum degree at most $D$, let $f\in E(G)$, and let $R\ge r\ge 1$ be integers.
Then
\[ \sum_{e:\,d(e,f)=r} \prob\bigl(e\in \I_f(\bB_e^R(G))\bigr) \le C\eta^r . \]
\end{lem}

\begin{proof}
Write $f=(o_1 o_2)$. After deleting $f$, root the two components at $o_1$ and
$o_2$. It suffices to prove the estimate for $e$ belonging to one rooted component, and then
multiply the final constant by $2$.

Fix one rooted component, with root $o$, and let $e$ be a descendant edge with
$d(e,f)=r$. Put $H_e:=\bB_e^R(G)$. Since $R\ge r$, we have $f\in H_e$.
Let $T_e$ be the component of $H_e\setminus f$ containing $e$, rooted at
$o$. If we write the path from $o$ to $e$ as $o=x_0,x_1,\ldots,x_r$ with $e=(x_{r-1},x_r)$,
then Lemma~\ref{lem:product_upper_bound_tree}, applied inside $H_e$, gives
\[ \prob(e\in\I_f(H_e)) \le \E\Big[ e^{-\B(o,T_e)} \prod_{j=0}^{r-1}\rho_j^e \Big], \]
where
\[ \rho_j^e := \frac{\xi_{j,x_{j+1}}^e} {1+\sum_{y\in\mathcal C_{T_e}(x_j),\,y\ne x_{j+1}}\xi_{j,y}^e}, \qquad \xi_{j,y}^e := \ind_{\{w_{(x_jy)}>\B(y,(T_e)_y)\}} . \]
Let $m_0,C_0,\theta$ be the constants  from
Lemma~\ref{lem:annular_sum}. The cases $r\le m_0$ are
absorbed into the constant. For $r>m_0$, since
$e^{-\B(o,T_e)}\le1$ and $0\le\rho_j^e\le1$,
\[ \prob(e\in\I_f(H_e)) \le \E\prod_{j=m_0}^{r-1}\rho_j^e . \]

Fix an edge $h$ at distance $m_0$ from $f$ in the chosen rooted component, and
let $x$ be its endpoint farther from $f$. We sum over all  edges $e$
below $h$ with $d(e,f)=r$. Set
\[ n:=r-m_0, \qquad K_e:=H_e\cap T_x, \]
where $T_x$ is the full descendant subtree below $x$ in the chosen rooted
component. Regard $K_e$ as rooted at $x$. If we write the path from $x$ to $e$ as $x=y_0,y_1,\ldots,y_n$ with $e=(y_{n-1} y_n)$, then the factors $\rho_{m_0+j}^e$, $0\le j\le n-1$, are exactly the local
transition factors from $y_j$ to $y_{j+1}$ computed inside $K_e$.

We check the hypothesis of Lemma~\ref{lem:annular_sum} with
full tree $T_x$, local subtree $K_e$, and $m=m_0$. Fix $0\le j\le n-1$. Any
descendant of $y_j$ at depth at most $m_0+j$ has distance to $e$ at most $(m_0+j)+(n-j)=r\le R$.
Hence it belongs to $H_e$, and since it is below $x$, it belongs to $K_e$.
Thus the descendant subtree of $K_e$ rooted at $y_j$ agrees with
$(T_x)_{y_j}$ to depth at least $m_0+j$.

Therefore Lemma~\ref{lem:annular_sum} gives, for each fixed
$h$,
\[ \sum_{\substack{e:\,e\text{ below }h\\ d(e,f)=r}} \prob(e\in\I_f(H_e))
\le \sum_{\substack{e:\,e\text{ below }h\\ d(e,f)=r}} \E\prod_{j=m_0}^{r-1}\rho_j^e
\le C_0\theta^{r-m_0}. \]
There are at most $D^{m_0}$ choices of $h$ in the chosen rooted component.
Thus
\[\sum_{\substack{e:\,e\text{ in the chosen component}\\ d(e,f)=r}}
\prob(e\in\I_f(\bB_e^R(G))) \le D^{m_0}C_0\theta^{r-m_0}. \]
Multiplying by $2$ for the two components of $G\setminus f$ and enlarging
the constant to cover $r\le m_0$ gives
\[ \sum_{e:\,d(e,f)=r} \prob(e\in\I_f(\bB_e^R(G))) \le
C\eta^r \]
with $\eta:=\theta\in(0,1)$.
\end{proof}

\begin{proof}
    We are now ready to prove Proposition~\ref{prop:annular_delete_edge}. For
$r\ge 1$, set $L_r:=\{e\in E(G):d(e,f)=r\}$.
Fix $e\in L_r$ with $r\le R$, and let $H_e:=\bB_e^R(G)$. Since $G$ is a tree
and $d(e,f)\le R$, we have $f\in H_e$. Moreover, $\bB_e^R(G\setminus f)$ is
the component of $H_e\setminus f $ containing $e$. Since maximum matchings
on disjoint components decouple, we have $X_e^R(G)=\ind_{\{e\in \M_{H_e}\}}$ and $X_e^R(G\setminus f) = \ind_{\{e\in \M_{H_e\setminus f }\}}$.
Hence, by \eqref{eq:symmetric_difference_equals_If},
\[ A_e := \{X_e^R(G)\ne X_e^R(G\setminus f)\} = \{e\in \I_f(H_e)\}. \]
By Lemma~\ref{lem:local_alt_path},  for
all $1\le r\le R$,
\begin{equation}
\label{eq:annular_Ae}
\sum_{e\in L_r}\prob(A_e) \le 
\sum_{e\in L_r} \prob\bigl(e\in \I_f(\bB_e^R(G))\bigr) \le C \eta^r .
\end{equation}
Let $p>1$ and set $a=1-1/p$. By Hölder's inequality,
\[ \sum_{e\in L_r}\E[w_e\ind_{A_e}] \leq \Big(\sum_{e \in L_r} \E w_e^p \Big)^{1/p} \Big( \sum_{e \in L_r} \prob(A_e)\Big)^a
\le (\E w_e^p)^{1/p} |L_r|^{1 - a} \Big( \sum_{e \in L_r} \prob(A_e)\Big)^a.\]
Using $|L_r|\le C D^r$ and \eqref{eq:annular_Ae}, this gives
\[ \sum_{e\in L_r}\E[w_e\ind_{A_e}] \le
C D^{(1-a)r}\eta^{ar} = C\lambda^r, \qquad \lambda:=D^{1-a}\eta^a . \]
We can make $\lambda<1$ by choosing $p$ large enough.

Finally,
\begin{equation*}
\begin{aligned}
\Big| \sum_{e: \delta R\le d(e,f)\le R} \left( \E w_e X_e^R(G)
- \E w_e X_e^R(G\setminus f) \right) \Big|
&\le \sum_{r=\lceil\delta R\rceil}^{ R } \sum_{e\in L_r}\E[w_e\ind_{A_e}]  \\
&\le \sum_{r=\lceil\delta R\rceil}^{ R } C_1\lambda^r \le C_2\lambda^{\delta R}
= C_2e^{-cR},
\end{aligned}
\end{equation*}
where $c:=\delta(-\log\lambda)>0$. This proves the proposition.
\end{proof}

\subsection{Proof of Proposition~\ref{prop:far_delete_edge_tail}}

\begin{lem}
\label{prop:boundary_tail}
Let $G$ be a finite simple graph with maximum degree at most $D$. Assume that
$\bB_e^{L}(G)$ is a tree for every $e\in E(G)$. Then there exist constants
$K\ge1$, $C>0$, and $\theta\in(0,1)$, depending only on $D$, such that for
every $f\in E(G)$ and every integer $R$ with $4\le R\le L/K$,
\[ \prob\bigl(\I_f(G)\textnormal{ reaches }\partial\bB_f^R(G)\bigr) \le C\theta^R .\]
\end{lem}

\begin{proof}
By Lemma~\ref{lem:tree_boundary_tail}, there are constants $C_0>0$ and
$\eta\in(0,1)$ such that, for every finite tree $T$ of maximum degree at most
$D$, every $h\in E(T)$, and every $r\ge4$,
\begin{equation} \label{eq:tree_tail_for_prop}
\prob\bigl(\I_h(T)\text{ reaches }\partial\bB_h^r(T)\bigr) \le C_0\eta^r.
\end{equation}
Also, Theorem~\ref{thm:vertex_bonus} implies that there are constants
$C_1>0$ and $\gamma\in(0,1)$ such that, whenever $\bB_g^r(F)$ is a tree,
\begin{equation} \label{eq:local_global_cd_for_prop}
\prob\big( \ind_{\{g\in\M_F\}} \ne \ind_{\{g\in\M_{\bB_g^r(F)}\}} \big) \le C_1\gamma^r.
\end{equation}
For disconnected $F$, this is applied to the component containing $g$.

Choose a positive integer $K$, divisible by $4$, such that
\begin{equation} \label{eq:alpha_choice_for_prop}
D\gamma^{K/4-1}\le \eta .
\end{equation}
Fix $f=(uv)$ and $R$ with $4\le R\le L/K$. Put
\[ S:=\frac{KR}{4}, \qquad H:=\bB_f^R(G), \qquad T:=\bB_f^S(G). \]
Since $S\le L$ and $\bB_f^L(G)$ is a tree, $T$ is a tree. Let $E_R:=\{\I_f(G)\text{ reaches }\partial\bB_f^R(G)\}.$
Define
\[ \mathcal A_R := \left\{ \M_G\cap E(H)=\M_T\cap E(H) \right\} \cap
\left\{ \M_{G\setminus f}\cap E(H)=\M_{T\setminus f}\cap E(H) \right\}.\]
On $E_R\cap\mathcal A_R$, some edge of $H$ incident to
$\partial\bB_f^R(G)$ lies in $\M_T\triangle\M_{T\setminus f}$ and hence,
by \eqref{eq:symmetric_difference_equals_If}, belongs to $\I_f(T)$.
Therefore,
\begin{equation} \label{eq:event_reduction_for_prop}
E_R \subseteq \{\I_f(T)\text{ reaches }\partial\bB_f^R(G)\} \cup \mathcal A_R^c .
\end{equation}
Let $m:=S-R=\big(K/4-1\big)R.$
For every $g\in E(H)$, $\bB_g^m(G)=\bB_g^m(T)$, since any vertex within distance $m$ of $g$ has distance at most $R+m=S$ from
$f$. The same equality holds after deleting $f$, in the component containing
$g$. These balls are trees. Thus, by \eqref{eq:local_global_cd_for_prop} and
the triangle inequality,
\[ \prob( \ind_{\{g\in\M_G\}} \ne \ind_{\{g\in\M_T\}} ) \le 2C_1\gamma^m, \]
and the same bound holds with $G,T$ replaced by $G\setminus f,T\setminus f$.
Therefore,
\begin{equation*}
\prob(\mathcal A_R^c) \le 4C_1 |E(H)|\gamma^m.
\end{equation*}
Since $|E(H)|\le C_2D^{R+1}$, \eqref{eq:alpha_choice_for_prop} gives
\begin{equation} \label{eq:good_error_final_for_prop}
\prob(\mathcal A_R^c) \le C_3\bigl(D\gamma^{K/4-1}\bigr)^R \le C_3\eta^R.
\end{equation}

Since $R<S$, we have $\bB_f^R(T)=\bB_f^R(G)$. Hence, by
\eqref{eq:tree_tail_for_prop},
\[\prob\bigl(\I_f(T)\text{ reaches }\partial\bB_f^R(G)\bigr) =
\prob\bigl(\I_f(T)\text{ reaches }\partial\bB_f^R(T)\bigr) \le C_0\eta^R . \]
Combining this with \eqref{eq:event_reduction_for_prop} and
\eqref{eq:good_error_final_for_prop}, we get $\prob(E_R) \le (C_0+C_3)\eta^R$.
Thus the lemma holds with $\theta:=\eta$ and $C:=C_0+C_3$.
\end{proof}

\begin{proof}[Proof of Proposition~\ref{prop:far_delete_edge_tail}]
Let $K,C_0,\theta$ be the constants from
Lemma~\ref{prop:boundary_tail}, and set $\kappa:=-\log\theta>0.$
Choose $A_0$ sufficiently large, depending only on $\delta$ and $D$, so that $A_0\ge 16K\delta$ and $\kappa A_0 \ge 16 \delta K \log D.$
Fix $A\ge A_0$. 
By \eqref{eq:symmetric_difference_equals_If}, $X_e(G)-X_e(G\setminus f)=0$ unless  $e\in \I_f(G)$.
Define
\[ Y_R := \sum_{e:\,d(e,f)\ge \delta R} w_e\bigl(X_e(G)-X_e(G\setminus f)\bigr). \]
The left-hand side of \eqref{eq:far_delete_edge_tail} is $|\E Y_R|$. Since
$0\le \W_G-\W_{G\setminus f}\le w_f$, we have
\[ |Y_R| = \Big|\W_G-\W_{G\setminus f} - \sum_{e: d(e, f) < \delta R} w_e\bigl(X_e(G)-X_e(G\setminus f)\bigr) \Big| \le w_f+\sum_{e\in \I_f(G): d(e,f)<\delta R}w_e . \]
Let $E_r:=\{ \I_f(G) \text{ reaches }\partial\bB_f^r(G)\}.$
If $Y_R \ne0$, then $E_{\lceil\delta R\rceil}$ occurs, and so
\[|Y_R| \le \ind_{E_{\lceil\delta R\rceil}} \Big( w_f+\sum_{e\in \I_f(G) : d(e,f)<\delta R}w_e \Big).\]
\begin{figure}[ht]
\centering
\scalebox{0.7}{%
\begin{tikzpicture}[
    line cap=round,
    line join=round,
    cone/.style={draw=black!55,line width=0.85pt,dash pattern=on 2.2pt off 2pt},
    shell/.style={draw=black!55,line width=0.8pt},
   shortpath/.style={draw=black!55,line width=1.45pt},
    longpath/.style={draw=black!80,line width=1.45pt},
    lab/.style={font=\small,fill=white,inner sep=1.3pt},
    pt/.style={circle,fill=black,inner sep=1.35pt}
]

\def\angL{205}
\def\angR{335}
\def\rd{1.55}
\def\rS{3.35}
\def\rA{5.00}

\coordinate (x0) at (0,0);
\coordinate (xm1) at (0,0.62);

\path[fill=blue!12!gray, opacity=0.10]
    (x0) --
    ($(x0)+(\angL:\rd)$)
    arc[start angle=\angL,end angle=\angR,radius=\rd]
    -- cycle;

\coordinate (x0) at (0,0);
\coordinate (xm1) at (0,0.62);

\draw[line width=1.0pt] (xm1) -- (x0);
\node[pt] at (xm1) {};
\node[pt] at (x0) {};
\node[font=\small, right=1pt] at ($(xm1)!0.55!(x0)$) {$f$};

\draw[cone] (x0) -- ($(x0)+(\angL:\rA)$);
\draw[cone] (x0) -- ($(x0)+(\angR:\rA)$);

\draw[shell] ($(x0)+(\angL:\rd)$)
    arc[start angle=\angL,end angle=\angR,radius=\rd];
\draw[shell] ($(x0)+(\angL:\rS)$)
    arc[start angle=\angL,end angle=\angR,radius=\rS];
\draw[shell] ($(x0)+(\angL:\rA)$)
    arc[start angle=\angL,end angle=\angR,radius=\rA];

\node[lab] at ($(x0)+(339:\rd)+(0.12,0.12)$) {$\delta R$};
\node[lab] at ($(x0)+(339:\rS)$) {$S$};
\node[lab] at ($(x0)+(339:\rA)$) {$AR$};

\draw[shortpath]
  plot[smooth,tension=0.75] coordinates {
    (x0)
    (-0.55,-0.55)
    (-0.96,-0.96)
    (-0.72,-1.34)
    (-1.34,-1.72)
    (-1.55,-2.05)
    (-1.30,-2.34)
    (-1.70,-2.58)
    (-2.10,-2.35)
  };

\draw[longpath]
  plot[smooth,tension=0.58] coordinates {
    (x0)
    (0.22,-0.55)
    (0.50,-1.12)
    (0.38,-1.48)
    (0.88,-2.02)
    (0.80,-2.44)
    (1.20,-3.02)
    (1.14,-3.48)
    (1.55,-4.22)
    (1.72,-5.28)
    (0.30,-6.32)
    (-1.15,-5.46)
    (-1.90,-4.84)
    (-1.45,-4.08)
    (-1.18,-3.55)
    (-1.28,-3.08)
    (-0.80,-2.48)
    (-0.88,-1.95)
    (-0.32,-1.15)
    (-0.06,-1.63)
    (0.35,-2.24)
    (0.22,-2.75)
    (0.70,-3.25)
    (0.62,-3.86)
    (0.96,-5.20)
    (2.18,-5.32)
    (3.65,-4.66)
    (2.55,-3.80)
    (2.10,-3.28)
    (2.26,-2.78)
    (1.55,-2.28)
    (1.38,-1.76)
    (0.65,-1.02)
    (0.95,-0.89)
  };
\end{tikzpicture}}
\caption{Two possible geometries of $\I_f(G)$ reaching $\partial\bB_f^{\delta R}(G)$. The shorter path stops before reaching $\partial\bB_f^{S}(G)$, while the longer path crosses $\partial\bB_f^{S}(G)$ and returns to $\bB_f^{\delta R}(G)$ a couple of times through cycles outside $\bB_f^{AR}(G)$.}
\label{fig:far_delete_edge_tail_cone}
\end{figure}
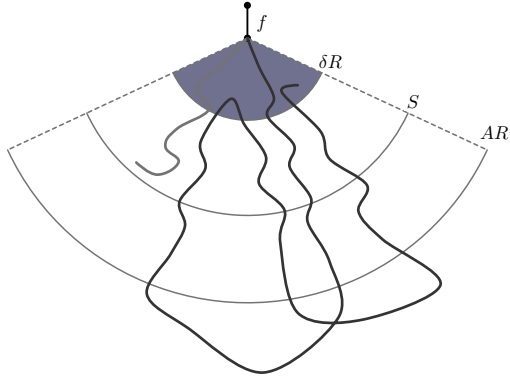

We first treat large $R$. Set $S:=\left\lfloor \tfrac{AR}{4K}\right\rfloor.$
For all large enough $R$, uniformly in $A\ge A_0$,
\[ 4 \le \lceil\delta R\rceil \le \frac{AR}{K},
\qquad S\ge4, \qquad S\ge2 \lceil\delta R\rceil,
\qquad S\ge\frac{AR}{8K}, \qquad S\le \frac{AR}{K}. \]
Since $S\le AR$ and $\bB_f^{AR}(G)$ is a tree, $\bB_f^S(G)$ is a tree.
Applying Lemma~\ref{prop:boundary_tail} with $L=AR$ gives
\begin{equation} \label{eq:boundary_tail_used_far}
\prob(E_r) \le C_0\theta^r = C_0e^{-\kappa r}, \qquad r \in \{\lceil\delta R\rceil,S \}.
\end{equation}
Let $Z_R:=\max_{g\in E(\bB_f^{\lceil\delta R\rceil}(G))} w_g$.
On $E_{\lceil\delta R\rceil}\setminus E_S$, the component $\I_f(G)$ is contained in the tree
$\bB_f^S(G)$, hence is a path. Therefore
\[ w_f+\sum_{\substack{e\in \I_f(G):\\ d(e,f)<\delta R }}w_e \le
(2\lceil\delta R\rceil +2)Z_R. \]
On $E_S$, we use the crude bound
\[ w_f+\sum_{\substack{e\in \I_f(G):\\ d(e,f)<\delta R}}w_e \le 
CD^{\lceil\delta R\rceil} Z_R . \]
Thus
\begin{equation*}
|Y_R| \le (2\lceil\delta R\rceil+2)Z_R\ind_{E_{\lceil\delta R\rceil}}
+ CD^{\lceil\delta R\rceil}Z_R\ind_{E_S}.
\end{equation*}
Since $|E(\bB_f^{\lceil\delta R\rceil}(G))|\le CD^{\lceil\delta R\rceil}$, by Lemma~\ref{lem: exp}, $\E Z_R^2\le C(1+\lceil\delta R\rceil)^2$. Together with Cauchy-Schwarz and \eqref{eq:boundary_tail_used_far}, we obtain
\[ \E\big[(2\lceil\delta R\rceil+2)Z_R\ind_{E_{\lceil\delta R\rceil}}\big]
\le CR^2e^{-\kappa\lceil\delta R\rceil/2} \le Ce^{-c_1R}, \]
where $c_1>0$ depends only on $\delta$ and $D$. Similarly,
\[ \begin{aligned} \E\big[CD^{\lceil\delta R\rceil} Z_R \ind_{E_S}\big]
&\le CR\exp\Big(\lceil\delta R\rceil \log D-\frac{\kappa S}{2}\Big )  \\
&\le CR\exp\Big( (\delta R+1)\log D-\frac{\kappa A}{16K}R\Big)
\le Ce^{-c_2R}, 
\end{aligned}\]
for some $c_2>0$, by the choice of $A_0$. Hence, for all large $R$,
\[ |\E Y_R| \le \E|Y_R| \le Ce^{-cR}, \qquad c:=\min\{c_1,c_2\}. \]

For the remaining bounded range of $R$,
\[ |Y_R| \le w_f+\sum_{e:\,d(e,f)<\delta R }w_e \]
implies $\E|Y_R|\le C$, after enlarging $C$. This proves
Proposition~\ref{prop:far_delete_edge_tail}.
\end{proof}

\medskip\noindent\textbf{Declaration on the use of AI.}
The TikZ figures were generated with the assistance of ChatGPT (OpenAI). ChatGPT was also used in part for grammar and language editing and for formatting references.

\bibliographystyle{plain}
\bibliography{matchrefnew}

@misc{lamsen2026,
  author        = {Lam, Wai-Kit and Sen, Arnab},
  title         = {Correlation decay for maximum weight matchings on sparse graphs},
  year          = {2026},
  note          = {arXiv:2511.18861},
  eprint        = {2511.18861},
  archivePrefix = {arXiv},
  primaryClass  = {math.PR},
  url           = {https://arxiv.org/abs/2511.18861},
}

@incollection{AldousSteele,
  author        = {Aldous, David and Steele, J. Michael},
  title         = {The objective method: probabilistic combinatorial optimization and local weak
                   convergence},
  editor        = {Kesten, Harry},
  booktitle     = {Probability on Discrete Structures},
  series        = {Encyclopaedia of Mathematical Sciences},
  volume        = {110},
  pages         = {1--72},
  publisher     = {Springer},
  address       = {Berlin, Heidelberg},
  year          = {2004},
}

@article{Gamarnik,
  author        = {Gamarnik, David and Nowicki, Tomasz and Swirszcz, Grzegorz},
  title         = {Maximum weight independent sets and matchings in sparse random graphs. {E}xact
                   results using the local weak convergence method},
  journal       = {Random Structures Algorithms},
  volume        = {28},
  number        = {1},
  pages         = {76--106},
  year          = {2006},
}

@article{enriquez2025optimal,
  author        = {Enriquez, Nathana{\"e}l and Liu, Mike and M{\'e}nard, Laurent and Perchet,
                   Vianney},
  title         = {Optimal unimodular matchings},
  journal       = {Probab. Theory Related Fields},
  volume        = {194},
  number        = {3--4},
  pages         = {1849--1916},
  year          = {2026},
}

@article{Cao,
  author        = {Cao, Sky},
  title         = {Central limit theorems for combinatorial optimization problems on sparse
                   {Erd{\H{o}}s--R{\'{e}}nyi} graphs},
  journal       = {Ann. Appl. Probab.},
  volume        = {31},
  number        = {4},
  pages         = {1687--1723},
  year          = {2021},
}

@misc{SturmWemheuer,
  author        = {Sturm, Anja and Wemheuer, Moritz},
  title         = {A central limit theorem for functions on weighted sparse inhomogeneous random
                   graphs},
  year          = {2024},
  note          = {arXiv:2404.12740},
  eprint        = {2404.12740},
  archivePrefix = {arXiv},
  primaryClass  = {math.PR},
  url           = {https://arxiv.org/abs/2404.12740},
}

@article{KrishnanRay,
  author        = {Krishnan, Kesav and Ray, Gourab},
  title         = {Uniqueness and {CLT} for the ground state of the disordered monomer-dimer model
                   on {$\mathbb{Z}^d$}},
  journal       = {Int. Math. Res. Not. IMRN},
  volume        = {2025},
  number        = {12},
  pages         = {rnaf169},
  year          = {2025},
}

@inproceedings{KarpSipser,
  author        = {Karp, Richard M. and Sipser, Michael},
  title         = {Maximum matching in sparse random graphs},
  booktitle     = {22nd Annual Symposium on Foundations of Computer Science},
  pages         = {364--375},
  publisher     = {IEEE Computer Society},
  year          = {1981},
}

@article{AronsonFriezePittel,
  author        = {Aronson, Jonathan and Frieze, Alan and Pittel, Boris G.},
  title         = {Maximum matchings in sparse random graphs: {Karp--Sipser} revisited},
  journal       = {Random Structures Algorithms},
  volume        = {12},
  number        = {2},
  pages         = {111--177},
  year          = {1998},
}

@article{ElekLippner,
  author        = {Elek, G{\'a}bor and Lippner, G{\'a}bor},
  title         = {{Borel} oracles. {A}n analytical approach to constant-time algorithms},
  journal       = {Proc. Amer. Math. Soc.},
  volume        = {138},
  number        = {8},
  pages         = {2939--2947},
  year          = {2010},
}

@article{BordenaveLelargeSalez,
  author        = {Bordenave, Charles and Lelarge, Marc and Salez, Justin},
  title         = {Matchings on infinite graphs},
  journal       = {Probab. Theory Related Fields},
  volume        = {157},
  number        = {1--2},
  pages         = {183--208},
  year          = {2013},
}

@phdthesis{Kreacic2017,
  author        = {Krea{\v{c}}i{\'c}, Eleonora},
  title         = {Some problems related to the {Karp--Sipser} algorithm on random graphs},
  school        = {University of Oxford},
  year          = {2017},
  url           = {https://ora.ox.ac.uk/objects/uuid:3b2eb52a-98f5-4af8-9614-e4909b8b9ffa},
}

@article{GlasgowKwanSahSawhney,
  author        = {Glasgow, Margalit and Kwan, Matthew and Sah, Ashwin and Sawhney, Mehtaab},
  title         = {A central limit theorem for the matching number of a sparse random graph},
  journal       = {J. Lond. Math. Soc.},
  volume        = {111},
  number        = {4},
  pages         = {e70101},
  year          = {2025},
}

@article{AthreyaYogeshwaran,
  author        = {Athreya, Siva and Yogeshwaran, D.},
  title         = {Central limit theorem for statistics of subcritical configuration models},
  journal       = {J. Ramanujan Math. Soc.},
  volume        = {35},
  number        = {2},
  pages         = {109--119},
  year          = {2020},
  eprint        = {1808.06778},
  archivePrefix = {arXiv},
  url           = {https://www.mathjournals.org/jrms/2020-035-002/2020-035-002-001.html},
}

@misc{enriquez2026optimal,
  author        = {Enriquez, Nathana{\"e}l and Liu, Mike and M{\'{e}}nard, Laurent and Perchet,
                   Vianney},
  title         = {Optimal matching under size priority},
  year          = {2026},
  note          = {arXiv:2601.20502},
  eprint        = {2601.20502},
  archivePrefix = {arXiv},
  primaryClass  = {math.PR},
  url           = {https://arxiv.org/abs/2601.20502},
}

@article{DeyKrishnan,
  author        = {Dey, Partha S. and Krishnan, Kesav},
  title         = {Disordered monomer-dimer model on cylinder graphs},
  journal       = {J. Stat. Phys.},
  volume        = {190},
  number        = {8},
  pages         = {146},
  year          = {2023},
}

@article{lamsenpositivetemp,
  author        = {Lam, Wai-Kit and Sen, Arnab},
  title         = {Central limit theorem in disordered monomer-dimer model},
  journal       = {Random Structures Algorithms},
  volume        = {66},
  number        = {1},
  pages         = {e21256},
  year          = {2025},
}

@article{BayatiGamarnikTetali,
  author        = {Bayati, Mohsen and Gamarnik, David and Tetali, Prasad},
  title         = {Combinatorial approach to the interpolation method and scaling limits in sparse
                   random graphs},
  journal       = {Ann. Probab.},
  volume        = {41},
  number        = {6},
  pages         = {4080--4115},
  year          = {2013},
}

@article{Salez,
  author        = {Salez, Justin},
  title         = {The interpolation method for random graphs with prescribed degrees},
  journal       = {Combin. Probab. Comput.},
  volume        = {25},
  number        = {3},
  pages         = {436--447},
  year          = {2016},
}

@article{Huang,
  author        = {Huang, Brice},
  title         = {Convergence of maximum bisection ratio of sparse random graphs},
  journal       = {Electron. Commun. Probab.},
  volume        = {23},
  pages         = {1--10},
  year          = {2018},
  note          = {Paper No.~51},
}

@article{BarbourRollin,
  author        = {Barbour, A. D. and R\"{o}llin, Adrian},
  title         = {Central limit theorems in the configuration model},
  journal       = {Ann. Appl. Probab.},
  volume        = {29},
  number        = {2},
  pages         = {1046--1069},
  year          = {2019},
}

@article{Pittel1990,
  author        = {Pittel, Boris},
  title         = {On tree census and the giant component in sparse random graphs},
  journal       = {Random Structures Algorithms},
  volume        = {1},
  number        = {3},
  pages         = {311--342},
  year          = {1990},
}

@misc{KangLiu2026,
  author        = {Kang, Mihyun and Liu, Mike},
  title         = {Uniqueness and locality of the ground state of the disordered monomer-dimer
                   models on independently weighted unimodular {Bienaym{\'{e}}--Galton--Watson}
                   trees},
  year          = {2026},
  note          = {arXiv:2603.19003},
  eprint        = {2603.19003},
  archivePrefix = {arXiv},
  primaryClass  = {math.PR},
  url           = {https://arxiv.org/abs/2603.19003},
}

@inproceedings{Chatterjee14,
  author        = {Chatterjee, Sourav},
  title         = {A short survey of {Stein}'s method},
  booktitle     = {Proceedings of the {I}nternational {C}ongress of {M}athematicians---{S}eoul
                   2014. {V}ol. {IV}},
  pages         = {1--24},
  publisher     = {Kyung Moon Sa},
  address       = {Seoul},
  year          = {2014},
}

@article{Janson,
  author        = {Janson, Svante},
  title         = {Random graphs with given vertex degrees and switchings},
  journal       = {Random Structures Algorithms},
  volume        = {57},
  number        = {1},
  pages         = {3--31},
  year          = {2020},
}

@article{Bollobas,
  author        = {Bollob\'{a}s, B\'{e}la},
  title         = {A probabilistic proof of an asymptotic formula for the number of labelled
                   regular graphs},
  journal       = {European J. Combin.},
  volume        = {1},
  number        = {4},
  pages         = {311--316},
  year          = {1980},
}

@article{chengoldsteinrollin,
  author        = {Chen, Louis H. Y. and Goldstein, Larry and R\"{o}llin, Adrian},
  title         = {{Stein}'s method via induction},
  journal       = {Electron. J. Probab.},
  volume        = {25},
  pages         = {1--49},
  year          = {2020},
  note          = {Paper No.~132},
}

@article{DS,
  author        = {Diaconis, Persi and Shahshahani, Mehrdad},
  title         = {Generating a random permutation with random transpositions},
  journal       = {Z. Wahrsch. Verw. Gebiete},
  volume        = {57},
  number        = {2},
  pages         = {159--179},
  year          = {1981},
}

@article{BobkovTetali,
  author        = {Bobkov, Sergey G. and Tetali, Prasad},
  title         = {Modified logarithmic {Sobolev} inequalities in discrete settings},
  journal       = {J. Theoret. Probab.},
  volume        = {19},
  number        = {2},
  pages         = {289--336},
  year          = {2006},
}

@article{Chatterjee19,
  author        = {Chatterjee, Sourav},
  title         = {A general method for lower bounds on fluctuations of random variables},
  journal       = {Ann. Probab.},
  volume        = {47},
  number        = {4},
  pages         = {2140--2171},
  year          = {2019},
}

@book{vdHofstadvol1,
  author        = {van der Hofstad, Remco},
  title         = {Random graphs and complex networks. {V}ol. 1},
  series        = {Cambridge Series in Statistical and Probabilistic Mathematics},
  volume        = {43},
  publisher     = {Cambridge University Press},
  address       = {Cambridge},
  year          = {2017},
}

@book{BollobasBook,
  author        = {Bollob\'as, B\'ela},
  title         = {Random graphs},
  series        = {Cambridge Studies in Advanced Mathematics},
  volume        = {73},
  edition       = {Second},
  publisher     = {Cambridge University Press},
  address       = {Cambridge},
  year          = {2001},
}

@book{vdHofstad,
  author        = {van der Hofstad, Remco},
  title         = {Random graphs and complex networks. {V}ol. 2},
  series        = {Cambridge Series in Statistical and Probabilistic Mathematics},
  volume        = {54},
  publisher     = {Cambridge University Press},
  address       = {Cambridge},
  year          = {2024},
}

@article{Sjostrand,
  author        = {Sj\"{o}strand, Jonas},
  title         = {Making multigraphs simple by a sequence of double edge swaps},
  journal       = {Discrete Math.},
  volume        = {344},
  number        = {5},
  pages         = {112328},
  year          = {2021},
}

@book{Gut,
  author        = {Gut, Allan},
  title         = {Probability: a graduate course},
  series        = {Springer Texts in Statistics},
  edition       = {Second},
  publisher     = {Springer},
  address       = {New York},
  year          = {2013},
}

@misc{Mordant2026,
  author        = {Mordant, Gilles},
  title         = {A Central Limit Theorem for the Random Assignment Problem},
  year          = {2026},
  eprint        = {2608.05123},
  note          = {arXiv:2608.05123},
  archiveprefix = {arXiv},
  primaryclass  = {math.PR}
}

@article{MezardParisi1985,
  author  = {M{\'e}zard, Marc and Parisi, Giorgio},
  title   = {Replicas and Optimization},
  journal = {Journal de Physique Lettres},
  volume  = {46},
  number  = {17},
  pages   = {L771--L778},
  year    = {1985}
}

@article{MezardParisi1986,
  author  = {M{\'e}zard, Marc and Parisi, Giorgio},
  title   = {Mean-Field Equations for the Matching and the Travelling Salesman Problems},
  journal = {Europhysics Letters},
  volume  = {2},
  number  = {12},
  pages   = {913--918},
  year    = {1986}
}

@article{Wastlund2012,
  author  = {W{\"a}stlund, Johan},
  title   = {Replica Symmetry of the Minimum Matching},
  journal = {Annals of Mathematics},
  volume  = {175},
  number  = {3},
  pages   = {1061--1091},
  year    = {2012}
}

@article{Aldous2001,
  author  = {Aldous, David J.},
  title   = {The {$\zeta(2)$} Limit in the Random Assignment Problem},
  journal = {Random Structures \& Algorithms},
  volume  = {18},
  number  = {4},
  pages   = {381--418},
  year    = {2001},
  doi     = {10.1002/rsa.1015}
}
\end{document}